\documentclass[11pt,leqno]{article}
\usepackage{a4wide}
\usepackage{amssymb,upgreek}
\usepackage{amsmath}
\usepackage{amsthm}
\usepackage{mathrsfs}
\usepackage{euscript,csquotes}
\usepackage{xypic}
\usepackage[usenames,dvipsnames]{color}
\usepackage{endnotes,comment}

\usepackage{hyperref}
\hypersetup{colorlinks,  citecolor=ForestGreen,  linkcolor=Mahogany, urlcolor=black}

\theoremstyle{plain}

\newcommand{\id}{\operatorname{id}}
\newcommand{\im}{\operatorname{im}}
\newcommand{\gr}{\operatorname{gr}}
\newcommand{\pr}{\operatorname{pr}}
\newcommand{\Hom}{\operatorname{Hom}}
\newcommand{\End}{\operatorname{End}}
\newcommand{\Ext}{\operatorname{Ext}}
\newcommand{\Fil}{\operatorname{Fil}}
\newcommand{\Mod}{\operatorname{Mod}}

\newcommand{\ind}{\operatorname{ind}}
\newcommand{\Jac}{\operatorname{Jac}}
\newcommand{\Tor}{\operatorname{Tor}}
\newcommand{\supp}{\operatorname{supp}}
\newcommand{\coker}{\operatorname{coker}}

\newcommand{\cX}{\underline{\underline{\mathbf{X}}}}

\def\D{\EuScript{D}}
\def\P{\EuScript{P}}

\def\fhom{\mathfrak h}
\def\ften{\mathfrak t}
\def\X{{\mathbf X}}
\def\sur{{$\mathbf{(sur)\,}$}}

\newtheorem{theorem}{Theorem}[section]
\newtheorem{corollary}[theorem]{Corollary}
\newtheorem{lemma}[theorem]{Lemma}
\newtheorem{proposition}[theorem]{Proposition}

\theoremstyle{definition}
\newtheorem{remark}[theorem]{Remark}
\title{\textbf{A canonical torsion theory for pro-$p$ Iwahori-Hecke modules}}
\author{Rachel Ollivier, Peter Schneider}
\date{}

\newcommand\blfootnote[1]{%
  \begingroup
  \renewcommand\thefootnote{}\footnote{#1}%
  \addtocounter{footnote}{-1}%
  \endgroup
}

\begin{document}

\maketitle

\begin{abstract}
Let $\mathfrak F$ be a locally compact nonarchimedean field with residue characteristic $p$ and $G$ the group of $\mathfrak{F}$-rational points of a connected split reductive group over $\mathfrak{F}$. We define a torsion pair in the category $\Mod(H)$ of modules over the pro-$p$-Iwahori Hecke $k$-algebra $H$ of $G$, where $k$ is an arbitrary field. We prove that, under a certain hypothesis, the torsionfree class embeds fully faithfully into the category $\Mod^I(G)$ of smooth $k$-representations of $G$ generated by their pro-$p$-Iwahori fixed vectors.

If the characteristic of $k$ is different from $p$ then this hypothesis is always satisfied and the torsionfree class is the whole category $\Mod(H)$.

If $k$ contains the residue field of $\mathfrak F$ then we study the case $G = \mathbf{SL_2}(\mathfrak F)$. We show that our hypothesis is satisfied, and we describe explicitly the torsionfree and the torsion classes. If $\mathfrak F\neq \mathbb Q_p$ and $p\neq 2$, then an $H$-module is in the torsion class if and only if it is a union of supersingular finite length submodules; it lies in the torsionfree class if and only if it does not contain any nonzero supersingular finite length module. If $\mathfrak{F} = \mathbb{Q}_p$, the torsionfree class is the whole category $\Mod(H)$, and we give a  new proof of the fact that $\Mod(H)$ is equivalent to $\Mod^I(G)$. These results are based on the computation of the $H$-module structure of certain natural cohomology spaces for the pro-$p$-Iwahori subgroup $I$ of $G$.
\end{abstract}

\addcontentsline{toc}{section}{Introduction}
\tableofcontents

\section*{Introduction}

\subsection{Motivation and context}

Motivated by the mod $p$ and $p$-adic local Langlands program, the smooth mod $p$ representation theory of $p$-adic reductive groups started taking shape around the year 2000. The case of $\mathbf{GL_2}(\mathbb Q_p)$ is now well understood. In particular, its irreducible smooth mod $p$ representations are classified (including the most subtle ones called supersingular) and the mod $p$/$p$-adic Langlands local correspondence is given by Colmez' functor (\cite{Col}, \cite{Pas}).

Beyond $\mathbf{GL_2}(\mathbb Q_p)$ and $\mathbf{SL_2}(\mathbb Q_p)$ (\cite{Koz}, \cite{Abd}) very little is known about  the supersingular representations of an arbitrary $p$-adic connected reductive group and more generally about the category of its smooth mod $p$ representations.

Let $\mathfrak F$ be a locally compact nonarchimedean field with residue characteristic $p$ and residue field  $\mathbb F_q$ with $q$ elements,   and  let    $G$ be the group of $\mathfrak{F}$-rational points of a connected reductive group  $\mathbf G$ over $\mathfrak{F}$.  For simplicity, we suppose that $\mathbf G$ is $\mathfrak F$-split in this article. Let $k$ be a field of characteristic $p$ and let $\Mod(G)$ denote the category of all smooth $k$-representations of $G$.
In order to understand $\Mod(G)$ one can first study the category $\Mod(H)$ of all modules over the Hecke $k$-algebra $H$ of a fixed pro-$p$ Iwahori subgroup $I$ of $G$. This is helpful since there is a  natural left exact  functor
\begin{align*}
 \fhom: \Mod(G) & \longrightarrow \Mod(H) \\
              V & \longmapsto V^I = \Hom_{k[G]}(\mathbf{X},V) \ ,
\end{align*}
sending a nonzero representation onto a nonzero module. Its left adjoint is
\begin{align*}
\ften_0:\Mod(H) & \longrightarrow  \Mod^I(G)\subseteq \Mod(G) \\
                M & \longmapsto \mathbf{X} \otimes_H M \ .
\end{align*} Here $\mathbf X$ denotes the space of $k$-valued functions with compact support on $G/I$ with the natural left action of $G$.
The functor $\ften_0$ has values in the category  $\Mod^I(G)$ of all smooth $k$-representations of $G$ generated by their $I$-fixed vectors. If $k$ is algebraically closed and  $G = \mathbf{GL_2}(\mathbb Q_p)$ (\cite{Oll1}) or $G = \mathbf{SL_2}(\mathbb Q_p)$ (\cite{Koz}) the functors $\fhom$ and $\ften_0$ are quasi-inverse  to each other,  yielding an equivalence  between $\Mod(H)$ and $\Mod^I(G)$. However, more generally these functors are not well behaved and this approach needs to be refined. This is the main goal of this article.

Here are some general remarks which motivate our results.

{\textbf A)} There is a derived version of the functors $\fhom$ and $\ften_0$ providing  an equivalence between the derived category of smooth representations of $G$ in $k$-vector spaces and the derived category of differential graded modules over a certain Hecke differential graded pro-$p$ Iwahori-Hecke algebra $H^\bullet$.
This theorem  was proved in 2004  in \cite{SDGA} (assuming the extra hypotheses that $\mathfrak F$ is an extension of $\mathbb Q_p$ and $I$ is torsionfree.)

While this theorem is fundamental in our context, it has barely been exploited so far because of a lack of concrete understanding of  the derived objects in question.  For example, the cohomology algebra of $H^\bullet$ is the algebra $\Ext_{\Mod(G)}^*(\mathbf X, \mathbf X)$.
In this article we  give the first explicit calculations in the lowest graded pieces of the algebra $\Ext_{\Mod(G)}^*(\mathbf X, \mathbf X)$ in the case of $G = \mathbf{SL_2}(\mathfrak F)$, for arbitrary $\mathfrak F$ (Prop.\ \ref{propintro}, Thm.\ \ref{maintheosocle}, \S\ref{sec:fullspaces}).

{\textbf B)} Recall that a finite length $H$-module is supersingular if it is annihilated by a positive power of a certain ideal of the center of $H$ (\S\ref{sec:charH}). The difficulty in understanding  the functors $\fhom$ and  $\ften_0$ is related to  their behavior with respect to these supersingular modules. For example, if $k$ is algebraically closed and $G = \mathbf{GL_2}(\mathfrak F)$,  the functor $\fhom$  matches up the simple  nonsupersingular modules of $H$ with the irreducible subquotients of principal series representations of $G$ (\cite{Viggl2}). However, when $\mathfrak F\neq \mathbb Q_p$, there exist irreducible supersingular representations whose image under $\fhom$ are not  simple $H$-modules (\cite{BP});
moreover, the image of a simple supersingular module under the composition $\fhom\circ \ften_0$ is an $H$-module with infinite length (at least if $k$ is algebraically closed). The latter  was obtained in \cite{Oll1} by direct computation. However it would be more meaningful to explain the behavior of $\fhom$ and  $\ften_0$ depending on the choice of $\mathfrak F$ in terms of the derived equivalence discussed above in \textbf{A)}. To do this we would like to connect the notion of supersingularity to certain higher cohomology spaces of $I$. We achieve this when $G = \mathbf{SL_2}(\mathfrak F)$ (see Prop.s \ref{propintro}, \ref{propintro2}, Thm.\ \ref{theointro}).

{\textbf C)} Although the definition of the supersingular $H$-modules is largely combinatorial, they seem to carry some number theoretic information: in the case when $G$ is the general linear group, there is a  numerical coincidence between the set of  irreducible mod $p$ Galois representations (with fixed determinant of a Frobenius)  and a certain family of irreducible supersingular Hecke modules (\cite{Vig}, \cite{Oll2}).  In fact, a natural bijection between these two sets  is now induced by the functor constructed in \cite{GK}. Current developments in the mod $p$/$p$-adic Langlands program as mentioned in \cite{Har} aim towards a Langlands correspondence given by (derived) functors between suitable categories. Our observations suggest that the (derived) category of Hecke modules is likely to play a role in the picture.

Both authors acknowledge support by the Institute for Mathematical Sciences, National University of Singapore during a visit in 2013. The first author is partially supported by NSERC Discovery Grant.

\subsection{Summary of Results}

In Section \ref{sec:1} we allow $k$ to be a field with arbitrary characteristic. In \S\ref{sec:defiTP} we introduce a family $\mathcal Z=(Z_m)_{m\geq 1}$ of right $H$-modules which generates a torsion pair
$(\mathcal  F_G, \mathcal T_G)$ in $\Mod(H)$. In \S\ref{setting} we define a modified version  $\ften$ of the functor $\ften_0$ and prove the following:

\begin{theorem}[Cor.\ \ref{fully-faithful}]
Under the hypothesis \sur, the functor $\ften$ is a fully faithful embedding of the torsionfree class $\mathcal F_G$ into the category $\Mod ^I(G)$.\end{theorem}

The hypothesis  \sur  says that the natural  image of $H$ in $\mathbf X$ is a direct summand as a $H$-module (Remark \ref{sur}).

Note that  if $k$ has characteristic different from $p$, then \sur is satisfied, $Z_m = 0$ for all $m \geq 1$ and  $\mathcal F_G$ is the whole category $\Mod(H)$ (Lemma \ref{not-p-zero}).  So from now on we suppose that $k$ has characteristic $p$.

In Section \ref{sec:2},  we assume that $\mathbf G$ has semisimple rank $1$. In this case we show that \sur is satisfied (\S\ref{hyposur}). The proof relies on the description  of the graded pieces associated to the natural filtration of  the $H$-module $\mathbf X$ by the successive submodules  $\mathbf X^{K_m}$, where $K_m$ denotes the $m^{\rm th}$ congruence subgroup of $G$ (Thm.\ \ref{level-acyclic}).
Using the fact that $H$ is Gorenstein (\cite{OS}) we then show that any reflexive $H$-module belongs to  the torsionfree class $\mathcal F_G$. This gives the first nontrivial examples of modules contained in $\mathcal F_G$ in the case when  $k$ has characteristic $p$. Moreover, if $\mathbf G$ is semisimple, that is to say $\mathbf G = \mathbf{SL_2}$ or $\mathbf{G} = \mathbf{PGL_2}$, then $H$ is Auslander-Gorenstein with self-injective dimension $1$ and we obtain  the following expression for $Z_m$ (Cor.\ \ref{Zm-dual})
\begin{equation}\label{formulaZm}
  Z_m = \Ext^1_H(H^1(I,\mathbf{X}^{K_m}),H) = \Ext^1_H(H^1(I,\mathbf{X}^{K_m})^1,H)
\end{equation} where $H^1(I,\mathbf{X}^{K_m})^1$ is the largest finite dimensional sub-$H$-module of
$H^1(I,\mathbf{X}^{K_m})$.

In Section \ref{sec:3} we take $\mathbf{G} = \mathbf{SL_2}$ and compute $Z_m$ while describing the torsion class $\mathcal T_G$. First we show the following:

\begin{proposition}[Cor.\ \ref{supersingular}]\label{propintro}
For any $m\geq 1$ the finite dimensional $H$-modules  $H^1(I,\mathbf{X}^{K_m})^1$  and $Z_m$ are    supersingular  (possibly trivial).
\end{proposition}

We recall in \S\ref{sec:charH} that a  finite length $H$-module is supersingular if and only if it is $\zeta$-torsion where $\zeta$ is the central element of $H$ defined in \S\ref{sec:zeta}. It is called nonsupersingular if it is $\zeta$-torsionfree.

The precise decomposition of the socle of $H^1(I,\mathbf{X}^{K_m})^1$ is given in Thm.\ \ref{maintheosocle}.
The case $\mathfrak F=\mathbb Q_p$ stands out since it is the only case where all  $H^1(I,\mathbf{X}^{K_m})^1$ and $Z_m$ are trivial. As a corollary, we obtain:

\begin{proposition}[Prop.\ \ref{Qp-equivalence}]\label{propintro2}
If $\mathfrak{F} = \mathbb{Q}_p$ then $\mathcal{F}_G = \Mod(H)$ and the functors  $\fhom$ and
$\ften_0$ are quasi-inverse equivalences of categories between
$\Mod(H)$ and   $\Mod^I(G)$.
\end{proposition}

This result provides a different proof of one of the main results in \cite{Koz} (where $k$ is assumed algebraically closed and $p\neq 2$). When $\mathfrak F\neq \mathbb Q_p$, $p\neq 2$ every simple supersingular module appears with multiplicity $\geq 1$ in the  socle of   $H^1(I,\mathbf{X}^{K_m})^1$ for some $m\geq 2$. We deduce:

\begin{theorem}[Thm.s \ref{FnotQp} and \ref{theo:torfree}]\label{theointro}
Suppose that  $p \neq 2$, $\mathbb{F}_q \subseteq k$, and $\mathfrak{F} \neq \mathbb{Q}_p$.  Then an $H$-module $M$ lies in $\mathcal{T}_G$ if and only if  $M$ is a union of supersingular finite length submodules, or equivalently if it is $\zeta$-torsion.  It lies in $\mathcal{F}_G$ if and only if it does not contain any nonzero supersingular finite length module, or equivalently  if it is $\zeta$-torsionfree.
\end{theorem}

We denote by $H_\zeta$ the localization in $\zeta$ of $H$. The category $\Mod(H_\zeta)$ identifies with the subcategory of all modules in $\Mod(H)$
on which $\zeta$ acts bijectively. We prove the following theorem which, in particular, explains why the functor  $\ften_0$ behaves well on nonsupersingular modules (compare with \textbf {B}) above).

\begin{theorem}[Thm.\ \ref{theo:localcoincide}]\label{theointroloc}
 Suppose that $\mathbb F_q\subseteq k$ and that either $p \neq 2$ or $\mathfrak{F} \neq \mathbb{Q}_p$. The functor $\ften$ coincides with $\ften_0$  on  $\Mod(H_\zeta)$. It induces an exact fully faithful functor $\Mod(H_\zeta)\rightarrow \Mod(G)$ and we have
$\fhom\circ \ften_0\vert_{\Mod(H_\zeta)}=\fhom\circ \ften\vert_{\Mod(H_\zeta)}={\rm id} _{\Mod(H_\zeta)}.$
\end{theorem}

Lastly in \S\ref{sec:fullspaces}, we   are interested in the $H$-modules  $H^j(I, \X^{K_m})$ for $j,m\geq 1$. In particular, we compute the full  $H$-module structure when $m=1$ and $j=1$.   More in-depth  calculations in the cohomology spaces $H^j(I, \X^{K_m})$ and $H^j(I, \X)$, for $j\geq 1$ and $m\geq 1$, will follow in subsequent work of the authors.

\section{\label{sec:1}The formalism}

\subsection{A brief reminder of torsion pairs}

Let $R$ be any ring and denote by $\Mod(R)$ the category of left $R$-modules. We fix a class $\mathcal{Z}$ of modules in $\Mod(R)$ and we introduce the following subcategories of $\Mod(R)$.
\begin{itemize}
  \item[--] $\mathcal{F}$ is the full subcategory of all modules $M$ such that $\Hom_R(Z,M) = 0$ for any $Z \in \mathcal{Z}$.
  \item[--] $\mathcal{T}$ is the full subcategory of all modules $N$ such that $\Hom_R(N,M) = 0$ for any $M$ in $\mathcal{F}$.
\end{itemize}
The pair $(\mathcal{F}, \mathcal{T})$ is called the torsion pair in $\Mod(R)$ generated by $\mathcal{Z}$ with $\mathcal{F}$ being the torsionfree class and $\mathcal{T}$ being the torsion class. One easily checks the following facts.
\begin{enumerate}
  \item $\mathcal{F}$ is closed under the formation of submodules, extensions, and arbitrary direct products.
  \item $\mathcal{T}$ is closed under the formation of factor modules, extensions, and arbitrary direct sums. Of course, $\mathcal{Z}$ is contained in $\mathcal{T}$.
  \item Every module $M$ in $\Mod(R)$ has a unique largest submodule $t(M)$ which is  contained in $\mathcal{T}$. Moreover, the factor module $M/t(M)$ lies in $\mathcal{F}$.
  \item A simple $R$-module lies in $\mathcal{T}$ if and only if it is isomorphic to a quotient of some module in $\mathcal{Z}$ (cf.\ \cite{Ste} VI Prop.\ 2.5).
\end{enumerate}

\subsection{The setting\label{setting}}

We fix a locally compact nonarchimedean field $\mathfrak{F}$ (of any characteristic) with ring of integers $\mathfrak{O}$, the maximal ideal $\mathfrak{M}$,  and a prime element $\pi$, and we let $G := \mathbf{G}(\mathfrak{F})$ be the group of $\mathfrak{F}$-rational points of a connected reductive group $\mathbf{G}$ over $\mathfrak{F}$ which is $\mathfrak{F}$-split. The residue field $\mathfrak{O}/\pi \mathfrak{O}$ of $\mathfrak{F}$ is $\mathbb{F}_q$ for some power $q = p^f$ of the residue characteristic $p$.

We fix a chamber $C$ in the (semisimple) building $\mathscr{X}$ of $G$ as well as a hyperspecial vertex $x_0$ of $C$. The stabilizer of $x_0$ in $G$ contains a good maximal compact subgroup $K$ of $G$. The pointwise stabilizer $J \subseteq K$ of $C$ is an Iwahori subgroup. In fact, let $\mathbf{G}_{x_0}$ and $\mathbf{G}_C$ denote the Bruhat-Tits group schemes over $\mathfrak{O}$ whose $\mathfrak{O}$-valued points are $K$ and $J$, respectively. Their reductions over the residue field of $\mathbb{F}_q$ are denoted by $\overline{\mathbf{G}}_{x_0}$ and $\overline{\mathbf{G}}_C$. By \cite{Tit} 3.4.2, 3.7, and 3.8 we have:
\begin{itemize}
  \item[--] $\overline{\mathbf{G}}_{x_0}$ is connected reductive and $\mathbb{F}_q$-split.
  \item[--] The natural homomorphism $\overline{\mathbf{G}}_C \longrightarrow \overline{\mathbf{G}}_{x_0}$ has a connected unipotent kernel and maps $\overline{\mathbf{G}}_C^\circ$ onto a Borel subgroup $\overline{\mathbf{B}}$ of $\overline{\mathbf{G}}_{x_0}$. Hence $\overline{\mathbf{G}}_C / \overline{\mathbf{G}}_C^\circ$ is isomorphic to a subgroup of $\mathrm{Norm}_{\overline{\mathbf{G}}_{x_0}}(\overline{\mathbf{B}})
      /\overline{\mathbf{B}}$. But $\mathrm{Norm}_{\overline{\mathbf{G}}_{x_0}}(\overline{\mathbf{B}}) = \overline{\mathbf{B}}$ and therefore $\overline{\mathbf{G}}_C = \overline{\mathbf{G}}_C^\circ$.
\end{itemize}
Let $\overline{\mathbf{N}}$ denote the unipotent radical of $\overline{\mathbf{B}}$. We put
\begin{equation*}
    K_1 := \ker \big(\mathbf{G}_{x_0}(\mathfrak{O}) \xrightarrow{\; \pr \;} \overline{\mathbf{G}}_{x_0} (\mathbb{F}_q) \big) \quad\textrm{and}\quad I := \{g \in K : \pr(g) \in \overline{\mathbf{N}}(\mathbb{F}_q) \}
\end{equation*}
and obtain the chain
\begin{equation*}
    K_1 \subseteq I \subseteq J \subseteq K
\end{equation*}
of compact open subgroups in $G$ such that
\begin{equation*}
    K/K_1 = \overline{G} := \overline{\mathbf{G}}_{x_0} (\mathbb{F}_q) \supseteq I/K_1 = \overline{I} := \overline{\mathbf{N}}(\mathbb{F}_q) \ .
\end{equation*}
The subgroup $I$ is pro-$p$ and is called the pro-$p$-Iwahori subgroup. We also will need the fundamental system of open normal subgroups
\begin{equation*}
    K_m := \ker \big(\mathbf{G}_{x_0}(\mathfrak{O}) \xrightarrow{\; \pr \;} \mathbf{G}_{x_0} (\mathfrak{O}/\pi^m \mathfrak{O}) \big) \qquad\textrm{for $m \geq 1$}
\end{equation*}
of $K$.

Let $k$ be an arbitrary but fixed field and let $\Mod(G)$ denote the abelian category of smooth representations of $G$ in $k$-vector spaces. The compact induction $\mathbf{X} := \ind_I^G(1)$ of the trivial $I$-representation over $k$ lies in $\Mod(G)$, and we put
\begin{equation*}
    H := \End_{k[G]}(\mathbf{X})^{\mathrm{op}}
\end{equation*}
so that $\mathbf{X}$ becomes a $(G,H)$-bimodule. We often will identify $H$, as a right $H$-module, via the map
\begin{align*}
    H & \xrightarrow{\; \cong \;} \ind_I^G(1)^I \subseteq \mathbf{X} \\
    h & \longmapsto (\mathrm{char}_I) h
\end{align*}
(where $\mathrm{char}_A$, for any compact open subset $A \subseteq G$, denotes the characteristic function of $A$) with the submodule $\ind_I^G(1)^I$ of $I$-fixed vectors in $\ind_I^G(1)$.

We have
\begin{itemize}
  \item[--] the left exact functor
\begin{align*}
    \fhom : \Mod(G) & \longrightarrow \Mod(H) \\
              V & \longmapsto V^I = \Hom_{k[G]}(\mathbf{X},V) \ ,
\end{align*}

  \item[--] and the right exact functor
\begin{align*}
    \ften_0 : \Mod(H) & \longrightarrow \Mod(G) \\
                M & \longmapsto \mathbf{X} \otimes_H M \ .
\end{align*}
\end{itemize}
Let $\Mod^I(G)$ denote the full subcategory of $\Mod(G)$ of all smooth $G$-representations $V$ which are generated by $V^I$.

\begin{remark}\phantomsection\label{easy}
\begin{itemize}
  \item[i.] $\ften_0$ is left adjoint to $\fhom$.
  \item[ii.] The image of $\ften_0$ is contained in $\Mod^I(G)$.
  \item[iii.] The restriction $\fhom | \Mod^I(G)$ is faithful.
\end{itemize}
\end{remark}

For any left, resp.\ right, $H$-module $M$ we introduce the right, resp.\ left, $H$-module
\begin{equation*}
    M^\ast := \Hom_H(M,H) \ .
\end{equation*}
As usual, $M$ is called reflexive if the canonical map $M \longrightarrow M^{\ast\ast}$ is an isomorphism.
In particular, we have the $(H,G)$-bimodule
\begin{equation*}
    \mathbf{X}^\ast = \Hom_H(\mathbf{X},H) \ ,
\end{equation*}
but on which the $G$-action is no longer smooth. We consider the left exact functor
\begin{align*}
    \ften_1 : \Mod(H) & \longrightarrow \Mod(k[G]) \\
                M & \longmapsto \Hom_H(\mathbf{X}^\ast,M)
\end{align*}
together with the natural transformation
\begin{align}\label{defitauM}
    \tau_M :\ften_0(M) = \mathbf{X} \otimes_H M & \longrightarrow \Hom_H(\mathbf{X}^\ast,M) = \ften_1(M) \\
             x \otimes m & \longmapsto [\lambda \mapsto \lambda(x)m] \ .    \nonumber
\end{align}
This allows us to introduce the functor
\begin{align}\label{modiftens}
    \ften : \Mod(H) & \longrightarrow \Mod^I(G) \\
                M & \longmapsto \im(\tau_M) \ .   \nonumber
\end{align}
The main goal of this paper is to investigate its properties.

\begin{remark}\label{inj-surj}
The functor $\ften$ preserves injective as well as surjective maps.
\end{remark}

\subsection{On the $H$-module structure of $\mathbf{X}$ \label{sec:sur}}
Bruhat-Tits associate with each facet $F$ of the semisimple building $\mathscr{X}$, in a $G$-equivariant way, a smooth affine $\mathfrak{O}$-group scheme $\mathbf{G}_F$ whose general fiber is $\mathbf{G}$ and such that $\mathbf{G}_F(\mathfrak{O})$ is the pointwise stabilizer in $G$ of the preimage of $F$ in the extended building. The  connected component of $\mathbf{G}_F$ is denoted by $\mathbf{G}_F^\circ$ so that the reduction $\overline{\mathbf{G}}_F^\circ$ over $\mathbb{F}_q$ is a connected smooth algebraic group. The subgroup $\mathbf{G}_F^\circ(\mathfrak{O})$ of $G$ is compact open. Let
\begin{equation*}
    I_F := \{ g \in \mathbf{G}_F^\circ(\mathfrak{O}) :( g \bmod \pi) \in\ \textrm{unipotent radical of $\overline{\mathbf{G}}_F^\circ$} \}.
\end{equation*}
The $I_F$ are compact open pro-$p$ subgroups in $G$ which satisfy $I_C = I$, $I_{x_0} = K_1$, as well as
\begin{equation}\label{f:pr1}
    gI_F g^{-1} = I_{gF} \qquad\textrm{for any $g \in G$},
\end{equation}
and
\begin{equation}\label{f:pr2}
    I_{F'} \subseteq I_F \qquad\textrm{whenever $F' \subseteq \overline{F}$}.
\end{equation}
Let $\P_F^\dagger$ denote the stabilizer in $G$ of the facet $F$. If $F$ is contained in the closure $\overline{C}$ of $C$, then $I_F \subseteq I \subseteq J \subseteq \P_F^\dagger$ with $I_F$ being normal in $\P_F^\dagger$.

Let $F$ be a facet contained in $\overline{C}$. Extending functions by zero induces embeddings
\begin{equation*}
    \mathbf{X}_F := \ind^{{\mathbf G}_F^\circ(\mathfrak O)}_I(1) \hookrightarrow \mathbf{X}_F^\dagger := \ind^{\P_F^\dagger}_I(1) \hookrightarrow \mathbf{X} \ .
\end{equation*}
Correspondingly we have the $k$-subalgebras
\begin{equation*}
    H_F := \End_{k[{\mathbf G}_F^\circ(\mathfrak O)]}(\mathbf{X}_F)^{\mathrm{op}} = [\ind^{{\mathbf G}_F^\circ(\mathfrak O)}_I(1)]^I \quad\text{and}\quad H_F^\dagger := \End_{k[\P_F^\dagger]}(\mathbf{X}_F^\dagger)^{\mathrm{op}} = \ind^{\P_F^\dagger}_I(1)^I
\end{equation*}
of $H$.

\begin{proposition}\label{firstlevel}
For any facet $F \subseteq \overline{C}$ the maps
\begin{equation*}
    \mathbf{X}_F ^\dagger\otimes_{H_F^\dagger} H  \xrightarrow{\cong} \mathbf{X}^{I_F} , \
    \mathbf{X}_F \otimes_{H_F} H_F^\dagger \xrightarrow{\cong}  \mathbf{X}_F ^\dagger , \ \text{and} \
    \mathbf{X}_F \otimes_{H_F} H  \xrightarrow{\cong} \mathbf{X}^{I_F}
\end{equation*} given by $f \otimes h \longmapsto h(f)$
 are isomorphisms of $(\P_F^\dagger, H)$-, $({\mathbf G}_F^\circ(\mathfrak O), H_F^\dagger)$-, and $({\mathbf G}_F^\circ(\mathfrak O), H)$-bimodules, respectively.
\end{proposition}
\begin{proof}
See \cite{OS} Lemma 3.8, Lemma 4.24, Prop.\ 4.25.
\end{proof}

\begin{proposition}\label{H-free}
For any facet $F \subseteq \overline{C}$ we have:
\begin{itemize}
  \item[i.] The algebra $H$ is free as a left or right $H_F^\dagger$-module; it also is free as a left or right $H_F$-module.
  \item[ii.] $H_F$ is a Frobenius algebra.
\end{itemize}
\end{proposition}
\begin{proof}
For i., see  \cite{OS} Prop.\ 4.21.i.; for  ii. see  \cite{Tin} Prop.\ 3.7 and \cite{Saw} Thm.\ 2.4.
\end{proof}

\subsubsection{$^\ast$-acyclicity of $H$-modules}

A left or right $H$-module $X$ is called $^\ast$-acyclic if
\begin{equation*}
    \Ext_H^i(X,H) = 0 \qquad\textrm{for any $i \geq 1$}
\end{equation*}
holds true. The functors $\Ext_H^i(.,H)$ transform arbitrary direct sums into direct products. Hence arbitrary direct sums of $^\ast$-acyclic modules are $^\ast$-acyclic.

Note that,  $H_F$ being a Frobenius algebra, it is self-injective and we have $\Ext_{H_F}^i(Y,H_F) = 0$ for any $i \geq 1$ and any $H_F$-module $Y$.

\begin{corollary}\label{basechange-acyclic}
Let $Y$ be a right $H_F$-module which is a possibly infinite direct sum of finitely generated $H_F$-modules; then the $H$-module $Y \otimes_{H_F} H$ is $^\ast$-acyclic.
\end{corollary}
\begin{proof}
We may assume that $Y$ is finitely generated. The ring $H_F$ being noetherian, we then find a projective resolution $P_\bullet \longrightarrow Y$ by finitely generated projective $H_F$-modules. Then $P_\bullet \otimes_{H_F} H \longrightarrow Y \otimes_{H_F} H$, because of Prop.\ \ref{H-free}.i, is a projective resolution of $H$-modules, and we compute
\begin{equation*}
    \Ext_H^i(Y \otimes_{H_F} H,H) \\
    = h^i(\Hom_H(P_\bullet \otimes_{H_F} H,H))
    = h^i(\Hom_{H_F}(P_\bullet,H)) \\
    = \Ext_{H_F}^i(Y,H) \ .
\end{equation*}
Since the $P_\bullet$ are finitely generated modules the last identity in the above computation shows that the functors $\Ext_{H_F}^i(Y,.)$ commute with arbitrary direct sums. Hence using Prop.\ \ref{H-free}.i again we obtain
\begin{equation*}
    \Ext_H^i(Y \otimes_{H_F} H,H) = \Ext_{H_F}^i(Y,H) = H \otimes_{H_F} \Ext_{H_F}^i(Y,H_F) \ .
\end{equation*}
By the self-injectivity of the algebra $H_F$ the right hand side vanishes for $i \geq 1$ which establishes the $^\ast$-acyclicity of $Y \otimes_{H_F} H$.
\end{proof}

\begin{proposition}\label{level1-acyclic}
The $H$-modules $\mathbf{X}^{K_1}$ and $\mathbf{X}^{K_1}/H$ are $^\ast$-acyclic.
\end{proposition}
\begin{proof}
Since $\mathbf{X}_{x_0}$ is a finitely generated $H_{x_0}$-module the $^\ast$-acyclicity of $\mathbf{X}^{K_1}$ follows immediately from Prop.\ \ref{firstlevel} and Cor.\ \ref{basechange-acyclic}.
By the same argument $\mathbf{X}^{K_1}/H = (\mathbf{X}_{x_0}/H_{x_0}) \otimes_{H_{x_0}} H$ is $^\ast$-acyclic.
\end{proof}

\subsubsection{The hypothesis \sur}

As will become apparent the key object to understand is the homomorphism of left $H$-modules
\begin{equation*}
  \rho : \mathbf{X}^* \longrightarrow H^* = H
\end{equation*}
which is dual to the inclusion $H \subseteq \mathbf{X}$. Since we will achieve a complete understanding only in special cases we will usually work under the following assumption.\\

\noindent\textbf{Hypothesis \sur:} The map $\rho : \mathbf{X}^* \longrightarrow H^* = H$ is surjective.

\begin{remark}\label{sur}
Equivalent are:
\begin{itemize}
  \item[i.] \sur holds true.
  \item[ii.] $H \subseteq \mathbf{X}$ is a direct factor as a (right) $H$-module.
\end{itemize}
\end{remark}

\begin{lemma}\label{not-p-sur}
If the characteristic of $k$ is different from $p$, then \sur holds true.
\end{lemma}
\begin{proof}
Under our assumption the map ``convolution from the left with the unit element $\mathrm{char}_I \in H$'' explicitly given by
\begin{align*}
  \mathbf{X} & \longrightarrow H \\
  f & \longmapsto \frac{1}{[I:I_f]} \sum_{g \in I/I_f} f(g^{-1}.) \ ,
\end{align*}
where $I_f \subseteq I$ is any open subgroup fixing $f$, is well defined and provides a splitting of the inclusion $H \subseteq \mathbf{X}$.
\end{proof}

\subsection{\label{sec:defiTP}The relevant torsion pair}

We recall the descending sequence $K_1 \supseteq \ldots \supseteq K_m \supseteq \ldots$ of open normal subgroups in $I$ as defined in \S\ref{setting}. For any $m \geq 1$ there is the obvious commutative diagram
\begin{equation}\label{gradedtau}
    \xymatrix{
      \mathbf{X} \otimes_H M  \ar[r]^-{\tau_M} &
      \Hom_H(\mathbf{X}^\ast,M)  \\
      \mathbf{X}^{K_m} \otimes_H M \ar[u] \ar[r] &
      \Hom_H((\mathbf{X}^{K_m})^\ast,M). \ar[u]  }
\end{equation}
The inclusion $\mathbf{X}^{K_m} \subseteq \mathbf{X}$ dualizes to a map of $(I,H)$-bimodules $\mathbf{X}^\ast \longrightarrow (\mathbf{X}^{K_m})^\ast$. We define
\begin{equation*}
  \mathbf{Z}_m := \im (\mathbf{X}^\ast \longrightarrow (\mathbf{X}^{K_m})^\ast) \ .
\end{equation*}
which is an $(I,H)$-bimodule. The above commutative diagram  can now be rewritten as
\begin{equation*}
    \xymatrix{
      \mathbf{X} \otimes_H M  \ar[r]^-{\tau_M} &
      \Hom_H(\mathbf{X}^\ast,M)  \\
      \mathbf{X}^{K_m} \otimes_H M \ar[u] \ar[r] &
      \Hom_H(\mathbf{Z}_m,M). \ar[u]  }
\end{equation*}
such that the right perpendicular arrow has become injective. Since $\mathbf{X} = \bigcup_m \mathbf{X}^{K_m}$ we see that
\begin{equation*}
    \ften (M) \subseteq \bigcup_m \Hom_H(\mathbf{Z}_m,M) \subseteq \Hom_H(\mathbf{X}^\ast,M) \ .
\end{equation*}
We now introduce the family $\mathcal{Z} = \{Z_m\}_{m \geq 1}$ of $H$-modules
\begin{equation}\label{defiZm}
    Z_m := \ker \big( (\mathbf{Z}_m)_I \longrightarrow ((\mathbf{X}^{K_m})^\ast)_I \xrightarrow{\; (\subseteq^\ast)_I \;} (\mathbf{X}^I)^\ast = H \big)
\end{equation}
and let $(\mathcal{F}_G, \mathcal{T}_G)$ denote the torsion pair in $\Mod(H)$ generated by $\mathcal{Z}$.

\begin{theorem}\label{F-embeds}
Assuming \sur we have $\fhom \circ \ften | \mathcal{F}_G = \id_{\mathcal{F}_G}$.
\end{theorem}
\begin{proof}
For any $M$ in $\mathcal{F}_G$ we have the chain of maps
\begin{multline*}
     M \longrightarrow \Hom_G(\mathbf{X},\mathbf{X} \otimes_H M) = (\mathbf{X} \otimes_H M)^I \longrightarrow \ften(M)^I \\
     \xrightarrow{\; \subseteq \;}
    \bigcup_m \Hom_H(\mathbf{Z}_m,M)^I = \bigcup_m \Hom_H((\mathbf{Z}_m)_I,M)
\end{multline*}
which sends an element $m$ to $\lambda \mapsto \lambda(\mathrm{char}_I)m$. As a consequence of our assumption \sur we have, on the other hand,  the exact sequence of $H$-modules
\begin{equation}\label{defiZm-sur}
  0 \longrightarrow Z_m \longrightarrow (\mathbf{Z}_m)_I \longrightarrow H \longrightarrow 0.
\end{equation}
It induces a map $M = \Hom_H(H,M) \longrightarrow \Hom_H((\mathbf{Z}_m)_I,M)$ which is easily checked to coincide with the above one. But by the definition of $\mathcal{F}_G$ this latter map must be an isomorphism. It follows that $\ften(M)^I = M$.
\end{proof}

\begin{corollary}\label{fully-faithful}
Assuming \sur we have:
\begin{itemize}
  \item[i.] The functor $\ften|\mathcal{F}_G : \mathcal{F}_G \longrightarrow \Mod^I(G)$ is fully faithful.
  \item[ii.] If $f : V_1 \longrightarrow V_2$ is a surjective homomorphism in the essential image of the functor $\ften |\mathcal{F}_G$ then the map $h(f) : V_1^I \longrightarrow V_2^I$ is surjective as well.
\end{itemize}
\end{corollary}
\begin{proof}
i. This follows from Thm.\ \ref{F-embeds} and Remark \ref{easy}.iii. ii. We factorize $f$ as
\begin{equation*}
  \xymatrix{
  V_1^I \ar[rr]^{h(f)} \ar[dr]_-{\overline{h(f)}}
                &  &    V_2^I     \\
                & \im(h(f)). \ar[ur]_-{\subseteq}                }
\end{equation*}
With $V_2^I$ also $\im(h(f))$ lies in $\mathcal{F}_G$. Therefore, in order to see that $\im(h(f)) = V_2^I$ holds true it suffices to check that $\ften(\subseteq) : \ften(\im(h(f))) \longrightarrow \ften(V_2^I) \cong V_2$ is bijective. But applying $\ften$ to the above diagram, by Remark \ref{inj-surj}, leads to a commutative diagram
\begin{equation*}
  \xymatrix{
  V_1 \ar[rr]^{f} \ar[dr]
                &  &  V_2 \cong \ften(V_2^I)     \\
                & \ften(\im(h(f))) \ar[ur]_-{\ften(\subseteq)}                }
\end{equation*}
in which the map $\ften(\subseteq)$ is injective. Hence the surjectivity of $f$ implies that $\ften(\subseteq)$ is bijective.
\end{proof}

\begin{proposition}\label{Z1}
$Z_1 = 0$.
\end{proposition}
\begin{proof}
By Prop.\ \ref{firstlevel} we have $\mathbf{X}_{x_0} \otimes_{H_{x_0}} H \cong \mathbf{X}^{K_1}$ as $(K,H)$-bimodules. It follows that $(\mathbf{X}^{K_1})^\ast = \Hom_{H_{x_0}}(\mathbf{X}_{x_0},H)$. Since $H$ is free as a right $H_{x_0}$-module by Prop.\ \ref{H-free}.i, we further have $\Hom_{H_{x_0}}(\mathbf{X}_{x_0},H) = H \otimes_{H_{x_0}} \Hom_{H_{x_0}}(\mathbf{X}_{x_0},H_{x_0})$. The $I$-action on the right hand term is through the projection of $I$ onto $\overline{\mathbf{N}}(\mathbb{F}_q)$ which acts on $\mathbf{X}_{x_0}$. Altogether we obtain
\begin{equation*}
    ((\mathbf{X}^{K_1})^\ast)_I = H \otimes_{H_{x_0}} \Hom_{H_{x_0}}(\mathbf{X}_{x_0},H_{x_0})_{\overline{\mathbf{N}}(\mathbb{F}_q)} \ .
\end{equation*}
We claim that the natural map
\begin{equation*}
    \Hom_{H_{x_0}}(\mathbf{X}_{x_0}, H_{x_0})_{\overline{\mathbf{N}}(\mathbb{F}_q)}  \longrightarrow \Hom_{H_{x_0}}(\mathbf{X}_{x_0}^{\overline{\mathbf{N}}(\mathbb{F}_q)} , H_{x_0}) = \Hom_{H_{x_0}}(H_{x_0},H_{x_0}) = H_{x_0}
\end{equation*}
is bijective, which implies that $((\mathbf{X}^{K_1})^\ast)_I = H$ and hence that $Z_1 = 0$. Since $H_{x_0}$ is Frobenius (Prop.\ \ref{H-free}.ii) we have a natural isomorphism of functors $\Hom_{H_{x_0}}(., H_{x_0}) \cong \Hom_k(.,k)$ on the category of all $H_{x_0}$-modules. Our map being a map between finite dimensional $k$-vector spaces its bijectivity therefore can be tested by dualizing. But when applying the functor $\Hom_{H_{x_0}}(., H_{x_0})$ both sides become equal to $\mathbf{X}_{x_0}^{\overline{\mathbf{N}}(\mathbb{F}_q)}$.
\end{proof}

\begin{remark}\label{adjoint-condition}
For any $V$ in $\Mod(G)$ we have $\fhom(V) \in \mathcal{F}_G$ if and only if
\begin{equation*}
    \Hom_{k[G]}(\mathbf{X} \otimes_H Z_m,V) = 0 \qquad \text{for any $m \geq 2$}.
\end{equation*}
\end{remark}
\begin{proof}
This is immediate from the adjointness of $\fhom$ and $\ften_0$.
\end{proof}

\begin{lemma}\label{not-p-zero}
Suppose that the characteristic of $k$ is different from $p$. Then $Z_m = 0$ for any $m \geq 1$, and the functor $\ften : \Mod(H) \longrightarrow \Mod^I(G)$ is fully faithful.
\end{lemma}
\begin{proof}
Under our assumption the natural projection map $U^{I/K_m} \xrightarrow{\cong} U_{I/K_m}$ from the invariants to the coinvariants is an isomorphism for any $k[I/K_m]$-module $U$. We apply this to the commutative diagram
\begin{equation*}
  \xymatrix{
    (\mathbf{Z}_m)_I = (\mathbf{Z}_m)_{I/K_m} \ar[r] & ((\mathbf{X}^{K_m})^*)_{I/K_m}  \ar@{->>}[r] & H \phantom{.} \\
    \mathbf{Z}_m^{I/K_m} \ar[u]^{\cong} \ar@{^{(}->}[r] & ((\mathbf{X}^{K_m})^*)^{I/K_m} \ar[u]^{\cong} \ar@{->>}[r] & H . \ar[u]^{=}  }
\end{equation*}
The composed map in the upper row is surjective by Lemma \ref{not-p-sur}. The left hand map in the lower row is injective for trivial reasons. It therefore remains to be seen that the right hand map in the lower row is an isomorphism. But this map
\begin{equation*}
  ((\mathbf{X}^{K_m})^*)^{I/K_m} = ((\mathbf{X}^{K_m})_{I/K_m})^* \longrightarrow (\mathbf{X}^I)^* = H
\end{equation*}
is the dual of the isomorphism $\mathbf{X}^I = (\mathbf{X}^{K_m})^{I/K_m} \xrightarrow{\cong} (\mathbf{X}^{K_m})_{I/K_m}$.

The second part of the assertion follows from the first part by Cor.\ \ref{fully-faithful}.i.
\end{proof}


\section{The rank 1 case \label{sec:2}}
\subsection{\label{hyposur}A filtration of $\mathbf X$ as an $H$-module and the hypothesis \sur}
Throughout this section we \textbf{assume} that $k$ has characteristic $p$  and that $\mathbf{G}$ is one of the groups $\mathbf{SL_2}$, $\mathbf{PGL_2}$, or $\mathbf{GL_2}$. The building $\mathscr{X}$ then is a tree. Let $\mathscr{X}_0$, resp.\ $\mathscr{X}_1$, resp.\ $\mathscr{X}_{(1)}$, denote the set of vertices, resp.\ edges, resp.\ oriented edges, of $\mathscr{X}$. The two vertices in the closure of any oriented edge $\vec{e}$ can be distinguished as the origin $o(\vec{e})$ and the target $t(\vec{e})$ of $\vec{e}$. Moreover, $\sigma(\vec{e})$ denotes the oriented edge with the same underlying edge as $\vec{e}$ but the reversed orientation. By a slight abuse of notation $e$ sometimes is understood to denote the edge which underlies an oriented edge $\vec{e}$. For any vertex $x \in \mathscr{X}_0$ we abbreviate $K_x := \mathbf{G}_x(\mathfrak{O})$.

As a consequence of \eqref{f:pr2} the family $\{\mathbf{X}^{I_F}\}_F$ of subspaces of $\mathbf{X}$ forms a $G$-equivariant coefficient system $\cX$ of right $H$-modules on $\mathscr{X}$. The associated augmented oriented chain complex
\begin{equation}\label{f:complex}
    0 \longrightarrow C_c^{or} (\mathscr{X}_{(1)}, \cX) \xrightarrow{\;\partial\;}  C_c (\mathscr{X}_0, \cX) \xrightarrow{\;\epsilon\;} \mathbf{X} \longrightarrow 0
\end{equation}
is an exact (cf.\ \cite{OS} Remark 3.2.1) sequence of $(G,H)$-bimodules. We recall that
\begin{align*}
    C_c (\mathscr{X}_0, \cX) & := \bigoplus_{x \in \mathscr{X}_0} \mathbf{X}^{I_x} , \quad \epsilon((v_x)_{x \in \mathscr{X}_0}) := \sum_{x \in \mathscr{X}_0} v_x \; , \\
    C_c^{or} (\mathscr{X}_{(1)}, \cX) & := \{(v_{\vec{e}})_{\vec{e} \in \mathscr{X}_{(1)}} \in \bigoplus_{\vec{e} \in \mathscr{X}_{(1)}} \mathbf{X}^{I_e} : v_{\sigma(\vec{e})} = - v_{\vec{e}} \ \text{for any $\vec{e} \in \mathscr{X}_{(1)}$} \} , \; \text{and} \\
    \partial((v_{\vec{e}})_{\vec{e} \in \mathscr{X}_{(1)}}) & := (\sum_{t(\vec{e}) = x} v_{\vec{e}})_{x \in \mathscr{X}_0} \ .
\end{align*}

Let $\mathscr{X}^{[m]}$, for any $m \geq 0$, denote the subtree of $\mathscr{X}$ with vertices of distance $\leq m$ from $x_0$.

\begin{remark}\label{fix}
For any $m \geq 0$ we have:
\begin{itemize}
  \item[i.] $\mathscr{X}^{[m]}$ is the fixed point set of $K_m$ in $\mathscr{X}$ (with $K_0 := K_{x_0}$);
 \item[ii.] for any vertex $x$ of $\mathscr{X}^{[m]}$ we have $I_x \supseteq K_{m+1}$. \end{itemize}
\end{remark}
\begin{proof}
i. This is easily seen by using the interpretation of the set $\mathscr{X}_0$ in terms of $\mathfrak{O}$-lattices in $\mathfrak{F}^2$. ii. Since $K_{m+1}$, by i., fixes $x$ we have $K_x \supseteq K_{m+1}$. But, in fact, $K_{m+1}$ fixes all neighbouring vertices of $x$. Hence $K_{m+1}$ must project into the center of $K_x/I_x$. But $K_{m+1}$ is a pro-$p$ group, and the order of this center is prime to $p$.
\end{proof}

We now restrict the coefficient system $\cX$ to $\mathscr{X}^{[m]}$. Since $\mathscr{X}^{[m]}$ is a finite tree we omit in the following the subscript `c' in the notation for spaces of chains on $\mathscr{X}^{[m]}$.

\begin{lemma}\label{shortexact}
We have, for any $m \geq 0$, the short exact sequence
\begin{equation}\label{f:m-complex}
    0 \longrightarrow C^{or} (\mathscr{X}^{[m]}_{(1)}, \cX) \xrightarrow{\;\partial\;} C (\mathscr{X}^{[m]}_0, \cX) \xrightarrow{\;\epsilon\;} \mathbf{X}^{K_{m+1}} \longrightarrow 0
\end{equation}
of $(K,H)$-bimodules.
\end{lemma}
\begin{proof}
Recall the short exact sequence \eqref{f:complex}. Due to the tree structure of $\mathscr{X}$ the exactness of \eqref{f:m-complex} at the left and the middle term is an immediate consequence. Remark \ref{fix}.ii implies
\begin{equation*}
    \epsilon \big( C^{or} (\mathscr{X}^{[m]}_{(0)}, \cX) \big) \subseteq \mathbf{X}^{K_{m+1}} \ .
\end{equation*}
It remains to be checked that this last inclusion indeed is an equality. It is shown in \cite{Oll1} Prop.\ 2.3 that $\mathbf{X}^{K_{m+1}}$, as an $H$-module, is generated by the characteristic functions $\mathrm{char}_{gI}$, for $g \in G$, which are fixed by $K_{m+1}$.
The latter condition is equivalent to $K_{m+1} \subseteq gIg^{-1} = I_{gC}$ (recall that $C$ denotes the standard edge). In particular, by Remark \ref{fix}.i, the two vertices of $gC$ lie in $\mathscr{X}^{[m+1]}$. One of them, call it $x$, already has to lie in $\mathscr{X}^{[m]}$, and we have $\mathrm{char}_{gI} \in \mathbf{X}^{gIg^{-1}} = \mathbf{X}^{I_{gC}} \subseteq \mathbf{X}^{I_x}$. Hence $\mathrm{char}_{gI}$ is the augmentation of a $0$-chain on $\mathscr{X}^{[m]}$ (supported on $x$).
\end{proof}

The following result is due to Ollivier (\cite{Oll1} Prop.\ 2.7). By using the formalism developed above we will reprove it in a slightly less computational way. Let $x_0$ and $x_1$ denote the two vertices in the closure of $C$.

\begin{proposition}\label{level-subquotients}
For any $m \geq 1$ there is an $i \in \{0,1\}$ and an integer $n_m \geq 1$ such that
\begin{equation*}
    \mathbf{X}^{K_{m+1}}/\mathbf{X}^{K_m} \cong (\mathbf{X}^{I_{x_i}}/\mathbf{X}^I)^{n_m}
\end{equation*}
as $H$-modules.
\end{proposition}
\begin{proof}
The complement $\mathscr{X}^{[m]} \setminus \mathscr{X}^{[m-1]}$ consists of the vertices $x$ of exact distance $d(x_0,x) = m$ from $x_0$ and, for each such vertex $x$, of the unique edge $C(x)$ emanating from $x$ in direction $x_0$. Therefore, if we divide the exact sequence \eqref{f:m-complex} for $m+1$ by the corresponding sequence for $m$ then we obtain an isomorphism
\begin{equation*}
    \bigoplus_{d(x,x_0) = m} \mathbf{X}^{I_x}/\mathbf{X}^{I_{C(x)}} = \mathbf{X}^{K_{m+1}}/\mathbf{X}^{K_m} \ .
\end{equation*}
of $(K,H)$-bimodules. The group $K$ acts transitively on the set of all vertices of distance $m$ from $x_0$. Let $n_m$ be the cardinality of this set and fix one vertex $y_m$ in this set. For any $g \in K$ the map
\begin{equation*}
  \mathbf{X}^{I_{y_m}}/\mathbf{X}^{I_{C(y_m)}} \xrightarrow{\; g \;} \mathbf{X}^{I_{gy_m}}/\mathbf{X}^{I_{C(gy_m)}}
\end{equation*}
is an isomorphism of $H$-modules. Hence $\mathbf{X}^{K_{m+1}}/\mathbf{X}^{K_m} \cong (\mathbf{X}^{I_{y_m}}/\mathbf{X}^{I_{C(y_m)}})^{n_m}$ as $H$-modules. Furthermore, since $\mathbf{SL_2}(\mathfrak{F})$ acts transitively on the set of all edges of $\mathscr{X}$ we find an element $g_m \in \mathbf{SL_2}(\mathfrak{F})$ such that $g_m C = C(y_m)$. Then $y_m = g_mx_i$ for $i = 0$ or $1$, and $\mathbf{X}^{I_{x_i}}/\mathbf{X}^{I_C} \xrightarrow{\; g_m \;} \mathbf{X}^{I_{y_m}}/\mathbf{X}^{I_{C(y_m)}}$ is an isomorphism of $H$-modules.
\end{proof}

\begin{theorem}\label{level-acyclic}
The $H$-modules $\mathbf{X}^{K_m}$, $\mathbf{X}^{K_{m+1}}/\mathbf{X}^{K_m}$, and $\mathbf{X}^{K_m}/H$, for any $m \geq 1$, are $^\ast$-acyclic.
\end{theorem}
\begin{proof}
The $^\ast$-acyclicity of $\mathbf{X}^{I_{x_i}}$ and $\mathbf{X}^{I_{x_i}}/H$ was shown, in general, in Prop.\ \ref{level1-acyclic}. This together with Prop.\ \ref{level-subquotients} implies the $^\ast$-acyclicity of $\mathbf{X}^{K_{m+1}}/\mathbf{X}^{K_m}$ for any $m \geq 1$. The remaining cases follow inductively by a long exact sequence argument.
\end{proof}

\begin{corollary}\label{restriction}
The maps $\mathbf{X}^* \longrightarrow (\mathbf{X}^{K_m})^\ast$, $(\mathbf{X}^{K_{m+1}})^\ast \longrightarrow (\mathbf{X}^{K_m})^\ast$ and $(\mathbf{X}^{K_{m+1}}/H)^\ast \longrightarrow (\mathbf{X}^{K_m}/H)^\ast$ dual to the obvious inclusions, for any $m \geq 1$, are surjective. In particular, $\mathbf{Z}_m = (\mathbf{X}^{K_m})^\ast$ for any $m \geq 1$.\end{corollary}

\begin{corollary}\label{X-acyclic}
The $H$-modules $\mathbf{X}$ and $\mathbf{X}/H$ are $^\ast$-acyclic.
\end{corollary}
\begin{proof}
We consider the spectral sequence
\begin{equation*}
    E_2^{i,j} = {\varprojlim_m}^{(i)} \Ext_{H}^j(\mathbf{X}^{K_m}/{H},{H}) \Longrightarrow \Ext_{H}^{i+j}(\varinjlim_m \mathbf{X}^{K_m}/H,H) = \Ext_H^{i+j}(\mathbf{X}/H,H) \ .
\end{equation*}
Because of Thm.\ \ref{level-acyclic} it degenerates into the isomorphisms
\begin{equation*}
    {\varprojlim_m}^{(i)} \Hom_H(\mathbf{X}^{K_m}/H,H) \cong \Ext_H^i(\mathbf{X}/H,H) \ .
\end{equation*}
But as a consequence of Cor.\ \ref{restriction} the left hand side vanishes for $i \geq 1$. Hence $\mathbf{X}/H$ is $^\ast$-acyclic. But with $\mathbf{X}/H$ and $H$ also $\mathbf{X}$ is $^\ast$-acyclic.
\end{proof}

\begin{corollary}\label{split}
$H \subseteq \mathbf{X}$ is a direct factor as an $H$-module. In particular, \sur holds true and  $Z_m$ sits in the following exact sequence  of $H$-modules\begin{equation*}
    0 \longrightarrow Z_m \longrightarrow ((\mathbf{X}^{K_m})^\ast)_I \longrightarrow H \longrightarrow 0 \ .
\end{equation*}
\end{corollary}
\begin{proof}
By Cor.\ \ref{X-acyclic} we have $\Ext_H^1(\mathbf{X}/H,H) = 0$. For the exact sequence, use \eqref{defiZm-sur} and Cor.\ \ref{restriction}.
\end{proof}

\begin{lemma}\label{univ-reflexive}
For any $m \geq 1$, the right $H$-module $\mathbf{X}^{K_m}$ is finitely generated and reflexive.
\end{lemma}
\begin{proof}
The finite generation follows immediately from Prop.s \ref{firstlevel} and \ref{level-subquotients}. On the other hand we know from \cite{OS} Thm.\ 0.1 that the algebra $H$ is Gorenstein. Over such a ring any finitely generated $\ast$-acyclic ($=$ maximal Cohen-Macaulay), by Prop.\ \ref{level-acyclic}, module is reflexive (\cite{Buc} Lemma 4.2.2(iii)).
\end{proof}

For later use  we also record the following technical consequence. Let $k[[I]]$ be the completed group ring of $I$ over $k$.

\begin{corollary}\label{proj-res}
For any finitely generated projective $k[[I]]$-module $P$ and any $m \geq 1$ the $H$-module $\Hom_{k[[I]]}(P,\mathbf{X}^{K_m})$ is finitely generated and $^\ast$-acyclic.
\end{corollary}
\begin{proof}
By finite direct sum arguments it suffices to consider the case $P = k[[I]]$. Then $\Hom_{k[[I]]}(P,\mathbf{X}^{K_m}) = \mathbf{X}^{K_m}$ and the assertion follows from Thm.\ \ref{level-acyclic} and Lemma \ref{univ-reflexive}.
\end{proof}

\begin{proposition}\label{reflexive}
Any reflexive left $H$-module $M$ belongs to $\mathcal{F}_G$.
\end{proposition}
\begin{proof}
Since the defining sequences
\begin{equation*}
    0 \longrightarrow Z_m \longrightarrow ((\mathbf{X}^{K_m})^\ast)_I \longrightarrow H \longrightarrow 0
\end{equation*}
split we have, for any left $H$-module $M$, the exact sequences
\begin{equation*}
    0 \longrightarrow \Hom_H(H,M) = M  \longrightarrow \Hom_H(((\mathbf{X}^{K_m})^\ast)_I,M) \longrightarrow \Hom_H(Z_m,M) \longrightarrow 0 \ .
\end{equation*}
Using Lemma \ref{univ-reflexive} we compute
\begin{align*}
    \Hom_H(((\mathbf{X}^{K_m})^\ast)_I,M) & = \Hom_H((\mathbf{X}^{K_m})^\ast,M)^I \\
    & = \Hom_H(M^\ast,\mathbf{X}^{K_m})^I = \Hom_H(M^\ast,\mathbf{X}^I) = \Hom_H(M^\ast,H) = M
\end{align*}
for any reflexive $M$. Hence the above exact sequences become
\begin{equation*}
  0 \longrightarrow M \xrightarrow{\;=\;} M \longrightarrow \Hom_H(Z_m,M) \longrightarrow 0 \ .
\end{equation*}
It follows that $\Hom_H(Z_m,M) = 0$.
\end{proof}

We know that $H$ is a Gorenstein ring. Therefore (\cite{Buc} Thm.\ 5.1.4) any finitely generated $H$-module $M$ sits in a short exact sequence of finitely generated $H$-modules
\begin{equation*}
    0 \longrightarrow M \longrightarrow Q \longrightarrow N \longrightarrow 0
\end{equation*}
where $N$ is maximal Cohen-Macaulay (hence reflexive) and $Q$ has finite projective dimension. In particular, we have $\Hom_H(Z_m,M) = \Hom_H(Z_m,Q)$.
Hence it is a basic problem to determine which finitely generated modules of finite projective dimension belong to $\mathcal{F}_G$.

\begin{remark}\label{Zm}
The $H$-modules $Z_m$, for any $m \geq 1$, are finitely generated and satisfy $Z_m^* = 0$.
\end{remark}
\begin{proof}
The second half of the assertion is immediate from the fact that $H$ belongs to $\mathcal{F}_G$ by Prop.\ \ref{reflexive}. For the first half we note that for any finitely generated module $M$ over the noetherian ring $H$ the module $M^*$ is finitely generated as well. Hence $((\mathbf{X}^{K_m})^\ast)_I$ is finitely generated by Lemma \ref{univ-reflexive} and then also its submodule $Z_m$.
\end{proof}

\subsection{A group cohomological formula for  the derived dual of $Z_m$}{\label{cohoformu}

For the rest of this section we \textbf{assume}  that $\mathbf{G}$ is $\mathbf{SL_2}$ or $\mathbf{PGL_2}$. Then we know from \cite{OS} Thm.\ 0.2 that $H$ is Auslander-Gorenstein of self-injective dimension equal to $1$.

\begin{lemma}\label{mod-AG}
For any finitely generated $H$-module $M$ we have:
\begin{itemize}
  \item[i.] $M$ has a (unique) submodule $M^1$ with the property that an arbitrary submodule $N$ of $M$ is contained in $M^1$ if and only if $N^* = 0$;
  \item[ii.] The canonical map $M \longrightarrow M^{**}$ is surjective and has kernel $M^1$;
  \item[iii.] $M^1$ is the largest submodule of $M$ which is finite dimensional over $k$.
  \item[iv.] $M/M^1$ is $^\ast$-acyclic.
\end{itemize}
\end{lemma}
\begin{proof}
i. and ii. \cite{Lev} Cor.\ 4.3. iii. \cite{OS} Lemma 6.9 and Cor.\ 6.17. iv. \cite{ASZ} proof of Prop.\ 2.5.2.
\end{proof}

\begin{proposition}\phantomsection\label{semireflexive}
\begin{itemize}
  \item[i.] $Z_m$, for any $m \geq 1$, is finite dimensional over $k$.
  \item[ii.] For any finitely generated $H$-module $M$ the quotient $M/M^1$ belongs to $\mathcal{F}_G$.
\end{itemize}
\end{proposition}
\begin{proof}
It follows from Remark \ref{Zm} that $Z_m^1 = Z_m$. Hence the assertion i. is implied by Lemma \ref{mod-AG}.iii. The assertion ii. follows from Prop.\ \ref{reflexive} and Lemma \ref{mod-AG}.ii.
\end{proof}

We establish in this section a group cohomological formula for the module $\Ext^1_H(Z_m,H)$. In the following all $I$-cohomology is the usual cohomology of profinite groups with discrete coefficients (cf.\ \cite{Se1}). In fact, all $I$-representations of which we will take the cohomology are such that they naturally are (discrete) modules over the completed group ring $k[[I]]$.

\begin{lemma}\label{resolution}
Let $M$ be an arbitrary discrete $k[[I]]$-module, and view $k$ as a trivial $k[[I]]$-module; we have:
\begin{itemize}
  \item[i.] If $\mathfrak{F}$ has characteristic zero then there exists an infinite exact sequence of $k[[I]]$-modules
  \begin{equation*}
    \ldots \longrightarrow P_i \longrightarrow \ldots \longrightarrow P_1 \longrightarrow P_0 \longrightarrow k \longrightarrow 0
  \end{equation*}
  where the $P_i$, for any $i \geq 0$, are finitely generated projective; moreover,
  \begin{equation*}
    H^i(I,M) = \Ext_{k[[I]]}^i(k,M) \qquad\textrm{for any $i \geq 0$};
  \end{equation*}
  \item[ii.] if $\mathfrak{F}$ has characteristic $p$ then there exists an exact sequence of $k[[I]]$-modules
  \begin{equation*}
    P_2 \longrightarrow P_1 \longrightarrow P_0 \longrightarrow k \longrightarrow 0
  \end{equation*}
  where the $P_i$, for $0 \leq i \leq 2$, are finitely generated projective; moreover,
  \begin{equation*}
    H^i(I,M) = \Ext_{k[[I]]}^i(k,M) \qquad\textrm{for $0 \leq i \leq 2$}.
  \end{equation*}
\end{itemize}
\end{lemma}
\begin{proof}
We first recall that the category $PC(k[[I]])$ of pseudocompact $k[[I]]$-modules is abelian and has enough projective objects. If we pick a projective resolution
\begin{equation*}
    \ldots \longrightarrow P'_i \longrightarrow \ldots \longrightarrow P'_1 \longrightarrow P'_0 \longrightarrow k \longrightarrow 0
\end{equation*}
of the trivial module in this category then
\begin{equation*}
  H^i(I,M) = h^i(\Hom_{PC(k[[I]])}(P'_\bullet,M)) \qquad\textrm{for any $i \geq 0$}
\end{equation*}
(\cite{Bru} Lemma 4.2(i)).
We also note that any finitely generated projective $k[[I]]$-module is a projective object in the category $PC(k[[I]])$ (\cite{Bru} Cor.\ 1.3).

i. In this situation $I$ is a $p$-adic analytic group, so that the ring $k[[I]]$ is noetherian (\cite{Laz} V.2.2.4). This immediately implies the existence of the asserted resolution, which also can be viewed as a projective resolution in the category of pseudocompact $k[[I]]$-modules. Hence $H^i(I,M) = h^i(\Hom_{k[[I]]}(P_\bullet,M)) = \Ext_{k[[I]]}^i(k,M)$.

ii. It is known from \cite{Lub} Thm.\ 1 that $I$ is a finitely presented pro-$p$ group, which implies (cf.\ \cite{Se1} I\S4.2+3) that the cohomology groups $H^i(I,\mathbb{F}_p)$ are finite for $0 \leq i \leq 2$. That this latter fact then implies the assertion is basically contained in \cite{Ple} Thm.\ 1.6. But, since we work with slightly more general coefficients, we give a simplified version of the argument for the convenience of the reader. We begin our partial resolution with the augmentation map $P_0 := k[[I]] \longrightarrow k \longrightarrow 0$. As a consequence of the finite generation of $I$ its kernel is a finitely generated left ideal in $k[[I]]$ (cf.\ \cite{pLG} Prop.\ 19.5.i). Hence we have an exact sequence of $k[[I]]$-modules
\begin{equation*}
  0 \longrightarrow \Omega_2 \longrightarrow P_1 \longrightarrow P_0 \longrightarrow k \longrightarrow 0
\end{equation*}
with $P_0$ and $P_1$ being finitely generated and projective. Correspondingly we may extend this partial resolution to a full projective resolution
\begin{equation*}
  \ldots \longrightarrow P'_i \xrightarrow{\;\partial_i\;} \ldots \xrightarrow{\;\partial_3\;} P'_2 \xrightarrow{\;\partial_2\;} P'_1 = P_1 \xrightarrow{\;\partial_1\;} P'_0 = P_0 \longrightarrow k \longrightarrow 0 \ .
\end{equation*}
of $k$ in the category $PC(k[[I]])$. We claim that the $k[[I]]$-module $\Omega_2$ necessarily is finitely generated. Since $k$ is the only simple $k[[I]]$-module (cf.\ \cite{pLG} Prop.\ 19.7) it suffices for this, by \cite{vdB} Lemma 4.17, to check that the $k$-vector space $\Hom_{PC(k[[I]])}(\Omega_2,k)$ is finite dimensional. But from the above full resolution we deduce the exact sequence
\begin{multline*}
  \Hom_{PC(k[[I]])}(P'_1,k) = \Hom_{k[[I]]}(P_1,k) \xrightarrow{\Hom_{PC(k[[I]])}(\partial_2,k)} \ker ( \Hom_{PC(k[[I]])}(\partial_3,k))  \\
  = \Hom_{PC(k[[I]])}(\Omega_2,k) \longrightarrow H^2(I,k) \longrightarrow 0 \ .
\end{multline*}
The first term is finite dimensional by the finite generation of $P_1$. Since cohomology commutes with arbitrary filtered direct limits of discrete modules (cf.\ \cite{Se1} I.2.2 Prop.\ 8) we see that also the last term $H^2(I,k) = H^2(I,\mathbb{F}_p) \otimes_{\mathbb{F}_p} k$ is finite dimensional. This establishes our claim, and we may choose the above full resolution in such a way that $P_2 := P'_2$ is finitely generated projective. Then $H^i(I,M) = h^i(\Hom_{PC(k[[I]])}(P'_\bullet,M)) = h^i(\Hom_{k[[I]]}(P_\bullet,M)) = \Ext_{k[[I]]}^i(k,M)$ for $0 \leq i \leq 2$.
\end{proof}

\begin{remark}\label{submod-acyc}
Any submodule of a $^\ast$-acyclic $H$-module is $^\ast$-acyclic.
\end{remark}
\begin{proof}
This is immediate from the long exact $\Ext$-sequence and the vanishing of $\Ext^2_H(.,H)$.
\end{proof}

\begin{proposition}\label{coh-formula}
For any $m \geq 1$ we have:
\begin{itemize}
  \item[a)] Suppose that $\mathfrak{F}$ has characteristic zero; for any $j \geq 0$, the $H$-module $H^j(I,\mathbf{X}^{K_m})$ is finitely generated and there is an exact sequence of $H$-modules
\begin{equation*}
  0 \longrightarrow \Ext^1_H(H^j(I,\mathbf{X}^{K_m}),H) \longrightarrow \Tor^{k[[I]]}_{j-1}(k,(\mathbf{X}^{K_m})^*) \longrightarrow H^{j-1}(I,\mathbf{X}^{K_m})^* \longrightarrow 0 \ ;
\end{equation*}
  \item[b)] suppose that $\mathfrak{F}$ has characteristic $p$; the $H$-module $H^j(I,\mathbf{X}^{K_m})$ is finitely generated for any $0 \leq j \leq 2$,  and there is an exact sequence of $H$-modules
\begin{equation*}
  0 \longrightarrow \Ext^1_H(H^1(I,\mathbf{X}^{K_m}),H) \longrightarrow ((\mathbf{X}^{K_m})^*)_I \longrightarrow H \longrightarrow 0 \ .
\end{equation*}
\end{itemize}
\end{proposition}
\begin{proof}
We fix a resolution
\begin{equation*}
  \ldots \xrightarrow{\;\partial_2\;} P_1 \xrightarrow{\;\partial_1\;} P_0 \longrightarrow k \longrightarrow 0
\end{equation*}
of the trivial $k[[I]]$-module $k$ by projective $k[[I]]$-modules $P_i$. According to Lemma \ref{resolution} we may assume that $P_i$ is finitely generated for any $i \geq 0$ in case a), resp.\ for any $0 \leq i \leq 2$ in case b). For any left $k[[I]]$-module $M$ we have the natural homomorphism of $H$-modules
\begin{align*}
    \Hom_H(\mathbf{X}^{K_m},H) \otimes_{k[[I]]} M & \longrightarrow \Hom_H(\Hom_{k[[I]]}(M, \mathbf{X}^{K_m}),H) \\
    \alpha \otimes x & \longmapsto [\beta \mapsto \alpha (\beta (x))] \ .
\end{align*}
It is an isomorphism if $M$ is finitely generated projective. Hence we obtain
\begin{equation*}
  \Tor^{k[[I]]}_j(k,(\mathbf{X}^{K_m})^*) = h_j((\mathbf{X}^{K_m})^* \otimes_{k[[I]]} P_\bullet) = h_j(\Hom_{k[[I]]}(P_\bullet, \mathbf{X}^{K_m})^*)
\end{equation*}
for any $j \geq 0$, resp.\ for $j = 0,1$. On the other hand the complex $\Hom_{k[[I]]}(P_\bullet, \mathbf{X}^{K_m})$ computes the groups $\Ext^{\bullet}_{k[[I]])}(k,\mathbf{X}^{K_m})$. Let
\begin{align*}
  B_j & := \im \big(\Hom_{k[[I]]}(P_{j-1}, \mathbf{X}^{K_m}) \xrightarrow{\; d_j \;} \Hom_{k[[I]]}(P_j, \mathbf{X}^{K_m}) \big),   \\
  C_j & := \ker \big(\Hom_{k[[I]]}(P_j, \mathbf{X}^{K_m}) \xrightarrow{\; d_{j+1} \;} \Hom_{k[[I]]}(P_{j+1}, \mathbf{X}^{K_m}) \big), \text{and}  \\
  d_j & := \Hom_{k[[I]]}(\partial_j, \mathbf{X}^{K_m}).
\end{align*}
Using Lemma \ref{resolution} we see that we have short exact sequences of $H$-modules
\begin{equation*}
  0 \longrightarrow B_j \longrightarrow C_j \longrightarrow H^j(I,\mathbf{X}^{K_m}) \longrightarrow 0
\end{equation*}
for any $j \geq 0$, resp.\ for any $0 \leq j \leq 2$. By Cor.\ \ref{proj-res} and Remark \ref{submod-acyc} all $H$-modules
\begin{equation*}
  B_j \subseteq C_j \subseteq \Hom_{k[[I]]}(P_j, \mathbf{X}^{K_m})
\end{equation*}
are finitely generated and $^\ast$-acyclic. It follows that the $H$-module $H^j(I,\mathbf{X}^{K_m})$ is finitely generated and that we have the exact sequence of $H$-modules
\begin{equation}\label{f:cohom1}
  0 \longrightarrow H^j(I,\mathbf{X}^{K_m})^* \longrightarrow C_j^* \longrightarrow B_j^* \longrightarrow \Ext^1_H(H^j(I,\mathbf{X}^{K_m}),H) \longrightarrow 0 \ .
\end{equation}
Moreover, the short exact sequences
\begin{equation*}
  0 \longrightarrow C_{j-1} \longrightarrow \Hom_{k[[I]]}(P_{j-1}, \mathbf{X}^{K_m}) \longrightarrow B_j \longrightarrow 0
\end{equation*}
dualize into short exact sequences
\begin{equation}\label{f:cohom2}
  0 \longrightarrow B_j^* \longrightarrow \Hom_{k[[I]]}(P_{j-1}, \mathbf{X}^{K_m})^* \longrightarrow C_{j-1}^* \longrightarrow 0 \ .
\end{equation}
Combining \eqref{f:cohom1} for $j-1$ and \eqref{f:cohom2} for $j$ gives the commutative exact diagram
\begin{equation*}
  \xymatrix{
    &  & & 0 \ar[d]  &   \\
    0  \ar[r] & B_j^* \ar[d]^{=} \ar[r] & \ker(d_{j-1}^*) \ar[d]^{\subseteq} \ar[r] & H^{j-1}(I,\mathbf{X}^{K_m})^* \ar[d] \ar[r] & 0 \\
    0 \ar[r] & B_j^* \ar[r] & \Hom_{k[[I]]}(P_{j-1}, \mathbf{X}^{K_m})^* \ar[d]^{d_{j-1}^*} \ar[r] & C_{j-1}^* \ar[d] \ar[r] & 0  \\
    & & \Hom_{k[[I]]}(P_{j-2}, \mathbf{X}^{K_m})^* & B_{j-1}^*. \ar@{_{(}->}[l] &  }
\end{equation*}
So far this reasoning holds for any $j \geq 0$, resp.\ for any $0 \leq j \leq 2$. In addition, by combining \eqref{f:cohom1} for $j$ and \eqref{f:cohom2} for $j$ and $j+1$, we have the commutative diagram
\begin{equation*}
   \xymatrix{
    & C_j^* \ar[d] & \Hom_{k[[I]]}(P_j, \mathbf{X}^{K_m})^* \ar[d]^{d_j^*} \ar@{->>}[l] & &   \\
    0 \ar[r] & B_j^* \ar[d] \ar[r] & \Hom_{k[[I]]}(P_{j-1}, \mathbf{X}^{K_m})^* \ar[r] & C_{j-1}^* \ar[r] & 0  \\
    & \Ext^1_H(H^j(I,\mathbf{X}^{K_m}),H) \ar[d] & & & \\
    & 0 & & &
     }
\end{equation*}
for any $j \geq 0$, resp.\ for $j = 0,1$. This second diagram leads to the short exact sequence
\begin{equation*}
  0 \longrightarrow \im(d_j^*) \longrightarrow B_j^* \longrightarrow \Ext^1_H(H^j(I,\mathbf{X}^{K_m}),H) \longrightarrow 0
\end{equation*}
which together with the top row of the first diagram imply the asserted exact sequences since $\Tor^{k[[I]]}_{j-1}(k,(\mathbf{X}^{K_m})^*) = \ker(d_{j-1}^*)/\im(d_j^*)$.
\end{proof}

\begin{corollary}\label{Zm-dual}
For any $m \geq 1$ we have
\begin{equation*}
  Z_m = \Ext^1_H(H^1(I,\mathbf{X}^{K_m}),H) = \Ext^1_H(H^1(I,\mathbf{X}^{K_m})^1,H)
\end{equation*}
as $H$-modules.
\end{corollary}
\begin{proof}
The left equality is Prop.\ \ref{coh-formula} for $j=1$. The right equality follows from Prop.\ \ref{coh-formula} and Lemma \ref{mod-AG}.iv.
\end{proof}

Note that since $\mathbf G$ is semisimple of rank $1$ we know, by \cite{OS} Thm.\ 0.2 that
for any left (resp. right) $H$-module of finite length $\mathfrak m$, there is a natural isomorphism of right (resp. left) $H$-modules $$\Ext_H^1(\mathfrak{m},H) \cong \Hom_k(\upiota^\ast\mathfrak{m},k)$$  where  $\upiota^\ast\mathfrak{m}$ denotes the $H$-module $\mathfrak m$  with the action on $H$ twisted by a certain  involutive automophism $\upiota$ of $H$. The explicit formula for $\upiota$ is recalled in the current article in \S\ref{involution} in the case when $\mathbf{G} = \mathbf{SL_2}$. The corollary above can  therefore be formulated as follows: \begin{equation}\label{f:dual1}
  Z_m \cong  \Hom_k(\upiota ^*H^1(I,\mathbf{X}^{K_m})^1,k)
\end{equation}
or, dually:
\begin{equation}\label{f:dual2}
\Ext^1_H(Z_m,H)\cong  \Hom_k(\upiota ^*Z_m,k)\cong   H^1(I,\mathbf{X}^{K_m})^1.
\end{equation}

\section{\label{sec:3}The case $\mathbf{SL_2}(\mathfrak{F})$}

We keep the notations of the previous section as well as the assumption that $k$ has characteristic $p$, but we restrict to the case $G = \mathbf{SL_2}(\mathfrak{F})$.

\subsection{\label{rootdatum}Root datum}

To fix ideas we consider $K = \mathbf{SL_2}(\mathfrak{O})$ and $I = \left(
\begin{smallmatrix}
1+\mathfrak{M} & \mathfrak{O} \\ \mathfrak{M} & 1+\mathfrak{M}
\end{smallmatrix}
\right)$ (by abuse of notation all matrices are understood to have determinant one). We let $T \subseteq G$ be the torus of diagonal matrices, $T^0$ its maximal compact subgroup, $T^1$ its maximal pro-$p$ subgroup, and $N(T)$ the normalizer of $T$ in $G$. We choose the positive root with respect to $T$ to be $\alpha( \left(
\begin{smallmatrix}
t & 0 \\ 0 & t^{-1}
\end{smallmatrix}
\right) ) := t^2$, which corresponds to the Borel subgroup of upper triangular matrices. The affine Weyl group $W$ sits in the short exact sequence
\begin{equation*}
  0 \longrightarrow \Omega := T^0 / T^1 \longrightarrow \widetilde{W} := N(T)/T^1 \longrightarrow W := N(T)/T^0 \longrightarrow 0 \ .
\end{equation*}
Let $s_0 := s_\alpha := \left(
\begin{smallmatrix}
0 & 1 \\ -1 & 0
\end{smallmatrix}
\right)$, $s_1 := \left(
\begin{smallmatrix}
0 & -\pi^{-1} \\ \pi & 0
\end{smallmatrix}
\right)$, and $\theta := \left(
\begin{smallmatrix}
\pi & 0 \\ 0 & \pi^{-1}
\end{smallmatrix}
\right)$, such that $s_0 s_1 = \theta$. The images of $s_0$ and $s_1$ in $W$ are the two reflections corresponding to the two vertices of $C$ which generate $W$, i.e., we have $W = \langle s_0,s_1 \rangle = \theta^{\mathbb{Z}} \dot\cup s_0 \theta^{\mathbb{Z}}$ (by abuse of notation we do not distinguish in the notation between a matrix and its image in $W$ or $\widetilde{W}$). We let $\ell$ denote the length function on $W$ corresponding to these generators as well as its pull-back to $\widetilde{W}$. One has
\begin{equation*}
  \ell(\theta^i) = |2i| \qquad\text{and}\qquad \ell(s_0\theta^i) = |1 - 2i|. 
 \end{equation*}

\begin{remark}\label{remark:N}
Consider $\mathbf{SL_2}(\mathfrak F)$ as a subgroup of $\mathbf{GL_2}(\mathfrak F)$. Then
the matrix $N:=
\left(\begin{smallmatrix}
       0 & 1 \\
       \pi & 0
\end{smallmatrix}\right)$ normalizes $I$; furthermore, $s_1=N  s_0  N^{-1}.$
\end{remark}

\subsection{The pro-$p$ Iwahori-Hecke algebra $H$}\label{sec:relations}
\subsubsection{\label{gene-rel}Generators and relations}
The characteristic functions $\tau_w := \mathrm{char}_{IwI}$ of the double cosets $IwI$ form a $k$-basis of $H$.
Let $e_1 := - \sum_{\omega \in \Omega} \tau_\omega$. The relations in $H$ are
\begin{equation}\label{f:braid}
\tau_v\tau _w=\tau_{vw} \qquad\textrm{whenever}\ \ell(v)+\ell(w)=\ell(vw)
\end{equation}
and
\begin{equation}
\label{f:quad} \tau_{s_i}^2= -e_1\tau_{s_i} \qquad\textrm{for $i=0,1$.}
\end{equation}The elements $\tau_\omega \tau_{s_i}$, for $\omega \in \Omega$ and $i = 0,1$, generate $H$ as a $k$-algebra.  Note that the $k$-algebra $k[\Omega]$ identifies naturally with a subalgebra of  $H$ via $\omega\mapsto \tau_\omega$.

\begin{proof}[Proof of \eqref{f:braid} and \eqref{f:quad}]
Needing later on some of its details, we recall the proof of these relations. The braid relation \eqref{f:braid} is clear when $\ell(v)=0$. So, by induction, it is enough to verify it when $v = s \in \{s_0, s_1\}$ and $\ell(s)+\ell(w)=\ell(sw)$. By \cite{Vig} Lemma 5(i), the support of $\tau_s\tau_w$ is $I sw I$ and, by definition of the convolution product in $H$,  its value at $sw$ is $|I s I \cap swI w^{-1}I/I|\cdot 1_k$. By \cite{IM} \S 3.1, we have $I' s I' \cap swI' w^{-1} I' = sI'$ for the Iwahori subgroup $I \subseteq I' \subseteq K$. Therefore, for $x\in I$ such that $xs\in sw I w^{-1} I \subset sw I' w^{-1} I'$ we have $x\in sI' s^{-1} \cap I$. But the latter is a pro-$p$ subgroup of $s I' s^{-1}$ and therefore $x \in sIs^{-1} \cap I$. We proved that $I s I \cap swI w^{-1}I = sI$ and hence that $\tau_s\tau_w = \tau_{sw}$. Now we turn to the quadratic relations \eqref{f:quad}.

For any $z \in \mathfrak{O}/\mathfrak{M}$ let $[z] \in \mathfrak{O}$ denote its Teichm\"uller representative (cf.\ \cite{Se0} II.4 Prop.\ 8). We parameterize $\Omega$ by the isomorphism
\begin{align*}
  (\mathfrak{O}/\mathfrak{M})^\times & \xrightarrow{\;\cong\;} \Omega \\
  z & \longmapsto \omega_z := \left(
\begin{smallmatrix}
  [z]^{-1} & 0 \\
  0 & [z]
\end{smallmatrix}
\right) T^1 \ ,
\end{align*}
and we put
\begin{equation}\label{f:u}
  u^+_\omega := \left(
\begin{smallmatrix}
  1 & [z] \\
  0 & 1
\end{smallmatrix}
\right) \ , \
u^-_\omega := \left(
\begin{smallmatrix}
  1 & 0 \\
  \pi [z]^{-1} & 1
\end{smallmatrix}
\right)
  \qquad\text{if $\omega = \omega_z$}.
\end{equation}
We compute
\begin{align*}
  s_0Is_0I & = \left(
  \begin{smallmatrix}
  1 & 0 \\
  \mathfrak{O} & 1
  \end{smallmatrix}
  \right)
  \left(
  \begin{smallmatrix}
  -1 & 0 \\
  0 & -1
  \end{smallmatrix}
  \right) I = \dot\bigcup_{z \in \mathfrak{O}/\mathfrak{M}} \left(
  \begin{smallmatrix}
  1 & 0 \\
  [z] & 1
  \end{smallmatrix}
  \right)
  \left(
  \begin{smallmatrix}
  -1 & 0 \\
  0 & -1
  \end{smallmatrix}
  \right) I \\
  & = I
  \left(
  \begin{smallmatrix}
  -1 & 0 \\
  0 & -1
  \end{smallmatrix}
  \right) \, \dot\cup \; \dot\bigcup_{z \in (\mathfrak{O}/\mathfrak{M})^\times} \left(
  \begin{smallmatrix}
  1 & 0 \\
  [z]^{-1} & 1
  \end{smallmatrix}
  \right)
  \left(
  \begin{smallmatrix}
  -1 & 0 \\
  0 & -1
  \end{smallmatrix}
  \right) I
   \\
   & = I
   \left(
  \begin{smallmatrix}
  -1 & 0 \\
  0 & -1
  \end{smallmatrix}
  \right)
   \, \dot\cup \; \dot\bigcup_{z \in (\mathfrak{O}/\mathfrak{M})^\times} \left(
  \begin{smallmatrix}
  1 & [z] \\
  0 & 1
  \end{smallmatrix}
  \right)
  \left(\begin{smallmatrix}
  [z] & 0 \\
  0 & [z]^{-1}
  \end{smallmatrix}
  \right) s_0
   I \\
   & = I s_0^2 \; \dot\cup \; \dot\bigcup_{\omega \in \Omega} u^+_\omega s_0 \omega I
\end{align*}
and hence obtain
\begin{equation}\label{f:supports0}
  s_0Is_0I \subseteq Is_0^2 I \; \dot\cup \; \dot\bigcup_{\omega \in \Omega} I s_0 \omega I \ .
\end{equation}
The support of $\tau_{s_0}^2$ is contained in $I s_0 Is_0I$. Its value at $h \in G$ is equal to $|I s_0 I \cap hIs_0^{-1}I| \cdot 1_k$. For $h = s_0^2$ this value is equal to $|Is_0I/I| \cdot 1_k = [I : s_0 I s_0^{-1}] \cdot 1_k = q \cdot 1_k = 0$.  The calculation above shows that we have $I s_0 I \cap s_0 \omega Is_0^{-1}I = u^+_{s_0^2 \omega} s_0 I$ and hence that the value for $h = s_0 \omega$ is equal to $1_k$. This proves \eqref{f:quad} for $s = s_0$. After  conjugating by the matrix $N$ of Remark \ref{remark:N}, we get from the previous identities that
\begin{equation*}
s_1I s_1I= I s_1^2 \; \dot\cup \; \dot\bigcup_{\omega \in \Omega} u^-_\omega s_1 \omega I
\end{equation*}
and
\begin{equation}\label{f:supports1}
  s_1Is_1I \subseteq Is_1^2 I \; \dot\cup \; \dot\bigcup_{\omega \in \Omega} I s_1 \omega I
\end{equation}
which yields the quadratic relation for $s_1$.
\end{proof}

\subsubsection{The central element $\zeta$}\label{sec:zeta}

Consider the element  $\zeta := (\tau_{s_0} + e_1)(\tau_{s_1} + e_1) + \tau_{s_1}\tau_{s_0}$. It is central in $H$.
\begin{remark}\label{explicit-identities}
Note that $\zeta= (\tau_{s_1} + e_1)(\tau_{s_0} + e_1) + \tau_{s_0}\tau_{s_1}$. The following identities will be used in the proof of Lemmas \ref{zeta-basis} and \ref{modulo}. For any $i \geq 0$ we have:
\begin{itemize}
\item[i.]   $\zeta^i= ((\tau_{s_0} + e_1)(\tau_{s_1} + e_1))^i + (\tau_{s_1}\tau_{s_0}) ^i =  ((\tau_{s_1} + e_1)(\tau_{s_0} + e_1))^i + (\tau_{s_0}\tau_{s_1})^i$ if $i \geq 1$.
  \item[ii.] $\zeta^i \tau_{s_0} = \tau_{s_0} (\tau_{s_1} \tau_{s_0})^i$ and $\zeta^i \tau_{s_1} = \tau_{s_1} (\tau_{s_0} \tau_{s_1})^i$;
  \item[iii.] $\zeta^{i+1} = \zeta^i (\tau_{s_0} + 1)e_1 + (\zeta^i\tau_{s_1})(\tau_{s_0} + e_1) + \tau_{s_0}(\zeta^i \tau_{s_1}) = \zeta^i (\tau_{s_1} + 1)e_1 + (\zeta^i\tau_{s_0})(\tau_{s_1} + e_1) + \tau_{s_1} (\zeta^i \tau_{s_0})$.
\end{itemize}
\end{remark}
\begin{proof}
i. and ii. By the quadratic relations \eqref{f:quad} we have $\tau_{s_0}(\tau_{s_0} + e_1)(\tau_{s_1} + e_1) = (\tau_{s_0} + e_1)(\tau_{s_1} + e_1)\tau_{s_1} = 0$. This implies, for $i\geq 1$, that $\zeta^i= ((\tau_{s_0} + e_1)(\tau_{s_1} + e_1))^i + (\tau_{s_1}\tau_{s_0}) ^i$ and $\zeta^i  \tau_{s_1}= (\tau_{s_1}\tau_{s_0}) ^i\tau_{s_1}=\tau_{s_1}(\tau_{s_0}\tau_{s_1})^i$. Likewise, since  $\zeta= (\tau_{s_1} + e_1)(\tau_{s_0} + e_1) + \tau_{s_0}\tau_{s_1}$, we obtain $\zeta^i= ((\tau_{s_1} + e_1)(\tau_{s_0} + e_1))^i + (\tau_{s_0}\tau_{s_1})^i$ and $\zeta^i \tau_{s_0}=(\tau_{s_0}\tau_{s_1})^i\tau_{s_0}=\tau_{s_0}(\tau_{s_1}\tau_{s_0})^i$
for $i\geq 1$.

iii. We compute $\zeta^{i+1} = \zeta^i \zeta = \zeta^i(\tau_{s_0}\tau_{s_1} + \tau_{s_1}\tau_{s_0} + (\tau_{s_0} + \tau_{s_1} + 1)e_1) = \zeta^i\tau_{s_1}(\tau_{s_0} + e_1) + \zeta^i \tau_{s_0}\tau_{s_1} + \zeta^i (\tau_{s_0} + 1)e_1 = \zeta^i\tau_{s_1}(\tau_{s_0} + e_1) + \tau_{s_0} \zeta^i \tau_{s_1} + \zeta^i (\tau_{s_0} + 1)e_1$ and correspondingly $\zeta^{i+1} = \zeta^i\tau_{s_0}(\tau_{s_1} + e_1) + \zeta^i \tau_{s_1}\tau_{s_0} + \zeta^i (\tau_{s_1} + 1)e_1 = \zeta^i\tau_{s_0}(\tau_{s_1} + e_1) + \tau_{s_1} \zeta^i \tau_{s_0} + \zeta^i (\tau_{s_1} + 1)e_1$.
\end{proof}

By Prop.\ \ref{H-free}.i, $H$ is a free left $H_{x_0}$-module  as well as a free left $H_{x_1}$-module. Explicitly:

\begin{lemma}\label{zeta-basis}
 Let $\epsilon = 0$ or $1$; as a left $H_{x_\epsilon}$-module, $H$ is free with basis $\{\zeta^i\tau_{s_{1-\epsilon}}, \zeta^i, \: i\geq 0\}$.
\end{lemma}
\begin{proof}
The other case being completely analogous we only give the argument for $\epsilon = 0$. An explicit basis is given in \cite{OS} Prop.\ 4.21 and its proof. It is the set $ \{\tau_{d}, d\in \D\}$ where $\D=\{(s_1 s_0)^i, (s_1s_0)^is_1, i\geq 0 \}$. First recall that $\zeta^i\tau_{s_1}=\tau_{(s_1s_0)^i s_1}$ for $i\geq 0$ by Remark \ref{explicit-identities}.ii. Then the lemma follows from the following fact: For $i\geq 1$,  the element $\zeta ^i$ decomposes into the sum of $\tau_{(s_1s_0)^i}$ and of an element in $\oplus_{d\in \D, \ell(d)< 2i }\:\: H_{x_0} \tau_{d}$.   We prove this by induction. For $i=1$, compute $\zeta\in \tau_{s_1s_0}+ H_{x_0}+ H_{x_0} \tau_{s_1}$.
Suppose the fact holds true for some $i \geq 1$. Then $\zeta^{i+1}\in \tau_{(s_1s_0)^i} \zeta+\sum_{\ell(d)< 2i } H_{x_0} \tau_d \zeta=\tau_{(s_1s_0)^{i+1}}+\sum_{\ell(d)< 2i} H_{x_0} \tau_d \zeta$. For $d$ such that $\ell(d)<2i$, we have either $d= (s_1s_0)^j$ with $j< i$ in which case $\tau_d\zeta= \tau_{(s_1s_0)^{j+1}}$ and $\ell((s_1s_0)^{j+1}) = 2(j+1)< 2(i+1)$, or  $d= (s_1s_0)^j s_1$ with $2j+1< 2i$ in which case $\tau_d\zeta= \tau_{(s_1s_0)^{j+1}}\tau_{s_1}$ and $\ell((s_1s_0)^{j+1}s_1) = 2(j+1)+1< 2(k+1)$. So the fact is true for $i+1$.
\end{proof}

The lemma implies in particular that the subalgebra $k[\zeta]$ of $H$ generated by $\zeta$ is the algebra of polynomials in the variable $\zeta$. In general, the center of $H$ contains strictly the subalgebra $k[\zeta]$. We describe the full center of $H$ in Remark \ref{rema-center}.
The lemma also implies that $\zeta$ is not a zero divisor in $H$ and that the $k$-algebra $H/H\zeta$ is finite dimensional.

\begin{corollary}\label{isobimo}
Let $\epsilon = 0$ or $1$; the morphism of  $(H_{x_\epsilon},k[\zeta])$-bimodules
\begin{equation*}
\begin{array}{ccccc}
H_{x_\epsilon} \otimes_k k[\zeta] & \oplus & H_{x_\epsilon} \otimes_k k[\zeta] & \longrightarrow & H \cr 1 \otimes 1 &&& \longmapsto & 1\cr & & 1\otimes 1 & \longmapsto & \tau_{s_{1-\epsilon}}
\end{array}
\end{equation*}
is an isomorphism.
\end{corollary}

\subsubsection{Involution}\label{involution}

Using \eqref{f:braid} and \eqref{f:quad}, one checks that here is a unique well-defined involutive automorphism of  the $k$-algebra $H$ defined by
$\upiota(\tau_\omega) = \tau_\omega$ and $\upiota(\tau_{s_i}) = -e_1 - \tau_{s_i}$ for $i\in\{0,1\}$. This is the involution denoted by $\iota_C$ in \cite{OS}.

\subsubsection{Idempotents}\label{sec:idempo}

The element  $e_1$ is a central idempotent in $H$. More generally, to any $k$-character $\lambda\in \Omega \rightarrow k^\times$ of $\Omega$, we  associate the following idempotent in $H$:
\begin{equation}\label{defel}
\ e_\lambda := - \sum_{\omega \in \Omega} \lambda(\omega^{-1}) \tau_\omega \ .
\end{equation}
Note that $e_\lambda\tau_\omega = \tau_\omega e_\lambda=\lambda(\omega)e_\lambda$ for any $\omega\in \Omega$. Suppose  for a moment that  $\mathbb{F}_q \subseteq k$.
Then all simple modules of $k[\Omega]$ are one dimensional.  The set $\widehat{\Omega}$ of all $k$-characters of $\Omega$ has cardinality $q-1$ which is prime to $p$. This implies that the family  $\{e_\lambda\}_\lambda\in \widehat\Omega$ is a family of orthogonal idempotents with sum equal to $1$. It  gives the following ring decomposition of   $k[\Omega] $:
\begin{equation}\label{f:idempotentsOmega}
  k[\Omega] = \prod_{\lambda \in \widehat{\Omega}} k e_\lambda
\end{equation}
 Let $\Gamma := \{ \{\lambda,\lambda^{-1}\} : \lambda \in \widehat{\Omega} \}$ denote the set of $s_0$-orbits in $\widehat{\Omega}$. To
 $\gamma\in \Gamma$ we attach the element $e_\gamma:=e_\lambda+e_{\lambda^{-1}}$ (resp.  $e_\gamma:=e_\lambda$) if $\gamma=\{\lambda,\lambda^{-1}\}$ with $\lambda\neq\lambda^{-1}$ (resp. $\gamma=\{\lambda\}$). Using the braid relations, one sees that $e_\gamma$ is a central idempotent in $H$ and we have the ring  decomposition
\begin{equation} \label{f:idempotents}
  H = \prod_{\gamma \in \Gamma} H e_\gamma.
\end{equation}

\begin{remark}\label{rema-center}
Consider the $k$-linear map $E:k[\mathbb Z\times \Omega]\longrightarrow H$  defined by
\begin{equation*}
  (i,\omega) \mapsto
  \begin{cases}
  \tau_\omega (\tau_{s_0}\tau_{s_1})^i & \textrm{if $i \geq 0$}, \\
    \tau_\omega((\tau_{s_1}+e_1)(\tau_{s_0}+e_1))^{ -i} & \textrm{if $i < 0$}.
  \end{cases}
\end{equation*}
 $$$$
By \cite{Vig} Thm. 2, it is injective and its image is a commutative subalgebra $A$ of $H$. The multiplication in $A$ is given by the following rule, for $i,j \in \mathbb Z$ and $\omega, \omega' \in \Omega$:
\begin{equation}\label{product}
   E(i,\omega)E(j, \omega' )=
\begin{cases}
   0 & \textrm{if $ij < 0$}, \\
   E(i+j, \omega \omega') & \textrm{if $ij \geq 0$}.
\end{cases}
\end{equation}
This shows that the $k$-linear automorphism of $A$ sending $E(i, \omega)$ to $E(-i, \omega^{-1})$ respects the product. By \cite{Vig} Thm. 4, the center $Z(H)$ of $H$ is the subring of those elements in $A$ which are invariant under this automorphism. The following family is a system of generators (but not a basis) for the $k$-vector space $Z(H)$:
\begin{equation}\label{elements-center}
  E(i, \omega) + E(-i, \omega^{-1})  \quad\textrm{for $\omega \in \Omega$ and $i \geq 0$ (and $\tau_1$ if $p=2$)}.
\end{equation}

In the case when $\mathbb F_q \subseteq k$ we describe the ring structure of $Z(H)$. It is the direct product of the centers $Z(e_\gamma H) = e_\gamma Z(H)$ of $e_\gamma H$, where $\gamma$ ranges over $\Gamma$. For any $\lambda \in \widehat\Omega$ with orbit $\gamma$, we put
\begin{equation*}
  X_\lambda :=  e_\lambda E(1, 1)+e_{\lambda^{-1}}E(-1, 1)
\end{equation*}
in $e_\gamma H$. Using \eqref{product} we see that
\begin{equation*}
  X_\lambda^i = e_\lambda E(i, 1) + e_{\lambda^{-1}}E(-i, 1)  \qquad\text{for any $i\geq 1$}
\end{equation*}
and that
\begin{equation*}
  e_\gamma(  E(i, \omega)+E(-i, \omega^{-1})) =
  \begin{cases}
  \lambda(\omega) \zeta ^i e_\gamma & \textrm{if $i \geq 1$ and $\gamma=\{\lambda\}$}, \\
   2 \lambda(\omega)  e_\gamma & \textrm{if $i=0$ and $\gamma = \{\lambda\}$}, \\
   \lambda(\omega) X_{\lambda}^i+ \lambda(\omega^{-1}) X_{\lambda^{-1}}^i & \textrm{if $i\geq 1$ and  $\gamma = \{\lambda \neq \lambda^{-1}\}$}, \\
   (\lambda(\omega) + \lambda(\omega^{-1})) e_\gamma & \textrm{if $i=0$ and $\gamma = \{\lambda \neq \lambda^{-1}\}$},
   \end{cases}
\end{equation*}
for any $\omega \in \Omega$.

Suppose that $\gamma = \{\lambda\}$ has cardinality $1$. Then $X_\lambda = e_\gamma \zeta$ and, by the above identity, $Z(e_\gamma H)$ is the $k$-algebra with unit $e_\gamma$ generated by $e_\gamma \zeta$. As a consequence of Lemma \ref{zeta-basis}, it is the algebra of polynomials in the variable $e_\gamma \zeta$. Now suppose that $\gamma = \{\lambda, \lambda^{-1}\}$ has cardinality $2$. Then
\begin{align*}
   X_\lambda & = - \sum_{\omega \in \Omega} \lambda(\omega^{-1}) \tau_\omega E(1,1) - \sum_{\omega \in \Omega} \lambda^{-1}(\omega^{-1}) \tau_\omega E(-1,1) \\
             & = - \sum_{\omega \in \Omega} \lambda(\omega^{-1}) \tau_\omega E(1,1) - \sum_{\omega \in \Omega} \lambda(\omega^{-1}) \tau_{\omega^{-1}} E(-1,1)  \\
             & = - \sum_{\omega \in \Omega} \lambda(\omega^{-1}) ( E(1,\omega) +   E(-1,\omega^{-1})) \in Z(H) \cap e_\gamma H = Z(e_\gamma H) \ .
\end{align*}
This shows that $Z(e_\gamma H)$ is the $k$-algebra with unit $e_\gamma$ generated by  $X_{\lambda}$ and $X_{\lambda^{-1}}$. We claim that it is isomorphic to the quotient algebra $k[X_\lambda,  X_{\lambda^{-1}}]/(X_\lambda X_{\lambda^{-1}})$. That $X_\lambda X_{\lambda^{-1}} = 0$ in $e_\gamma H$ is clear.  It is easy to see that it is the only relation because the map $E$ is injective (for fixed $i$, the vector space with basis $(E(i, \omega))_{\omega \in \Omega}$ is the same as  the vector space generated by the family  $(e_\lambda E(i,1))_{\lambda\in \widehat\Omega}$, therefore the family $(e_\lambda E(i, 1))_{\lambda \in \widehat\Omega, i\in \mathbb{Z}}$ is linearly independent in $H$).
\end{remark}

\begin{remark}\label{modulesS1}
Suppose that  $\mathbb{F}_q \subseteq k$. Let $\mathfrak{s}_1$ denote the two-sided ideal generated by $\tau_{s_1}$. It easily follows from the braid and quadratic relations that $\mathfrak{s}_1$ is $k$-linearly generated by all $\tau_w$ such that $s_1$ occurs in a reduced decomposition of $w$. It follows that
\begin{equation*}
  H/\mathfrak{s}_1 = k[\Omega] \oplus k[\Omega] \tau_{s_0} \quad\text{with}\ \tau_{s_0} \tau_\omega = \tau_{\omega^{-1}} \tau_{s_0} \ \text{and}\ \tau_{s_0}^2 = \tau_{s_0} (\sum_{\omega \in \Omega} \tau_\omega ) = (\sum_{\omega \in \Omega} \tau_\omega ) \tau_{s_0} \ .
\end{equation*} We have the ring decomposition
\begin{equation*}
  H/\mathfrak{s}_1 = \prod_{\gamma \in \Gamma} A_\gamma
\end{equation*}
where
\begin{equation*}
  A_\gamma :=
  \begin{cases}
  k e_\lambda \oplus k e_{\lambda^{-1}} \oplus k e_\lambda \tau_{s_0} \oplus k e_{\lambda^{-1}} \tau_{s_0} & \text{if $\gamma = \{\lambda,\lambda^{-1}\}$ with $\lambda \neq \lambda^{-1}$},\\
  k e_\lambda \oplus k e_\lambda \tau_{s_0} & \text{if $\gamma = \{\lambda\}$}.
  \end{cases}
\end{equation*} It is clear that
\begin{equation*}
  A_\gamma \cong
  \begin{cases}
  k[\tau_{s_0}]/\langle \tau_{s_0}^2 \rangle & \text{if $\gamma = \{\lambda\}$, but $\lambda \neq 1$}, \\
  k[\tau_{s_0}]/\langle \tau_{s_0}^2 + \tau_{s_0} \rangle = k \times k[\tau_{s_0}]/\langle \tau_{s_0} + 1 \rangle & \text{if $\gamma = \{1\}$}.
  \end{cases}
\end{equation*}
If $\gamma = \{\lambda,\lambda^{-1}\}$ with $\lambda \neq \lambda^{-1}$ then $A_\gamma \tau_{s_0}$ is an ideal of square zero in $A_\gamma$ and $A_\gamma / A_\gamma \tau_{s_0} \cong k \times k$. We see that in all cases the simple modules of $A_\gamma$ are one dimensional. It follows that the simple constituents of any finite length $H/\mathfrak{s}_1$-module are one dimensional $H$-modules.  This remark will be used in \S\ref{sec:smallq}.

\end{remark}

\subsubsection{Supersingular modules \label{sec:charH}}

A finite length $H$-module $M$ will be called supersingular if its base extension $\bar{k} \otimes_k M$ to an algebraic closure $\bar{k}$ of $k$ is a supersingular $\bar{k} \otimes_k H$-module in the sense of \cite{Oll3} Prop.-Def.\ 5.10.

\begin{lemma}\label{defi:supersing}
A finite length $H$-module is supersingular if and only if it is annihilated by a power of $\zeta$.
\end{lemma}
\begin{proof}
We may assume that $k$ is algebraically closed. By \cite{Oll3} Prop.-Def.\ 5.10,
a   finite length $H$-module $M$ is  supersingular if and only if it is
annihilated by a power of the ideal $\mathfrak{J}$ defined in loc.\ cit.\ \S5.2. By construction this ideal $\mathfrak{J}$ is generated by the element $z_\mu$ defined in \cite{Oll3} (2.3), where $\mu\in X_*(T)$ denotes the unique minimal dominant cocharacter. In the identification $T/T^0 \xrightarrow{\cong} X_*(T)$ sending $\theta$ to the antidominant cocharacter $t \mapsto
\left(\begin{smallmatrix}
t^{-1} & 0 \\
0 & t
\end{smallmatrix}\right)$, the element  $s_0(\theta)$ corresponds to $\mu$. By definition, and using \cite{Oll3} Remark 2.1, we have  $z_\mu = E^+_\mu + E^+_{s_0(\mu)} = \tau_{s_0(\theta)} + \tau^*_{-\theta} = \tau_{s_1s_0} + \tau^*_{s_1s_0}$ in the notation of \cite{Vig} Remark 7 and paragraph after Thm.\ 2.
On the other hand we compute $\zeta = \tau_{s_1}\tau_{s_0} +(\tau_{s_0} + e_1)(\tau_{s_1} + e_1) = \tau_{s_1s_0} + \tau^*_{s_0}\tau^*_{s_1} = \tau_{s_1s_0} + \tau^*_{s_1s_0}$. Hence $z_\mu = \zeta$.

\end{proof}

We fix a character $\chi$ of $H$, and we define the character $\lambda$ of $\Omega$ by $\lambda(\omega) := \chi(\tau_\omega)$. The quadratic relation \eqref{f:quad} for $\tau_{s_i}$ implies that $\chi(\tau_{s_i}) = 0$ or $-1$. Moreover, if $\chi(\tau_{s_i}) = -1$ for at least one $i$ then $\lambda$ is necessarily  the trivial character. In particular, we have precisely four characters with corresponding $\lambda = 1$:
\begin{align*}
  & \text{the trivial character $\chi_{triv}$ such that $\chi_{triv}(\tau_{s_0}) = \chi_{triv}(\tau_{s_1}) = 0$}, \\
  & \text{the sign character $\chi_{sign}$ such that $\chi_{sign}(\tau_{s_0}) = \chi_{sign}(\tau_{s_1}) = -1$}, \\
  & \text{the two characters $\chi_i$, for $i = 0,1$, such that $\chi_i(\tau_{s_i}) = -1$ and $\chi_i(\tau_{s_{1-i}}) = 0$}.
\end{align*}
One computes
\begin{equation*}
  \chi(\zeta) =
  \begin{cases}
  0 & \text{if $\chi \neq \chi_{triv}, \chi_{sign}$}, \\
  1 & \text{otherwise}.
  \end{cases}
\end{equation*}
Hence (cf.\ proof of Lemma.\ \ref{defi:supersing}) $\chi_{triv}$ and $\chi_{sign}$ are the only characters which are not supersingular.

\begin{lemma} \label{onedim}
If $\mathbb F_q \subseteq k$, then a simple supersingular module is one dimensional.
\end{lemma}
\begin{proof}
Since $\zeta$ is central in $H$ and by the previous lemma, a simple supersingular module is annihilated by $\zeta$.
Using the decomposition \eqref{f:idempotents}, we  check below  that such a module $M$ is one dimensional.

There is a unique $\gamma\in\Gamma$ such that $e_\gamma M \neq \{0\}$. Suppose that  $\gamma=\{1\}$. Then by \eqref{f:quad} the action of $\tau_{s_0}$ (resp. $\tau_{s_1}$) on $M$ is diagonalizable with eigenvalues $0$ and $-1$.   If $\tau_{s_0}$ acts by $-1$, then $\tau_{s_1}$ acts by zero because $\zeta$ acts by zero and since $M$ is simple, it is the character $\chi_{0}$. Otherwise, there is  $x\neq 0$ in the kernel of $\tau_{s_0}$.  If $\tau_{s_1}x=-x$, then $M$ is isomorphic to the character $\chi_{-1}$; otherwise $\tau_{s_1}x+x$ supports the character $\chi_{0}$.

Now suppose that  $\gamma$ contains a $\lambda\in \widehat\Omega$ with $\lambda\neq 1$. Then by \eqref{f:quad} the action of $\tau_{s_0}^2$ (resp. $\tau_{s_1}^2$) is trivial on $M$, so there is $x\neq 0$ in the kernel of $\tau_{s_0}$ and one can assume that $e_\lambda x=x$. If $\tau_{s_1}x=0$,  then $M$ is one dimensional spanned by $x$; otherwise $y:=\tau_{s_1}x$ spans $M$ since $\tau_{s_1}y=0$, $e_{\lambda^{-1}}y=y$ and
$\tau_{s_0}y= -\tau_{s_1}\tau_{s_0}x=0$ (because $\zeta$ acts by zero).
\end{proof}

It follows immediately:

\begin{proposition}\label{charH}
Suppose  $\mathbb{F}_q \subseteq k$. A simple $H$-module is supersingular if and only if it is a character of $H$ distinct from $\chi_{triv}$ and $\chi_{sign}$.

\end{proposition}
In the case when $k$ is algebraically closed,
Prop.\ \ref{charH} can  also be obtained as a direct corollary of \cite{Oll3} Prop.\ 5.11 and Lemma 5.12. Note that the proofs therein are actually valid under the weaker hypothesis  $\mathbb F_q\subseteq k$.

\subsubsection{Nonsupersingular modules}\label{sec:nonsupersing}

In view of Lemma \ref{defi:supersing}, we call a  finite length $H$-module   nonsupersingular if it is $\zeta$-torsionfree. Note that the zero module is both supersingular and nonsupersingular.

\begin{remark}\label{rema:nonsupersing}
For $M$ a finite length $H$-module, the following are equivalent:
\begin{itemize}
\item[i.]  $M$ is nonsupersingular,
\item[ii.]  none of the subquotients  of $M$ are supersingular,
\item[iii.] none of the submodules of $M$ are supersingular.
\end{itemize}
\end{remark}
\begin{proof}
A finite length $H$-module  $M$ is finite dimensional (this is independent of the field $k$, see for example \cite[Lemma 6.9]{OS}).
Therefore, $M$ is nonsupersingular if and only if  the operator $\zeta$ yields an  automorphism of the $H$-module $M$.
The equivalences follow, using the fact that the kernel of the operator $\zeta$ on a finite length  $H$-module is  supersingular, possibly trivial.
\end{proof}

\begin{remark}\label{P}
A finite length indecomposable $H$-module is $P(\zeta)$-torsion for some irreducible polynomial $P(X) \in k[X]$. Therefore,  a finite length $H$-module is the direct sum of a supersingular and of a nonsupersingular module. Compare with \cite{Oll3} \S 5.3.
\end{remark}

\subsection{An $I$-equivariant decomposition of $\mathbf{X}$}\label{sec:I-decomp-X}

As an $I$-representation $\mathbf{X}$ decomposes equivariantly into
\begin{equation*}
  \mathbf{X} = \oplus_{w \in \widetilde{W}} \; \mathbf{X}(w)
\end{equation*}
where
\begin{equation*}
  \mathbf{X}(w) := \{ f \in \ind_I^G(1) : \supp(f) \subseteq IwI \}.
\end{equation*}
Moreover, we have
\begin{equation*}
  \mathbf{X}(w) \cong \ind_{I \cap wIw^{-1}}^I (1) \ .
\end{equation*}
On the other hand, the characteristic functions $\tau_w := \mathrm{char}_{IwI}$ of the double cosets $IwI$ form a $k$-basis of $H$. We note that the elements $\tau_\omega \tau_{s_i}$, for $\omega \in \Omega$ and $i = 0,1$, generate $H$ as a $k$-algebra.

\begin{lemma}\label{length-inj}
For any $w, v \in \widetilde{W}$ we have:
\begin{itemize}
  \item[i.] If $\ell(wv) = \ell(w) + \ell(v)$ then the endomorphism $\tau_v$ of $\mathbf{X}$ restricts to an $I$-equivariant injective map $\tau_v : \mathbf{X}(w) \longrightarrow \mathbf{X}(wv)$;
  \item[ii.] if $s \in \{s_0,s_1\}$ is such that $\ell(ws) = \ell(w) - 1$ then $\mathbf{X}(w) \tau_s \subseteq \mathbf{X}(ws) \oplus \bigoplus_{\omega \in \Omega} \mathbf{X}(w\omega)$;
  \item[iii.] if $\ell(wv) < \ell(w) + \ell(v)$ then $\mathbf{X}(w) \tau_v \subseteq \sum_{\ell(u) < \ell(w) + \ell(v)} \mathbf{X}(u)$.
\end{itemize}
\end{lemma}
\begin{proof}
i. The assertion is straightforward for $v = \omega \in \Omega$. Hence we may assume, by induction, that $v = s \in \{s_0, s_1\}$. The braid relation \eqref{f:braid} implies
\begin{align*}
  \mathrm{char}_{IwsI} & = \tau_{ws} = \tau_w \tau_s = (\sum_{g \in I/I \cap wIw^{-1}} \mathrm{char}_{gwI})\tau_s \\
  & = \sum_{g \in I/I \cap wIw^{-1}} \mathrm{char}_{gwI} \tau_s = \sum_{g \in I/I \cap wIw^{-1}} \mathrm{char}_{gwIsI} \ ,
\end{align*}
which shows that the sets $gwIsI$, for $g \in I/I \cap wIw^{-1}$, are pairwise disjoint. It follows that the $k$-basis $\{\mathrm{char}_{gwI} : g \in I/I \cap wIw^{-1}\}$ of $\mathbf{X}(w)$ is mapped by $\tau_s$ to the $k$-linearly independent elements $\mathrm{char}_{gwIsI}$ in $\mathbf{X}(ws)$.

ii. Using \eqref{f:supports0} and \eqref{f:supports1} as well as the braid relation \eqref{f:braid} for $ws^{-1}$ and $s$ we see that
\begin{equation*}
  wIsI = ws^{-1}sIsI \subseteq ws^{-1}Is^2 I \; \dot\cup \; \dot\bigcup_{\omega \in \Omega} ws^{-1}I s \omega I \subseteq Iws I \; \dot\cup \; \dot\bigcup_{\omega \in \Omega} I w \omega I \ .
\end{equation*}
It follows that $\mathrm{char}_{gwI} \tau_s = \mathrm{char}_{gwIsI} \in \mathbf{X}(ws) \oplus \bigoplus_{\omega \in \Omega} \mathbf{X}(w\omega)$ for any $g \in I$.

iii. This follows by a standard induction from i. and ii.
\end{proof}

For any $i \geq 0$ we define the subgroups
\begin{equation}\label{defiIn}
  I_i^+ := \left(
\begin{smallmatrix}
1+\mathfrak{M} & \mathfrak{O} \\ \mathfrak{M}^{i+1} & 1+\mathfrak{M}
\end{smallmatrix}
\right) \qquad\text{and}\qquad I_i^- := \left(
\begin{smallmatrix}
1+\mathfrak{M} & \mathfrak{M}^i \\ \mathfrak{M} & 1+\mathfrak{M}
\end{smallmatrix}
\right)
\end{equation}
of $I$.

\begin{remark}\label{remark:N2}
In $\mathbf{GL_2}(\mathfrak{F})$,  we have $I^+_i =
N^{-1} I^-_i N$ for any $i \geq 0$ with $N$ as in Remark \ref{remark:N}.
\end{remark}

Let $w\in W$. With our choice of the positive root $\alpha$ in \S\ref{rootdatum}, the chamber $wC$ lies in the dominant Weyl chamber if and only if $\ell(s_0w)=1+\ell(w)$. One checks that
\begin{equation}\label{f:cap}
  I \cap wIw^{-1} =
  \begin{cases}
  I_{\ell(w)}^+ & \text{if $wC$ lies in the dominant Weyl chamber}, \\
  I_{\ell(w)}^- & \text{if $wC$ lies in the antidominant Weyl chamber}.
  \end{cases}
\end{equation}
Since $K_m = \left(
\begin{smallmatrix}
1+\mathfrak{M}^m & \mathfrak{M}^m \\ \mathfrak{M}^m & 1+\mathfrak{M}^m
\end{smallmatrix}
\right)$ we deduce that
\begin{equation*}
  K_m(I \cap wIw^{-1}) =
  \begin{cases}
  I_{m-1}^+ & \text{if $wC$ lies in the dominant Weyl chamber and $\ell(w)+1 \geq m$}, \\
  I_m^- & \text{if $wC$ lies in the antidominant Weyl chamber and $\ell(w) \geq m$}.
  \end{cases}
\end{equation*}
It follows that
\begin{equation*}
  \mathbf{X}(w)^{K_m} =
  \begin{cases}
  \ind_{I_{m-1}^+}^I(1) & \text{if $wC$ lies in the dominant Weyl chamber and $\ell(w)+1 \geq m$}, \\
  \ind_{I_m^-}^I(1) & \text{if $wC$ lies in the antidominant Weyl chamber and $\ell(w) \geq m$}.
  \end{cases}
\end{equation*}
For purposes of abbreviation we put $\sigma(w) := 1$, resp.\ $:= 0$, if $wC$ lies in the dominant, resp.\ antidominant, Weyl chamber.

\begin{corollary}\label{length-iso}
Fix an $m \geq 1$ and let $w, v \in \widetilde{W}$ such that $\ell(w) + \sigma(w) \geq m$ and $\ell(wv) = \ell(w) + \ell(v)$; if $m = 1$ and $w \in \Omega$ we assume that $vC$ lies in the dominant Weyl chamber; then the endomorphism $\tau_v$ of $\mathbf{X}$ restricts to an $I$-equivariant isomorphism $\tau_v : \mathbf{X}(w)^{K_m} \xrightarrow{\; \cong \;} \mathbf{X}(wv)^{K_m}$.
\end{corollary}
\begin{proof}
Suppose that $w \not\in \Omega$. We see from the diagram below that $wC$ lies in the dominant, resp.\ antidominant, Weyl chamber if and only if a reduced decomposition of $w$ starts with $s_1$, resp.\ with $s_0$. The assumption that $\ell(wv) = \ell(w) + \ell(v)$ then implies that $wv$ has the same property. It therefore follows from our assumptions that $\sigma(w) = \sigma(wv)$ for any $w$ and, in particular, that $\ell(wv) + \sigma(wv) \geq m$. We then conclude from the above computation that the source and the target of the map in question have the same finite $k$-dimension. Hence the assertion is a consequence of Lemma \ref{length-inj}.
\end{proof}

\subsection{Supersingularity and the $H$-modules $Z_m$}\label{sec:supersing}

In this section, we establish results about the $H$-module structure of $Z_m$ for $m\geq 1$ that are independent of the choice of the field $\mathfrak F$.

\subsubsection{\label{sec:defiDn}The $H$-modules $H^1(I,\mathbf{X}^{K_m})^1$ and $Z_m$ are supersingular}

The cohomology $H^j(I,.)$ commutes with arbitrary direct sums of discrete $I$-modules (cf.\ \cite{Se1} I.2.2 Prop.\ 8). Hence we have
\begin{equation*}
  H^j(I,\mathbf{X}^{K_m}) = \oplus_{w \in \widetilde{W}} \; H^j(I,\mathbf{X}(w)^{K_m}) \ .
\end{equation*}
The following observation will be crucial in this section and in later computations:

\begin{corollary}\label{length-support}
$H^j(I,\mathbf{X}^{K_m})^1 \subseteq \oplus_{w \in \widetilde{W}, \ell(w) + \sigma(w) < m} \; H^j(I,\mathbf{X}(w)^{K_m})$.
\end{corollary}
\begin{proof}
Since $H^j(X,\mathbf{X}^{K_m})^1$ is finite dimensional over $k$ by Prop.\ \ref{coh-formula}.i and Lemma \ref{mod-AG}.iii we certainly have $H^j(X,\mathbf{X}^{K_m})^1 \subseteq \oplus_{w \in S} \; H^j(I,\mathbf{X}(w)^{K_m})$ for some finite subset $S \subseteq \widetilde{W}$. Choose the unique minimal such $S$. It easily follows from Cor.\ \ref{length-iso} that any $w \in S$ must satisfy $\ell(w) + \sigma(w) < m$.
\end{proof}

It immediately follows that $H^j(I,\mathbf{X}^{K_1})^1 = 0$. In particular, together with Cor.\ \ref{Zm-dual} this confirms in the present setting that $Z_1 = 0$.

The action of $W$ on the apartment corresponding to $T$ is as follows:
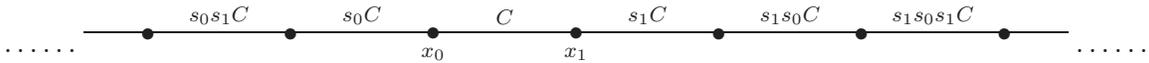
\begin{equation*}
    \xymatrix{
      \ldots\ldots \ar@{-}@<1ex>[r] &  *=0{\underset{}{\bullet}} \ar@{-}@<1ex>[rr]^{s_0 s_1 C} & & *=0{\underset{}{\bullet}} \ar@{-}@<1ex>[rr]^{s_0 C} & & *=0{\underset{x_0}{\bullet}} \ar@{-}@<1ex>[rr]^{C} & & *=0{\underset{x_1}{\bullet}} \ar@{-}@<1ex>[rr]^{s_1 C} & & *=0{\underset{}{\bullet}} \ar@{-}@<1ex>[rr]^{s_1 s_0 C} & & *=0{\underset{}{\bullet}} \ar@{-}@<1ex>[rr]^{s_1 s_0 s_1 C} & &
      *=0{\underset{}{\bullet}} \ar@{-}@<1ex>[r] & \ldots\ldots  }
\end{equation*}

In the following we fix $j \geq 0$ and $m \geq 1$ and introduce the abbreviations
\begin{equation*}
  D_n := \oplus_{\ell(w) + \sigma(w) = n} H^j(I,\mathbf{X}(w)^{K_m}) \quad\text{and}\quad C_n := D_1 \oplus \ldots \oplus D_n
\end{equation*}
for any $n \geq 1$. If
\begin{equation}\label{f:wn}
  w_n :=
  \begin{cases}
  (s_1 s_0)^{\frac{n-1}{2}} & \text{if $n$ is odd}, \\
  s_1(s_0 s_1)^{\frac{n-2}{2}} & \text{if $n$ is even}
  \end{cases}
\end{equation}
then
\begin{equation*}
  D_n = [\oplus_{\omega \in \Omega}\; H^j(I,\mathbf{X}(s_0 w_n\omega)^{K_m})] \oplus [\oplus_{\omega \in \Omega}\; H^j(I,\mathbf{X}(w_n\omega)^{K_m})] \ .
\end{equation*}
We observe that for any $\omega \in \Omega$ and $n \geq 1$ we have
\begin{equation}\label{f:length1}
  \ell(w_n\omega  s_{\epsilon(n)}) = \ell(w_n\omega ) + 1 \quad\text{and}\quad \ell(s_0 w_n\omega s_{\epsilon(n)}) = \ell(s_0 w_n\omega) + 1 \ ,
\end{equation}
whereas
\begin{equation}\label{f:length2}
  \ell(w_n\omega s_{\epsilon(n+1)}) = \ell(w_n\omega) - 1 \quad\text{and}\quad \ell(s_0 w_n\omega  s_{\epsilon(n+1)}) = \ell(s_0 w_n\omega ) - 1 \ .
\end{equation}
Here $\epsilon(n)$ denotes the parity of $n$, i.e., $\epsilon(n) \in \{0,1\}$ is such that $\epsilon(n) \equiv n \bmod 2$.

\begin{lemma}\label{tau-grading}
For any $n \geq 1$ we have:
\begin{itemize}
  \item[i.] $D_n \tau_{s_{\epsilon(n)}} \subseteq D_{n+1}$;
  \item[ii.] $C_n \tau_{s_{\epsilon(n+1)}} \subseteq C_n$;
  \item[iii.] $C_n \tau_{s_{\epsilon(n)}} \subseteq C_{n-1} \oplus D_{n+1}$ for $n \geq 2$;
  \item[iv.] $C_n \tau_{s_{\epsilon(n)}} \tau_{s_{\epsilon(n-1)}} \ldots \tau_{s_{\epsilon(2)}} \subseteq D_1 \oplus D_3 \oplus \ldots \oplus D_{2n-3} \oplus D_{2n-1}$;
  \item[v.] $C_n \tau_{s_{\epsilon(n)}} \ldots \tau_{s_{\epsilon(2)}} \tau_{s_{\epsilon(1)}} \tau_{s_{\epsilon(2)}} \ldots \tau_{s_{\epsilon(n)}} \subseteq D_{n+1} \oplus D_{n+3} \oplus \ldots \oplus D_{3n-3} \oplus D_{3n-1}$.
\end{itemize}
\end{lemma}
\begin{proof}
i. and ii. follow from \eqref{f:length1}, \eqref{f:length2}, and Lemma \ref{length-inj}.

iii. is obtained by combining i. and ii.

iv. Starting from iii. and and applying ii. and i. to the first and second summand, respectively, of the middle term we compute
\begin{equation*}
  C_n \tau_{s_{\epsilon(n)}} \tau_{s_{\epsilon(n-1)}} \subseteq C_{n-1} \tau_{s_{\epsilon(n-1)}} \oplus D_{n+1} \tau_{s_{\epsilon(n+1)}} \subseteq C_{n-2} \oplus D_n \oplus D_{n+2} \ .
\end{equation*}
Repeating this argument finitely many times gives the assertion.

v. Apply i. repeatedly to iv.
\end{proof}

Until the end of this section,  we  assume $m\geq 2$. We define
\begin{equation*}
  v(m) :=
  \begin{cases}
  s_1 (s_0 s_1)^{m-2} & \text{if $m$ is even}, \\
  s_0 (s_1 s_0)^{m-2} & \text{if $m$ is odd}.
  \end{cases}
\end{equation*}
One easily checks that $\tau_{v(m)} = \tau_{s_{\epsilon(m-1)}} \ldots \tau_{s_{\epsilon(2)}} \tau_{s_{\epsilon(1)}} \tau_{s_{\epsilon(2)}} \ldots \tau_{s_{\epsilon(m-1)}}$.

\begin{proposition}\label{tau-annihil}
$H^j(I,\mathbf{X}^{K_m})^1 \tau_{v(m)} = 0$.
\end{proposition}
\begin{proof}
We recall from Cor.\ \ref{length-support} that $H^j(I,\mathbf{X}^{K_m})^1 \subseteq C_{m-1}$. Hence Lemma \ref{tau-grading}.v implies that
\begin{equation*}
  H^j(I,\mathbf{X}^{K_m})^1 \tau_{v(m)} \subseteq H^j(I,\mathbf{X}^{K_m})^1 \cap (D_m \oplus D_{m+2} \oplus \ldots \oplus D_{3m-2}) = 0 \ .
\end{equation*}
\end{proof}

Recall that the  elements $e_1$ and $\zeta = (\tau_{s_0} + e_1)(\tau_{s_1} + e_1) + \tau_{s_1}\tau_{s_0}$ defined respectively in \S\ref{gene-rel}  and   \S\ref{sec:charH} are central in $H$, and $e_1$ is an idempotent.

\begin{lemma}\phantomsection\label{modulo}
\begin{itemize}
  \item[i.] $H^j(I,\mathbf{X}^{K_m})^1 \zeta^{m-2} \tau_{s_{\epsilon(m-1)}} = 0$.
  \item[ii.] $\zeta^{m - 1} \equiv (\tau_{s_{\epsilon(m)}} + 1) \zeta^{m - 2} e_1 \mod H \tau_{v(m)} H$.
\end{itemize}
\end{lemma}
\begin{proof}
This follows from Prop.\ \ref{tau-annihil} and Remark \ref{explicit-identities}.
\end{proof}

It easily follows from the braid and quadratic relations that the two-sided ideal in $H$ generated by $\tau_{s_i}$ is $k$-linearly generated by all $\tau_w$ such that $s_i$ occurs in a reduced decomposition of $w$. We see that
\begin{equation*}
  H e_1 / H\tau_{s_i} e_1 H = k e_1 \oplus k\tau_{s_{1-i}} e_1 \quad\text{with}\ (\tau_{s_{1-i}} e_1)^2 = - \tau_{s_{1-i}} e_1
\end{equation*}
and therefore that $H e_1 / H\tau_{s_i} e_1 H \cong k \times k[\tau_{s_{1-i}} e_1]/ \langle \tau_{s_{1-i}} e_1 + 1 \rangle$ is a semisimple $k$-algebra.

\begin{lemma}\label{-1}
$\tau_{s_{\epsilon(m)}}$ acts as $-1$ on $H^1(I,\mathbf{X}^{K_m})^1 \zeta^{m-2} e_1$.
\end{lemma}
\begin{proof}
By Lemma \ref{modulo}.i the $H$-module $H^1(I,\mathbf{X}^{K_m})^1 \zeta^{m-2} e_1$ in fact is an $H e_1 / H\tau_{s_{\epsilon(m-1)}} e_1 H$-module. Hence, as we  discussed above, $\tau_{s_{\epsilon(m)}}$ acts semisimply with possible eigenvalues $0$ and $-1$ on $H^1(I,\mathbf{X}^{K_m})^1 \zeta^{m-2} e_1$. But the eigenvalue $0$ does not occur, otherwise $H^1(I,\mathbf{X}^{K_m})^1 e_1$ would contain the trivial character $\chi_{triv}$ defined in \S\ref{sec:charH}: it is not the case as proved  in Cor.\ \ref{coro:case0} below.
\end{proof}

\begin{proposition}\label{zeta}
The central element $\zeta^{m-1}$ annihilates $H^1(I,\mathbf{X}^{K_m})^1$ as well as $Z_m$.
\end{proposition}
\begin{proof}
For $H^1(I,\mathbf{X}^{K_m})^1$ this follows from Prop.\ \ref{tau-annihil}, Lemma \ref{modulo}.ii, and Lemma \ref{-1}. On the other hand it is a formal fact (for any ring) that if a central element in $H$ annihilates a (right) $H$-module $M$ then it also annihilates the (left)$H$-modules $\Ext^j(M,H)$ (cf. \cite{BW} I Lemma 4.4). Hence for $Z_m$ it remains to recall Cor.\ \ref{Zm-dual}.
\end{proof}

An $H$-module $M$ will be called $\zeta$-torsion if every element of $M$ is annihilated by a power of $\zeta$.

\begin{corollary}\phantomsection\label{zeta-torsion}
\begin{itemize}
  \item[i.] Any $\zeta$-torsionfree $H$-module lies in $\mathcal F_G$.
  \item[ii.] Every $H$-module in $\mathcal{T}_G$ is $\zeta$-torsion.
\end{itemize}
\end{corollary}
\begin{proof}
Let $t_\zeta(M) := \{m \in M : \zeta^j m = 0\ \text{for some $j \geq 0$}\}$ denote the $\zeta$-torsion submodule of an $H$-module $M$. Because of Prop.\ \ref{zeta} a $\zeta$-torsionfree module cannot receive a nonzero map from some $Z_m$. Hence any $\zeta$-torsionfree module lies in $\mathcal{F}_G$. It follows that for any $M$ in $\mathcal{T}_G$ the projection map $M \longrightarrow M/t_\zeta (M)$ must be the zero map and consequently that $M = t_\zeta (M)$.
\end{proof}

\begin{corollary}\label{supersingular}
The $H$-modules $H^1(I,\mathbf{X}^{K_m})^1$ and $Z_m$ are supersingular.
\end{corollary}
\begin{proof} By Lemma \ref{defi:supersing}, it suffices to check that
 the (finite dimensional) $H$-modules in question are annihilated by  a power of $\zeta$ which is the claim of Prop.\ \ref{zeta}.
\end{proof}

\subsubsection{Main results on the fine structure of $Z_m$}

The following theorem will be proved in Section \ref{prooftheo}.

\begin{theorem}\phantomsection\label{maintheosocle}
\begin{enumerate}
\item Suppose  $\mathfrak{F} = \mathbb{Q}_p$. Then $H^1(I,\mathbf{X}^{K_m})^1 = 0$ for any $m \geq 1$.
\item Suppose $\mathbb{F}_q \subseteq k$, $\mathfrak{F} \neq \mathbb{Q}_p$, and  $p\neq 2$. Then  $H^1(I,\mathbf{X}^{K_1})^1 = 0$ and every supersingular character appears with nonzero multiplicity in the socle of $H^1(I,\mathbf{X}^{K_2})^1$ or of
$H^1(I,\mathbf{X}^{K_3})^1$, namely:\phantomsection
\begin{enumerate}
\item  For any $m\geq 1$ and $\epsilon(m) = i$ we have
\begin{equation*}
\dim_k H^1(I,\mathbf{X}^{K_m})^1(\chi_i) = \dim_k H^1(I,\mathbf{X}^{K_{m+1}})^1(\chi_i) =(m-1)f \ .
\end{equation*}
\item   For any supersingular character $\chi$ of $H$ with $\chi \neq \chi_0, \chi_1$ we have:
\begin{itemize}
  \item[i.] if $q \neq 3$, then
  \begin{equation*}
  \dim_k H^1(I,\mathbf{X}^{K_2})^1(\chi) =
  \begin{cases}
  f-1 & \text{if $z \mapsto \chi(\tau_{\omega_z})^{-1}$ is an automorphism of
         $\mathbb{F}_q$}, \\
  f & \text{otherwise};
  \end{cases}
\end{equation*}
  \item[ii.] if $q = p \neq 3$ and $\chi$ is the unique supersingular character
  such that  $z \mapsto \chi(\tau_{\omega_z})^{-1}$ is the identity map on
  $\mathbb{F}_p ^\times$, then
\begin{equation*}
  \dim_k H^1(I,\mathbf{X}^{K_3})^1(\chi) = 2f = 2 \ ;
\end{equation*}
  \item[iii.] if $q = p = 3$, then $z \mapsto \chi(\tau_{\omega_z})^{-1}$ is the identity map on
  $\mathbb{F}_p ^\times$ and we have
\begin{equation*}
  \dim_k H^1(I,\mathbf{X}^{K_2})^1(\chi) = 0 \qquad\text{and}\qquad  \dim_k H^1(I,\mathbf{X}^{K_3})^1(\chi) = f = 1 \ .
\end{equation*}
\end{itemize}
\end{enumerate}
\end{enumerate}
\end{theorem}
The case $\mathfrak F\neq \mathbb Q_p$ and $q=p=2$ or $3$ is studied  in \S\ref{sec:smallq}.

Recall that, by Prop.\ \ref{Z1}, we know that $Z_1=0$ for any $\mathfrak F$.

\begin{corollary}\phantomsection\label{Zm-zero}
\begin{enumerate}
\item Suppose  $\mathfrak{F} = \mathbb{Q}_p$. Then  $Z_m = 0$ for any $m \geq 1$.
\item Suppose  $\mathfrak{F} \neq \mathbb{Q}_p$,  $p \neq 2$ and that $\mathbb{F}_q \subseteq k$; then any simple supersingular $H$-module is isomorphic to a quotient of $Z_2$ or $Z_3$.
\end{enumerate}
\end{corollary}
\begin{proof} 1. Follows from point 1 of the Thm.\ by Cor.\ \ref{Zm-dual}.

2. Because of the assumption that $\mathbb{F}_q \subseteq k$ the simple supersingular $H$-modules are the characters $\chi \neq \chi_{triv}, \chi_{sign}$ (Prop.\ \ref{charH}). By the theorem, they all occur in the socles of $H^1(I,\mathbf{X}^{K_2})^1$ or $H^1(I,\mathbf{X}^{K_3})^1$. Therefore, by duality using Cor.\ \ref{Zm-dual} and \eqref{f:dual1}, the characters $\Hom_k(\upiota^*(\chi),k) \cong \upiota^*(\chi)$, where $\upiota$ is the involutive automorphism of $H$ defined in \S\ref{involution}, occur as quotients of $Z_2$ or $Z_3$. It is easy to see that with $\chi$ also $\upiota^*(\chi)$ runs over all supersingular characters.
\end{proof}

\begin{proposition}\label{Qp-equivalence}
If $\mathfrak{F} = \mathbb{Q}_p$ then $\mathcal{F}_G = \Mod(H)$ and the functor $\ften_0 : \Mod(H) \xrightarrow{\sim} \Mod^I(G)$ is an equivalence of categories with quasi-inverse $\fhom$.
\end{proposition}
\begin{proof}
By combining Cor.\ \ref{fully-faithful}.i with Cor.\ \ref{Zm-zero}.1 we obtain the first part of the assertion and the fact that $\ften$ is a fully faithful embedding. Thm.\ \ref{F-embeds} then tell us that $\fhom$ is a left quasi-inverse of $\ften$. By a straightforward variant of \cite{OSe} Prop.\ 3.2 we have that the $H_{x_i}$-modules $\mathbf{X}_{x_i}$ and $\mathbf{X}_{x_i}/H_{x_i}$, for $i \in \{0,1\}$, are projective. It then follows from Prop.\ \ref{firstlevel} and Prop.\ \ref{level-subquotients} that $\mathbf{X}$ is a projective $H$-module. This implies that the natural transformation $\tau : \ften_0 \xrightarrow{\;\simeq\;} \ften$ is an isomorphism. For any $V$ in $\Mod^I(G)$ we now consider the adjunction map
\begin{equation*}
  \ften_0(V^I) = \mathbf{X} \otimes_H V^I \longrightarrow V \ .
\end{equation*}
It is surjective by assumption. Let $V_0$ denote its kernel. By applying the functor $\fhom$ we obtain the commutative exact diagram
\begin{equation*}
  \xymatrix{
     0 \ar[r] & V^I_0 \ar[r] & (\mathbf{X} \otimes_H V^I)^I  \ar[r] & V^I \\
     &  & V^I , \ar[ur]_{=} \ar[u]^{\cong} &    }
\end{equation*}
where the perpendicular arrow is the adjunction homomorphism $V^I \longrightarrow \fhom \circ \ften_0(V^I)$ for the $H$-module $V^I$. As we have shown above it is an isomorphism. Hence $V_0^I$ and therefore also $V_0$ vanish. We conclude that $\fhom | \Mod^I(G)$ is a quasi-inverse of $\ften_0 \simeq \ften$.
\end{proof}

The above result, for $p > 2$ and $k = \overline{\mathbb{F}}_p$, reproves, by a completely different method, a main result in \cite{Koz}.

\begin{theorem}\label{FnotQp}
Let $p \neq 2$, $\mathbb{F}_q \subseteq k$, and $\mathfrak{F} \neq \mathbb{Q}_p$; for any $H$-module $M$ the following assertions are equivalent:
\begin{itemize}
  \item[i.]  $M$ lies in $\mathcal{T}_G$;
  \item[ii.]  $M$ is $\zeta$-torsion;
  \item[iii.] $M$ is a union of supersingular finite length submodules.
\end{itemize}
\end{theorem}
\begin{proof}
If $M$ lies in $\mathcal{T}_G$ then $M$ is $\zeta$-torsion by Cor.\ \ref{zeta-torsion}.ii. Next we assume that $M$ is $\zeta$-torsion, and we let $N \subseteq M$ be any finitely generated submodule. We have $\zeta^j N = 0$ for some $j \geq 0$. Since the factor algebra $H/\zeta^j H$ is finite dimensional over $k$ the module $N$ must have finite length and then is supersingular by the argument in the proof of Cor.\ \ref{supersingular}. Finally we assume that $M$ satisfies iii. Since $\mathcal{T}_G$ is closed under the formation of arbitrary direct sums and factor modules we may, in fact, assume that $M$ is of finite length and supersingular. By Cor.\ \ref{Zm-zero}.2 all its simple Jordan-Holder constituents lie in $\mathcal{T}_G$. Since $\mathcal{T}_G$ also is closed under the formation of extensions we obtain that $M$ lies in $\mathcal{T}_G$.
\end{proof}

\begin{theorem}\label{theo:torfree}
Let $p \neq 2$, $\mathbb{F}_q \subseteq k$, and $\mathfrak{F} \neq \mathbb{Q}_p$
; for any $H$-module $M$ the following assertions are equivalent:
\begin{itemize}
  \item[i.]  $M$ lies in $\mathcal{F}_G$;
  \item[ii.]  $M$ is $\zeta$-torsionfree;
  \item[iii.] $M$  does not contain any nonzero supersingular (finite length) submodule.
\end{itemize}
\end{theorem}

\begin{corollary}
 Under the assumptions of Thm.\ \ref{theo:torfree} a finite length $H$-module $M$ is in $\mathcal{F}_G$ if and only if it is nonsupersingular.
\end{corollary}
\begin{proof}[Proof of Thm.\ \ref{theo:torfree}] By Cor.\ \ref{isobimo}, the $k[\zeta]$-module $H$ is finitely generated. Therefore, the kernel of the operator $\zeta$ on $M$ is either zero or it contains  a  nonzero finite dimensional supersingular module. It proves that iii. implies ii. That ii. implies iii. is trivial. By Cor.\ \ref{zeta-torsion}.i we know that  ii. implies i. Now let $M$ in  $\mathcal{F}_G$. By definition,  $M$ does not contain any quotient of $Z_2$ or $Z_3$. By Cor.\ \ref{Zm-zero}, it does not contain any nonzero simple supersingular module and therefore i. implies iii.
\end{proof}

\subsection{Localization in $\zeta$ \label{local}}

 We denote by $H_\zeta$ the localization in $\zeta$ of $H$. As $\zeta$ is not a zero divisor in $H$ we have $H \subseteq H_\zeta$. Let $\epsilon = 0$ or $1$. By Cor.\ \ref{isobimo}, $H_\zeta$ is isomorphic, as a $(H_{x_\epsilon}, k[\zeta^{\pm 1}])$-bimodule, to a direct sum of two copies of $H_{x_\epsilon} \otimes_k k[\zeta^{\pm 1}]$. In particular, it is a free left $H_{x_\epsilon}$-module.

We are first going to prove the following result.

\begin{proposition}\label{baseext}
Suppose that $\mathbb F_q\subseteq k$. The base extension $M \otimes _{H_{x_\epsilon}} H_\zeta$ of any right $H_{x_\epsilon}$-module $M$ is a projective $H_\zeta$-module.
\end{proposition}

By symmetry we only need to prove the proposition for $H_{x_0}$. We first consider the case of simple modules.

\begin{lemma}\label{simples}
 Suppose that $\mathbb F_q\subseteq k$. If  $M$ is a  simple $H_{x_0}$-module, then  $M\otimes _{H_{x_0}} H_\zeta$   is a projective $H_\zeta$-module.
\end{lemma}
\begin{proof}
Recall that  $H_{x_0}$ is the finite dimensional  subalgebra of $H$ generated by $\tau_{s_0}$ and all $\tau_\omega$ for $\omega \in \Omega$. Since   $\mathbb F_q$ is contained in $k$,  the simple $H_{x_0}$-modules are one dimensional. When $k$ contains a primitive root of unity of order $(q-1)(q+1)$, this result is proved in \cite{CE} Thm.\ 6.10 (iii). We reproduce their argument here to check that it is valid under the hypothesis $\mathbb F_q\subseteq k$. Consider the $q-1$-idempotents $\{e_\lambda\}_{\lambda\in \hat\Omega}$  with sum $1$ in $k[\Omega]$ as defined in \ref{sec:idempo}. Let $M$ be a simple nontrivial $H_{x_0}$-module. There is a character $\lambda: \Omega \rightarrow k^\times $ such that $e_\lambda M\neq \{0\}$. If $\tau_{s_0}$ acts trivially on $e_\lambda M$ then, by irreducibility of $M$, we have $M = e_\lambda M$ and it is one dimensional. If there is $v \in M$ such that $v' := \tau_{s_0}e_\lambda v \neq 0$, then by \eqref{f:braid} and \eqref{f:quad} one easily checks that $kv'$ is stable under $H_{x_0}$ so $M=kv'$ is one dimensional.

A simple $H_{x_0}$-module may therefore be viewed as an algebra homomorphism $H_{x_0} \rightarrow k$. It is entirely determined by its restriction $\lambda$  to $\Omega$ and its value at $\tau_{s_0}$.

For $\lambda=1$, there are two corresponding characters  of $H_{x_0}$:
the restriction $\chi^0_{triv}$ of the trivial character $\chi_{triv}$ of $H$, and
the restriction $\chi^0_{sign}$ of the sign  character $\chi_{sign}$ of $H$ (see \S\ref{sec:charH}). Both these characters are projective $H_{x_0}$-modules since $e_1H_{x_0}$ decomposes as a right  $H_{x_0}$-module into $e_1\tau_{s_0}H_{x_0}\oplus e_1(\tau_{s_0}+1)H_{x_0}\cong \chi_{sign}^0\oplus \chi_{triv}^0$. Therefore if  $\chi\in\{\chi_{triv}^0, \chi_{sign}^0\}$, then $\chi\otimes _{H_{x_0}} H_\zeta$  is a projective $H_{\zeta}$-module.

If  $\lambda\neq 1$, then by  \eqref{f:quad}, there is only one corresponding character of $H_{x_0}$ and we denote it by $\chi_{\lambda}^0$.  It maps $\tau_{s_0}$ onto $0$. Let $e_\lambda$ be the idempotent of $H_{x_0}$ as defined in \ref{sec:idempo}. It is easy to check that $\chi_{\lambda}^0\cong\tau_{s_0}e_\lambda H_{x_0}$. Since $H_{\zeta}$ is a free hence flat $H_{x_0}$-module, $\chi_\lambda^0\otimes_{H_{x_0}} H_\zeta$  is isomorphic to $\tau_{s_0}e_\lambda H_\zeta$. By Remark \ref{explicit-identities}.ii, we have $\zeta \tau_{s_0}= \tau_{s_0} \tau_{s_1} \tau_{s_0}$ and hence $\tau_{s_0} = \zeta ^{-1} \tau_{s_0} \tau_{s_1}  \tau_{s_0}$ and $\tau_{s_0}e_\lambda = \zeta ^{-1} \tau_{s_0} e_\lambda \tau_{s_1} \tau_{s_0}$  in $H_\zeta$. The morphism of right $H_\zeta$-modules
\begin{align*}
H_\zeta & \longrightarrow  \tau_{s_0} e_\lambda H_\zeta \\
h       &\longmapsto \zeta^{-1}  \tau_{s_0} e_\lambda \tau_{s_1} h
\end{align*}
therefore  splits the inclusion $\tau_{s_0} e_\lambda H_\zeta \subseteq H_\zeta$, which proves that $\chi_\lambda^0\otimes_{H_{x_0}} H_\zeta\cong \tau_{s_0}  e_\lambda H_\zeta$ is a projective $H_\zeta$-module.

\end{proof}

Next we describe  the principal indecomposable modules of $H_{x_0}$ (provided that $\mathbb F_q\subseteq k$). As a right module over itself, $H_{x_0}$ decomposes into the direct sum of all $e_\lambda H_{x_0}$ for $\lambda\in \hat\Omega$ where $e_\lambda$ is the idempotent  defined in \ref{sec:idempo}. If $\lambda=1$, then $e_1 H_{x_0}$ is the sum of two projective simple modules as recalled in the proof of Lemma \ref{simples}. If $\lambda\neq 1$, one easily checks that $e_\lambda H_{x_0}$ is indecomposable with  length $2$. Its only simple submodule is $e_\lambda \tau_{s_0}H_{x_0}\cong \chi_{\lambda^{-1}}^0$ and
its only nontrivial quotient   is isomorphic to  $\chi_{\lambda}^0$ via the map $e_\lambda H_{x_0}\rightarrow \tau_{s_0}e_\lambda H_{x_0}$ induced by left multiplication by $\tau_{s_0}$. This shows that  $e_\lambda H_{x_0}$ is  the projective cover of   $\chi_{\lambda}^0$. This discussion  shows in particular that any principal indecomposable  right $H_{x_0}$-module has a unique composition series (of length $2$). The statement is true likewise  for principal indecomposable  left $H_{x_0}$-modules. Therefore, $H_{x_0}$ is a generalized uniserial ring in the sense of \cite{Nak} Appendix p.\ 19.
By \cite{Nak} Thm.\ 17, this implies that every $H_{x_0}$-module is a direct sum of quotients of the principal indecomposable modules  and hence is a direct sum of modules which are either simple or projective: Using Lemma \ref{simples}, this concludes the proof of Prop.\ \ref{baseext}.

For any   right  (resp. left) $H$-module $Y$, we let  $Y_\zeta$ denote its base extension $ Y\otimes_H H_\zeta$ (resp.  $H_\zeta\otimes_H Y$)  to a $H_\zeta$-module.
Let  $m\geq 1$.
Since $H_\zeta$ is flat over $H$, the $H_\zeta$-module $(\X^{K_m})_\zeta$ injects in $(\X^{K_{m+1}})_\zeta$ and in $\X_\zeta$. Furthermore, as a right $H_\zeta$-module we have   $\X_{\zeta}\cong \underset{m}{\varinjlim}(\mathbf{X}^{K_m})_\zeta$.

\begin{corollary}\label{coro:projlocal}
If $\mathbb F_q\subseteq k$ then we have:
\begin{itemize}
\item[i.] The $H_\zeta$-modules  $(\X^{K_m})_\zeta$ and   $(\X^{K_{m+1}})_\zeta/ (\X^{K_m})_\zeta$  are finitely generated, reflexive and projective for any $m\geq 1$.
\item[ii.] The $ H_\zeta$-module $\X_\zeta$ is projective.
\end{itemize}
\end{corollary}
\begin{proof}
We first verify that the $ H_\zeta$-modules $(\X^{K_1})_\zeta$ and $(\X^{K_{m+1}})_\zeta/ (\X^{K_m})_\zeta$  are projective.
This is because,  by Prop.s \ref{firstlevel} and \ref{level-subquotients}, we have isomorphisms of right $H_\zeta$-modules $(\X^{K_1})_\zeta\cong \X_{x_0}\otimes _{H_{x_0}}H_\zeta$ and  $(\X^{K_{m+1}}/\X^{K_m})_\zeta \cong (\X^{I_{x_\epsilon}}/ \X^I)_\zeta^{n_m} \cong (\X_{x_\epsilon}/H_{x_\epsilon})^{n_m} \otimes _{H_{x_\epsilon}} H_\zeta$,  for $\epsilon =0$ or $\epsilon =1$. Prop.\ \ref{baseext} then ensures that these are  projective $H_{\zeta}$-modules. The $H$-module $H_\zeta$ being flat, $(\X^{K_{m+1}})_\zeta/ (\X^{K_m})_\zeta \cong (\X^{K_{m+1}}/\X^{K_m})_\zeta$ as right $H_\zeta$-modules.  It implies that $(\X^{K_m})_\zeta$ is a projective right $H_{\zeta}$-module for any $m\geq 1$. It is  easy to see that a finitely generated projective module is also reflexive (use the fact that it is a direct summand of a free module of finite rank). This proves i.

For ii. let $M$ be any right $H_\zeta$-module. We consider the spectral sequence
\begin{equation*}
    E_2^{i,j} = {\varprojlim_m}^{(i)} \Ext_{H_\zeta}^j((\mathbf{X}^{K_m})_\zeta,M) \Longrightarrow \Ext_{H_\zeta}^{i+j}(\varinjlim_m (\mathbf{X}^{K_m})_\zeta,M) = \Ext_{H_\zeta}^{i+j}(\mathbf{X}_\zeta,M) \ .
\end{equation*}
By i., it degenerates into the isomorphisms
\begin{equation*}
    {\varprojlim_m}^{(i)} \Hom_{H_\zeta}((\mathbf{X}^{K_m})_\zeta,M) \cong \Ext_{H_\zeta}^i(\mathbf{X}_\zeta,M) \ .
\end{equation*}
But since  $(\X^{K_{m+1}})_\zeta/ (\X^{K_m})_\zeta$ is projective, the natural maps
\begin{equation*}
  \Hom_{H_\zeta}((\mathbf{X}^{K_{m+1}})_\zeta,M)\rightarrow  \Hom_{H_\zeta}((\mathbf{X}^{K_m})_\zeta,M)
\end{equation*}
are surjective. Consequently the above $\varprojlim^{(i)}$-terms vanish for $i \geq 1$, and $\mathbf{X}_\zeta$ is a projective ${H_\zeta}$-module.
\end{proof}

 We remark that, by a descent argument with respect to the finite Galois extension $k \mathbb{F}_q /k$, Prop.\ \ref{baseext} and Cor.\ \ref{coro:projlocal} remain true even without the assumption that $\mathbb{F}_q \subseteq k$.

The category $\Mod(H_\zeta)$ embeds into $ \Mod(H)$ as it identifies with the abelian subcategory of all $H$-modules on which the action of $\zeta$ is invertible. By  Prop.\ \ref{Qp-equivalence} and Thm.\ \ref{theo:torfree} it is contained in the torsionfree class and it contains all finite length modules of the torsionfree class provided the respective assumptions hold. We want to study the restriction of $\ften$ to  $\Mod(H_\zeta)$.

\begin{lemma} \label{lemma:duallocal}
Let $m\geq 1$. The natural morphism of  right $H_\zeta$-modules
\begin{equation*}
  ((\X^{K_m})^*)_\zeta \xrightarrow{\;\cong\;} \Hom_ {H_\zeta}((\X^{K_m})_\zeta, H_\zeta)
\end{equation*}
is an isomorphism.
\end{lemma}
\begin{proof}
Let $Y$ be any right $H$-module and consider the morphism of left $H$-modules
\begin{align}\label{duallocal}
(Y^*)_\zeta = H_\zeta\otimes_H Y^* & \longrightarrow \Hom_ {H}(Y, H_\zeta) \\
h \otimes \lambda & \longmapsto [y\mapsto h \lambda(y)].   \nonumber
\end{align}
It is easily checked to be injective since an element in $(Y^*)_\zeta$ can be written in the form $\zeta^{-i} \otimes \lambda$ for $i\geq 0$ and $\lambda\in  Y^*$. Now suppose that $Y$ is  finitely generated over $H$. Then, for any $\varphi\in \Hom_ {H}(Y, H_\zeta)$, there is  $i\geq 1$ such that $\zeta^i\varphi$ has values in $H$ and lies in $Y^*$. One checks that $\varphi$ is the image of $\zeta^{-i}\otimes  \zeta^i\varphi$.  So \eqref{duallocal} is an isomorphism. It induces an isomorphism of left $H_\zeta$-modules
\begin{equation*}
    (Y^*)_\zeta\cong \Hom_ {H_\zeta}(Y_\zeta, H_\zeta) \ .
\end{equation*}
By Lemma  \ref{univ-reflexive} we may apply this to $Y=\X^{K_m}$, which proves the lemma.
\end{proof}

\begin{theorem}\label{theo:localcoincide}
 Suppose that $\mathbb F_q\subseteq k$ and that either $p \neq 2$ or $\mathfrak{F} = \mathbb{Q}_p$. The functor $\ften$ coincides with $\ften_0$  on  $\Mod(H_\zeta)$.  It induces
an exact fully faithful functor $\Mod(H_\zeta)\rightarrow \Mod(G)$, and we have
$\fhom\circ \ften_0\vert_{\Mod(H_\zeta)}=\fhom\circ \ften\vert_{\Mod(H_\zeta)}={\rm id} _{\Mod(H_\zeta)}$.
\end{theorem}

\begin{remark}
 In the situation of Thm.\ \ref{theo:localcoincide}, if $V$ is in the essential image of $\ften_0\vert_{\Mod(H_\zeta)}=\ften\vert_{\Mod(H_\zeta)}$ then it lies in $\Mod^I(G)$ and the natural map  $\ften_0(V^I)\xrightarrow{\cong} V$ is an isomorphism.
\end{remark}

\begin{proof}[Proof of  Thm.\  \ref{theo:localcoincide}]
Let $P\in \Mod(H_\zeta)$ be a finitely generated projective (hence reflexive) $H_\zeta$-module. Then it is classical to establish that
the natural map
\begin{equation*}
   P \otimes_{H_\zeta} N \longrightarrow
\Hom_{H_\zeta}(\Hom_ {H_\zeta}(P, H_\zeta),N)
\end{equation*}
is an isomorphism for any $N\in \Mod(H_\zeta)$. If $P$ is free (of finite rank), it follows from its reflexivity. The claim follows  in general  since $P$ is a direct summand of  a  free module  of finite rank,  and since the functors $- \otimes_{H_\zeta} N$ and $\Hom_{H_\zeta}(\Hom_ {H_\zeta}(-, H_\zeta),N)$ commute with finite direct sums.

Now let $M$ in $\Mod(H)$ be any module with invertible action of $\zeta$. We need to prove that the map $\tau_M: \X\otimes_H M \longrightarrow \Hom_H(\X^*, M)$ defined in \eqref{defitauM} is injective. For this we  consider  again the diagram
\begin{equation}\label{gradedtau'}
    \xymatrix{
      \mathbf{X} \otimes_H M  \ar[r]^-{\tau_M} &
      \Hom_H(\mathbf{X}^\ast,M)  \\
      \mathbf{X}^{K_m} \otimes_H M \ar[u] \ar[r] &
      \Hom_H((\mathbf{X}^{K_m})^\ast,M), \ar[u]  }
\end{equation}
for any $m \geq 1$. As a consequence of Cor.\ \ref{restriction} the right perpendicular arrow is injective. Therefore, since $\X = \bigcup_m \X^{K_m}$,  it suffices to show that the lower horizontal map  is an isomorphism, for any $m \geq 1$. By Lemma \ref{lemma:duallocal},
we have isomorphisms
\begin{equation*}
  \Hom_{H_\zeta}( \Hom_ {H_\zeta}((\X^{K_m})_\zeta, H_\zeta),M) \cong \Hom_{H_\zeta}( ((\X^{K_m})^*)_\zeta,M_\zeta) \cong \Hom_H((\X^{K_m})^\ast,M)
\end{equation*}
and a commutative diagram
\begin{equation} \label{mloc}
  \xymatrix{
      \mathbf{X}^{K_m} \otimes_H M  \ar[r] &
      \Hom_H((\X^{K_m})^\ast,M)  \\
      (\mathbf{X}^{K_m})_\zeta \otimes_{H_\zeta} M \ar[u]^{\cong} \ar[r] &
      \Hom_{H_\zeta}(\Hom_ {H_\zeta}((\X^{K_m})_\zeta, H_\zeta),M), \ar[u]^{\cong}  }
\end{equation}
the lower horizontal map of which is an isomorphism by the first remark of this proof applied to $P=(\X^{K_m})_\zeta$ and using Cor.\ \ref{coro:projlocal}.i. Therefore, the top horizontal map of \eqref{mloc} is also an isomorphism. It is the lower horizontal map of \eqref{gradedtau'}. We have proved that $\ften_0$ and $\ften$ coincide on $\Mod(H_\zeta)$. Since $\ften_0$ is right exact on $\Mod(H)$ and $\ften$ preserves injective maps,  they are exact on $\Mod(H_\zeta)$. The last statement of the theorem is a direct consequence of Thm.\ \ref{F-embeds},  Prop.\ \ref{Qp-equivalence}, and Thm.\ \ref{theo:torfree}.
\end{proof}

\begin{remark}
If $V$ is a principal series representation, then it lies in $\Mod^I(G)$ (see, for example, \cite{Omjm} Prop.\ 4.3, which is a generalization of \cite{SS} p.\ 80, Prop.\ 11). The classification of the nonsupercuspidal representations of $G = \mathbf{SL_2}(\mathfrak{F})$ (\cite{Her}) shows that if $V$ is a subquotient of a principal series representation of $G$, then it is either irreducible or isomorphic to the principal series representation. In any case, it lies in $\Mod^I(G)$. Furthermore, $V^I$ is a finite dimensional nonsupersingular $H$-module (the proof is similar to the case of $G = \mathbf{GL_2}(\mathfrak{F})$ done in \cite{Viggl2} \S4.4 and \S4.5, see also \cite{Abd} Thm.\ 1.4). This is an example of a module in $\Mod(H_\zeta)$ and $V\cong\ften(V^I)\cong \ften_0(V^I) $ is in the essential image
 of $\ften\vert_{\Mod(H_\zeta)}= \ften_0\vert_{\Mod(H_\zeta)}$.

If $k$ is algebraically closed, the functor $\fhom$ induces a bijection  between  the isomorphism classes  of irreducible subquotients of  principal series representations on one side, and of simple nonsupersingular modules on the other side (\cite{Abd} Thm.\ 1.3). By Thm.\ \ref{theo:localcoincide}, the inverse bijection is induced by $\ften_0$ or $\ften$.
\end{remark}

\subsection{\label{prooftheo}Proof of Theorem \ref{maintheosocle}}

In this section we prove Thm.\  \ref{maintheosocle}: the statements will be verified in Cor.\ \ref{Zm-zero'} for the case $\mathfrak F=\mathbb Q_p$ and in Prop.s \ref{nosmall-q} and \ref{nosmall-q-chi1} for the case $\mathfrak F\neq \mathbb Q_p$.
Note that we first have to prove Cor.\ \ref{coro:case0}, the claim of which was used  above to show that
$H^1(I,\mathbf{X}^{K_m})^1$ and  $Z_m$ are supersingular modules (see Lemma \ref{-1} and subsequent corollaries).

Various technical group theoretical computations, which will be needed in the course of our proofs, are collected in the Appendix \S\ref{sec:App}.

\subsubsection{First results on the $H$-socle of $H^j(I,\mathbf{X}^{K_m})^1$}

For any $H$-module $M$ and any character $\chi$ of $H$  we let $M(\chi)$ denote the $\chi$-eigenspace in $M$.

\begin{remark}\phantomsection\label{case-sign}
\begin{itemize}
  \item[i.] $H^1(I,\mathbf{X}^{K_m})^1(\chi_{sign}) = 0$ for any $m \geq 2$.
  \item[ii.] $H^1(I,\mathbf{X}^{K_2})^1(\chi_1) = 0$.
\end{itemize}
\end{remark}
\begin{proof}
This is an immediate consequence of Prop.\ \ref{tau-annihil}.
\end{proof}

The following observation will later on allow an inductive reasoning.

\begin{lemma}\phantomsection\label{inductive}
\begin{itemize}
  \item[i.] The natural map $H^j(I,\mathbf{X}^{K_m})^1 \longrightarrow H^j(I,\mathbf{X}^{K_{m+1}})^1$ is injective.
  \item[ii.] Any $H$-submodule of $H^j(I,\mathbf{X}^{K_{m+1}})^1$ which is contained in $\bigoplus\limits_{\ell(w) + \sigma(w) < m} H^j(I,\mathbf{X}(w)^{K_{m+1}})$ is the image of an $H$-submodule in $H^j(I,\mathbf{X}^{K_m})^1$.
\end{itemize}
\end{lemma}
\begin{proof}
For $\ell(w) + \sigma(w) < m$ the natural map $H^j(I,\mathbf{X}(w)^{K_m}) \xrightarrow{\cong} H^j(I,\mathbf{X}(w)^{K_{m+1}})$ is an isomorphism. In view of Cor.\ \ref{length-support} this implies i. For ii. let $M_2 \subseteq H^j(I,\mathbf{X}^{K_{m+1}})^1$ be an $H$-submodule. If $M_2$ is contained in $\oplus_{\ell(w) + \sigma(w) < m} H^j(I,\mathbf{X}(w)^{K_{m+1}})$ Then it must be the image of an $H$-submodule $M_1 \subseteq \oplus_{\ell(w) + \sigma(w) < m} H^j(I,\mathbf{X}(w)^{K_m})$. With $M_2$ also $M_1$ is finite dimensional. Hence $M_1$ must be contained in $H^j(I,\mathbf{X}^{K_m})^1$.
\end{proof}

Here is a first application of this lemma:

\begin{lemma}\label{eigenspace-iso}
For $\epsilon(m) = i$ the natural map $H^j(I,\mathbf{X}^{K_m})^1(\chi_i) \xrightarrow{\cong} H^j(I,\mathbf{X}^{K_{m+1}})^1(\chi_i)$ is an isomorphism.
\end{lemma}
\begin{proof}
By Lemma \ref{inductive} it suffices to show that the projection of $H^j(I,\mathbf{X}^{K_{m+1}})^1(\chi_i)$ to $D_m$ is zero. We write $x \in H^j(I,\mathbf{X}^{K_{m+1}})^1(\chi_i)$ as $x = y + z$ with $y \in C_{m-1}$ and $z \in D_m$. It follows from Lemma \ref{tau-grading}.i that $z\tau_{s_i} = z\tau_{s_{\epsilon(m)}} = 0$. On the other hand, by Lemma \ref{tau-grading}.ii we have $C_{m-1}\tau_{s_i} \subseteq C_{m-1}$. Hence the condition $x\tau_{s_i} = -x$ forces $z$ to vanish.
\end{proof}

\subsubsection{Action of  the standard generators of $H$   on  $H^1(I,\mathbf{X}^{K_m})$ via Shapiro's lemma}\label{sec:explicit}

We fix an $m \geq 1$, and we consider a fixed but arbitrary $w \in \widetilde{W}$. As noted before in \S\ref{sec:I-decomp-X}, we have the $I$-equivariant isomorphisms $\mathbf{X}(w) \cong \ind _{I \cap wIw^{-1}}^I(1)$ and $\mathbf{X}(w)^{K_m} \cong \ind _{K_m(I \cap wIw^{-1})}^I(1)$. This leads to the isomorphisms
\begin{equation*}
  H^1(I,\mathbf{X}(w)^{K_m}) \xrightarrow{\; \cong \;} H^1(I,\ind _{K_m(I \cap wIw^{-1})}^I(1)) \xrightarrow{\; \cong \;} H^1(K_m(I \cap wIw^{-1}),k)
\end{equation*}
where the second one is given by Shapiro's lemma. In the following we denote the composed isomorphism by ``Sh''; it is induced by the map $\mathrm{ev}_w : \mathbf{X}(w) \longrightarrow k$ which sends $f$ to $f(w)$ (cf.\ \cite{Se1} p.11). Our goal is to identify the operators $\tau_\omega$ for $\omega \in \Omega$, $\tau_{s_0}$, and $\tau_{s_1}$ on its target.\\

\noindent\textbf{A) Action of $\tau_\omega$ on $H^1(I,\mathbf{X}(w)^{K_m})$:}\\

The easy case is $\tau_\omega$. By Lemma \ref{length-inj}.i it maps $H^1(I,\mathbf{X}(w)^{K_m})$ to $H^1(I,\mathbf{X}(w\omega)^{K_m})$. We obviously have $I \cap wIw^{-1} = I \cap w\omega I(w\omega)^{-1}$. One easily checks that the diagram
\begin{equation}\label{d:omega}
  \xymatrix{
  H^1(I,\mathbf{X}(w)^{K_m}) \ar[rr]^{\tau_\omega}_{\cong} \ar[dr]_{\cong}^{Sh}
                &  &    H^1(I,\mathbf{X}(w\omega)^{K_m})  \ar[dl]^{\cong}_{Sh}  \\
                & H^1(K_m(I \cap wIw^{-1}),k)               }
\end{equation}
is commutative.\\

\noindent\textbf{B) Action of $\tau_s$ on $H^1(I,\mathbf{X}(w)^{K_m})$ where $\ell(w s) = \ell(w) + 1$:}\\

For the next case we assume that $s \in \{s_0, s_1\}$ is such that $\ell(w s) = \ell(w) + 1$. Again Lemma \ref{length-inj}.i says that $\tau_s$ maps $H^1(I,\mathbf{X}(w)^{K_m})$ to $H^1(I,\mathbf{X}(ws)^{K_m})$. From \eqref{f:cap} we know that $I \cap wIw^{-1} \supset I \cap wsI(ws)^{-1}$. We claim that the diagram
\begin{equation}\label{d:s+}
  \xymatrix{
    H^1(I,\mathbf{X}(w)^{K_m}) \ar[d]_{Sh}^{\cong} \ar[r]^-{\tau_s} & H^1(I,\mathbf{X}(ws)^{K_m}) \ar[d]^{Sh}_{\cong} \\
    H^1(K_m(I \cap wIw^{-1}),k) \ar[r]^-{\mathrm{res}} & H^1(K_m(I \cap wsI(ws)^{-1}),k)   }
\end{equation}
is commutative where the bottom horizontal arrow is the cohomological restriction map. In the proof of Lemma \ref{length-inj}.i we have seen that the sets $gwIsI$, for $g \in I/I \cap wIw^{-1}$, are pairwise disjoint. It follows that
\begin{equation*}
  (\mathrm{char}_{gwI} \tau_s)(ws) = \mathrm{char}_{gwIsI}(ws) =
  \begin{cases}
  1 & \text{if $\mathrm{char}_{gwI}(w) = 1$}, \\
  0 & \text{if $\mathrm{char}_{gwI}(w) = 0$}
  \end{cases}
\end{equation*}
for any $g \in I$ and hence that $\mathrm{ev}_{ws} (f\tau_s) = \mathrm{ev}_w (f)$ for any $f \in \mathbf{X}(w)$.\\

\noindent\textbf{C) Action of $\tau_s$ on $H^1(I,\mathbf{X}(w)^{K_m})$ where $\ell(w s) = \ell(w) - 1$:}\\

The case where $s \in \{s_0, s_1\}$ is such that $\ell(w s) = \ell(w) - 1$ is the most complicated one.
By assumption we can write $w = vs$ with $\ell(w) = \ell(v) +1$. From Lemma \ref{length-inj}.ii we know that $\tau_s$ restricts to a map
\begin{equation*}
  \mathbf{X}(w) \longrightarrow \mathbf{X}(ws) \oplus \bigoplus_{\omega \in \Omega} \mathbf{X}(w\omega) \ ,
\end{equation*}
whose components we denote by
\begin{equation*}
  \tau_{w,s,0} : \mathbf{X}(w) \longrightarrow \mathbf{X}(ws) \qquad\text{and}\qquad \tau_{w,s,\omega} : \mathbf{X}(w) \longrightarrow \mathbf{X}(w \omega) \ .
\end{equation*}
By the formulas before \eqref{f:supports0} and \eqref{f:supports1}, respectively, they are given by
\begin{subequations}
\begin{align}
  (\mathrm{char}_{gwI})\tau_{w,s,0} & = \mathrm{char}_{gwsI},  \label{f:components-a} \\
  (\mathrm{char}_{gwI})\tau_{w,s,\omega} & =
  \begin{cases}
      \mathrm{char}_{g (vu^+_\omega v^{-1}) w \omega I} & \text{if $s = s_0$}, \\
      \mathrm{char}_{g (vu^-_\omega v^{-1}) w \omega I} & \text{if $s = s_1$}
  \end{cases} \label{f:components-b}
\end{align}
\end{subequations}
for any $g \in I$. From \eqref{f:cap} we know that $I \cap wIw^{-1} \subset I \cap wsI(ws)^{-1}$. The identity \eqref{f:components-a} implies that
\begin{equation*}
  (f \tau_{w,s,0})(hws) = \sum_{g \in I \cap wsI(ws)^{-1} / I \cap wIw^{-1}} f(hgw) \qquad\text{for any $f \in \mathbf{X}(w)$ and $h \in I$}.
\end{equation*}
It follows:
\begin{itemize}
  \item[--] If $K_m \cap wIw^{-1} \subsetneqq K_m \cap wsI(ws)^{-1}$ then $\tau_{w,s,0} : \mathbf{X}(w)^{K_m} \longrightarrow \mathbf{X}(ws)^{K_m}$ is the zero map.
  \item[--] If $K_m \cap wIw^{-1} = K_m \cap wsI(ws)^{-1}$ then the diagram
\begin{equation}\label{d:s-0}
  \xymatrix{
    H^1(I,\mathbf{X}(w)^{K_m}) \ar[d]_{Sh}^{\cong} \ar[r]^-{\tau_{w,s,0}} & H^1(I,\mathbf{X}(ws)^{K_m}) \ar[d]^{Sh}_{\cong} \\
    H^1(K_m(I \cap wIw^{-1}),k) \ar[r]^-{\mathrm{cores}} & H^1(K_m(I \cap wsI(ws)^{-1}),k)   }
\end{equation}
is commutative where the bottom horizontal arrow is the cohomological corestriction map (cf.\ \cite{Se1} p.11/12).
We point out that, by \eqref{f:cap}, the equality $K_m \cap wIw^{-1} = K_m \cap wsI(ws)^{-1}$ holds if and only if $K_m \subseteq I \cap wIw^{-1}$ if and only if $\ell(w) + \sigma(w) \leq m$.
\end{itemize}
Finally we contemplate the identity \eqref{f:components-b}.

\begin{remark}\label{normal}
Let $s \in \{s_0,s_1\}$ such that $\ell(ws) = \ell(w) - 1$ and $v := ws^{-1}$; we have:
\begin{itemize}
  \item[i.] $I \cap wIw^{-1}$ is normal in $I \cap vIv^{-1}$.
  \item[ii.] The map
  \begin{align*}
    u_{w,s} : \Omega & \xrightarrow{\; \sim \;} (I \cap vIv^{-1} / I \cap wIw^{-1}) \setminus \{1\} \\
    \omega & \longmapsto
    \begin{cases}
    v u^+_\omega v^{-1} & \text{if $s = s_0$}, \\
    v u^-_\omega v^{-1} & \text{if $s = s_1$}
    \end{cases}
  \end{align*}
  is a bijection.
  \item[iii.] $u_{w\omega_1,s}(\omega_2) = u_{w,s}(\omega_1^2 \omega_2)$ for any $\omega_1, \omega_2 \in \Omega$.
\end{itemize}
\end{remark}

The identity \eqref{f:components-b} says that the map $\tau_{w,s,\omega} : \mathbf{X}(w) \longrightarrow \mathbf{X}(w \omega)$ under the identifications $\mathbf{X}(w) \cong \ind_{I \cap wIw^{-1}}^I(1) \cong \mathbf{X}(w \omega)$ becomes the automorphism
\begin{align*}
  \ind_{I \cap wIw^{-1}}^I(1) & \longrightarrow \ind_{I \cap wIw^{-1}}^I(1) \\
  f & \longmapsto [g \mapsto f(gu_{w,s}(\omega)^{-1})] \ .
\end{align*}
It is standard that, after passing to cohomology and applying the Shapiro isomorphism, this latter map induces the conjugation map $\alpha \longmapsto u_{w,s}(\omega)^*(\alpha)(g) := \alpha(u_{w,s}(\omega) g u_{w,s}(\omega)^{-1})$ on $H^1(I \cap wIw^{-1},k)$.
We therefore have shown the commutativity of the diagram
\begin{equation}\label{d:s-omega}
  \xymatrix{
    H^1(I,\mathbf{X}(w)^{K_m}) \ar[d]_{Sh}^{\cong} \ar[rr]^-{\tau_{w,s,\omega}} & & H^1(I,\mathbf{X}(w\omega)^{K_m}) \ar[d]^{Sh}_{\cong} \\
    H^1(K_m(I \cap wIw^{-1}),k) \ar[rr]^-{u_{w,s}(\omega)^*} & & H^1(K_m(I \cap wIw^{-1}),k).   }
\end{equation}
Again there are two cases to distinguish:
\begin{itemize}
  \item[--] If $\ell(w) + \sigma(w) > m$ then $K_m(I \cap wIw^{-1}) = K_m(I \cap vIv^{-1})$ and all the $u_{w,s}(\omega)^*$ are the identity map. By \eqref{d:omega} this means that $\tau_{w,s,\omega} = \tau_\omega$.
  \item[--] If $\ell(w) + \sigma(w) \leq m$ then $K_m \subseteq I \cap wIw^{-1}$ and the $u_{w,s}(\omega)^*$ (together with the identity map) describe the natural action of the finite group $(I \cap vIv^{-1})/(I \cap wIw^{-1})$ on $H^1(K_m(I \cap wIw^{-1}),k) = H^1(I \cap wIw^{-1},k)$.\\
\end{itemize}
\begin{proof}[Proof of Remark \ref{normal}]
i. $I \cap s_0 I s_0^{-1} = I^-_1 = K_1$ and $I \cap s_1 I s_1^{-1} = I^+_1 =
\left(\begin{smallmatrix}
0 & 1 \\
\pi & 0
\end{smallmatrix}\right) K_1
\left(\begin{smallmatrix}
0 & 1 \\
\pi & 0
\end{smallmatrix}\right)^{-1}$ both are normal in $I$. It follows that $vIv^{-1}$ normalizes $vIv^{-1} \cap wIw^{-1}$ and hence that $I \cap vIv^{-1}$ normalizes $I \cap vIv^{-1} \cap wIw^{-1} = I \cap wIw^{-1}$.

ii. Let $s = s_0$. We may assume (cf.\ iii.) that $v = (s_0s_1)^j$ ($vC$ is antidominant if $j \neq 0$) or $v = s_1(s_0s_1)^j$ ($vC$ is dominant) for some $j \geq 0$. Suppose that $\omega = \omega_z$; then $u^+_\omega =
\left(\begin{smallmatrix}
1 & [z] \\
0 & 1
\end{smallmatrix}\right)$. A straightforward computation gives
\begin{equation*}
  v u^+_\omega v^{-1} =
  \begin{pmatrix}
  1 & \pi^{\ell(v)} [z] \\
  0 & 1
  \end{pmatrix}
  \in I^-_{\ell(v)}  \quad\text{and}\quad
  v u^+_\omega v^{-1} =
  \begin{pmatrix}
  1 & 0 \\
  -\pi^{\ell(v)+1} [z] & 1
  \end{pmatrix}
  \in I^+_{\ell(v)}, \text{respectively}.
\end{equation*}
Using \eqref{f:cap} we see that $v u^+_\omega v^{-1} \in I \cap vIv^{-1}$ in both cases. Moreover, due to the Iwahori factorization of all groups involved the asserted bijectivity is directly visible. We leave the analogous case $s = s_1$ to the reader.

iii. We have $ws = vs\omega_1 = (v\omega_1^{-1})s$. Hence $u_{w\omega_1,s}(\omega_2) = v\omega_1^{-1} u^+_{\omega_2} (v\omega_1^{-1})^{-1} = v\omega_1^{-1} u^+_{\omega_2} \omega_1 v^{-1}$ whereas $u_{w,s}(\omega_1^2 \omega_2) = v u^+_{\omega_1^2 \omega_2} v^{-1}$. Our claim therefore reduces to the identity $\omega_1^{-1} u^+_{\omega_2} \omega_1 = u^+_{\omega_1^2 \omega_2}$. This is verified by a simple computation.
\end{proof}
\subsubsection{\label{sec:prelilemm}Applications to the $H$-module structure of $H^1(I, {\mathbf X}^{K_m})^1$: preliminary lemmas}

As before we fix $m\geq 1$. The subspaces $D_n \subseteq H^1(I, {\mathbf X}^{K_m})$, for $n\geq 1$, were introduced in \S\ref{sec:defiDn}. In the following it is convenient to work with the decompositions
$$  D_n = D^-_n \oplus D^+_n  \quad\text{with} $$
\begin{equation*}
\quad D^-_n := \oplus_{\omega \in \Omega}\; H^1(I,\mathbf{X}(s_0 w_n\omega)^{K_m}), D^+_n := \oplus_{\omega \in \Omega}\; H^1(I,\mathbf{X}(w_n\omega)^{K_m})
\end{equation*}
for $n \geq 1$. We recall that:
\begin{itemize}
  \item[--] $D^\pm_n \tau_\omega = D^\pm_n$ for any $\omega \in \Omega$.
  \item[--] $D^\pm_n \tau_{s_{\epsilon(n)}} \subseteq D^\pm_{n+1}$.
  \item[--] $(D^\pm_{n-1} \oplus D^\pm_n) \tau_{s_{\epsilon(n+1)}} \subseteq D^\pm_{n-1} \oplus D^\pm_n$, for $n \geq 2$, and $D_1\tau_{s_0} \subseteq D_1$ (by Lemma \ref{length-inj}.ii).
\end{itemize}
Recall that $w_n$ was defined by \eqref{f:wn} and $I^{\pm}_n$ in \eqref{defiIn}. We observe that
\begin{equation}\label{f:wnIn}
I \cap w_n I w_n^{-1} = I_{n-1}^+ \qquad\textrm{and}\qquad I \cap (s_0 w_n) I(s_0 w_n)^{-1} = I_{n}^- \ .
\end{equation}
For $J\subseteq L$ two open subgroups of $I$, we will denote respectively by $\mathrm{res}^L_J$ and $\mathrm{cores}^J_L$ the restriction and corestriction maps
\begin{equation*}
  H^1(L,k) \xrightarrow{\;\mathrm{res}\;} H^1(J,k) \qquad\textrm{and}\qquad H^1(J,k) \xrightarrow{\;\mathrm{cores}\;} H^1(L,k) \ .
\end{equation*}
The discussion in the previous section \ref{sec:explicit} leads to the following explicit formulas. Assume that $2 \leq n < m$ and let $(\alpha, \beta)\in D_{n-1}^+\oplus D_n^+$. Using the Shapiro isomorphisms
\begin{align*}
  H^1(I,\mathbf{X}(w_{n-1}\omega)^{K_m}) & \cong H^1(K_m(I \cap w_{n-1} I w_{n-1}^{-1}),k) = H^1(I^+_{n-2},k) \ \text{and} \\
  H^1(I,\mathbf{X}(w_n\omega)^{K_m}) & \cong H^1(K_m(I \cap w_n I w_n^{-1}),k) = H^1(I^+_{n-1},k)
\end{align*}
we view $(\alpha,\beta)$ as elements
\begin{equation*}
  \alpha = (\alpha_\omega)_\omega \in \oplus_{\omega \in \Omega} \; H^1(I^+_{n-2},k) \quad\text{and}\quad  \beta = (\beta_\omega)_\omega \in \oplus_{\omega \in \Omega} \; H^1(I^+_{n-1},k) \ .
\end{equation*}
The commutativity of \eqref{d:omega} then implies that  for $\omega\in \Omega$, we have $(\alpha, \beta)\tau_{\omega}=(\alpha_{\tilde\omega\omega},\beta_{\tilde\omega\omega} )_{\tilde\omega}$. The commutativity of \eqref{d:s+}, \eqref{d:s-0}, and \eqref{d:s-omega}  implies that
\begin{equation*}
  (\alpha,\beta)\tau_{s_{\epsilon(n+1)}} = (\mathrm{cores}^{I^+_{n-1}}_{I^+_{n-2}} (\beta_{\omega^{-1}}), \mathrm{res}^{I^+_{n-2}}_{I^+_{n-1}}(\alpha_{\omega^{-1}}) + \sum_{\tilde\omega \in \Omega} u_{w_n \tilde\omega,s_{\epsilon(n+1)}}(\tilde\omega^{-1} \omega)^*(\beta_{\tilde\omega}) )_\omega \ .
\end{equation*}
By Remark \ref{normal}.iii we have $u_{w_n \tilde\omega,s_{\epsilon(n+1)}}(\tilde\omega^{-1} \omega) = u_{w_n,s_{\epsilon(n+1)}}(\tilde\omega \omega)$.  Therefore
\begin{equation}\label{f:explicit}
  (\alpha,\beta)\tau_{s_{\epsilon(n+1)}} = (\mathrm{cores}^{I^+_{n-1}}_{I^+_{n-2}} (\beta_{\omega^{-1}}), \mathrm{res}^{I^+_{n-2}}_{I^+_{n-1}}(\alpha_{\omega^{-1}}) + \sum_{\tilde\omega \in \Omega} u_{w_n ,s_{\epsilon(n+1)}}(\tilde\omega\omega)^*(\beta_{\tilde\omega}) )_\omega \ .
\end{equation}
There is a corresponding formula for $(\alpha, \beta)\in D_{n-1}^- \oplus D_n^-$ which we leave to the reader  to work out.

Finally, let $m \geq 1$ and $(\beta, \alpha)\in D_{1}^-\oplus D_1^+$. Using the Shapiro isomorphisms
\begin{align*}
  H^1(I,\mathbf{X}(s_0\omega)^{K_m}) & \cong H^1(K_m(I \cap s_0 I s_0^{-1}),k) = H^1(K_1,k) \ \text{and} \\
  H^1(I,\mathbf{X}(\omega)^{K_m}) & \cong H^1(K_mI),k) = H^1(I,k)
\end{align*}
we view $(\beta,\alpha)$ as elements
\begin{equation*}
  \beta = (\beta_\omega)_\omega \in \oplus_{\omega \in \Omega} \; H^1(K_1,k) \quad\text{and}\quad  \alpha = (\alpha_\omega)_\omega \in \oplus_{\omega \in \Omega} \; H^1(I,k) \ .
\end{equation*}
For $\omega\in \Omega$, we have $(\beta, \alpha)\tau_{\omega}=(\beta_{\omega'\omega},\alpha_{\omega'\omega} )_{\omega'}$
and the formula
\begin{equation}\label{f:explicit1}
  (\beta,\alpha)\tau_{s_0} = (\mathrm{res}^I_{K_1}(\alpha_{\omega^{-1}}) + \sum_{\tilde\omega \in \Omega} u_{s_0 ,s_0}(\tilde\omega\omega)^*(\beta_{\tilde\omega}) , \mathrm{cores}^{K_1}_I (\beta_{\omega^{-1}}))_\omega \ .
\end{equation}

\begin{lemma}\phantomsection\label{quadratic}
\begin{itemize}
  \item[i.] For $2 \leq n < m$ and $(\alpha,\beta) \in D^\pm_{n-1} \oplus D^\pm_n$ we have:
  \begin{itemize}
    \item[a)] Suppose that $\beta\tau_{s_{\epsilon(n)}} = 0$ and $(\alpha,\beta)\tau_{s_{\epsilon(n+1)}} = 0$; if $\mathfrak{F} = \mathbb{Q}_p$ then $\alpha \tau_{s_{\epsilon(n+1)}} = \beta = 0$;
    \item[b)] suppose that $(\alpha,\beta)\tau_\omega = (\alpha,\beta)$ for any $\omega \in \Omega$ and that $(\alpha,\beta)\tau_{s_{\epsilon(n+1)}} = 0$; then $\alpha \tau_{s_{\epsilon(n+1)}} = \beta$.
  \end{itemize}
  \item[ii.] For $m \geq 1$ and $(\beta,\alpha) \in D^-_1 \oplus D^+_1$ we have:
  \begin{itemize}
    \item[a)] Suppose that $(\beta,\alpha)\tau_{s_i} = 0$ for $i = 0,1$; if $\mathfrak{F} = \mathbb{Q}_p$ then $(\beta,\alpha) = 0$;
    \item[b)] suppose that $(\beta,\alpha)\tau_\omega = (\beta,\alpha)$ for any $\omega \in \Omega$, that $(\beta,\alpha)\tau_{s_0} = 0$, and that  $(\beta,\alpha)\tau_{w_n} = 0$ for some $n \geq 1$; then $(\beta,\alpha) = 0$.
  \end{itemize}
\end{itemize}
\end{lemma}

Recall that for any $H$-module $M$ we let $M(\chi)$ denote the $\chi$-eigenspace in $M$. The next result provides the statement which was needed in the proof of Lemma \ref{-1}.

\begin{corollary}\label{coro:case0}
Let $m\geq 2$. Then $H^1(I,\mathbf{X}^{K_m})^1(\chi_{triv}) = 0$.
\end{corollary}
\begin{proof}
Let $x \in H^1(I,\mathbf{X}^{K_m})^1$ be a $\chi_{triv}$-eigenvector. We then have $x\tau_{s_i} = 0$ for $i = 0,1$. By Cor.\ \ref{length-support} we may write
\begin{equation*}
  x = \alpha^-_{m-1} + \ldots + \alpha^-_1 + \alpha^+_1 + \ldots + \alpha^+_{m-1} \qquad\text{with $\alpha^\pm_n \in D^\pm_n$},
\end{equation*}
and Cor.\ \ref{length-support} and Lemma \ref{tau-grading} together show that
\begin{equation*}
  \alpha^\pm_{m-1} \tau_{s_{\epsilon(m-1)}} = 0 \ .
\end{equation*}
Furthermore, our assumptions together with Lemma \ref{length-inj}.ii imply that
\begin{equation*}
  (\alpha^\pm_{n-1} + \alpha^\pm_n)\tau_{s_{\epsilon(n+1)}} = 0 \quad\text{for any $2 \leq n \leq m-1$ and} \quad (\alpha^-_1 + \alpha^+_1)\tau_{s_0} = 0 \ .
  \end{equation*}
Therefore, applying Lemma \ref{quadratic}.i.b) gives
\begin{align*}
  \alpha^\pm_{m-1} & = \alpha^\pm_{m-2} \tau_{s_{\epsilon(m)}} \\
  \alpha^\pm_{m-2} & = \alpha^\pm_{m-3} \tau_{s_{\epsilon(m-1)}}  \\
  & \ \ \vdots \\
  \alpha^\pm_2 & = \alpha^\pm_1 \tau_{_{\epsilon(3)}} \ .
\end{align*}
It follows in particular that $\alpha^\pm_1 \tau_{w_{m-1}} = \alpha^\pm_1 \tau_{s_{\epsilon(3)}} \ldots \tau_{s_{\epsilon(m+1)}} = 0$. Then  Lemma \ref{quadratic}.ii.b) applies and gives $\alpha^\pm_1 = 0$.
\end{proof}

\begin{proof}[Proof of Lemma \ref{quadratic}]
i. The other case being completely analogous we only discuss the case $(\alpha,\beta) \in D^+_{n-1} \oplus D^+_n$. As above we view $(\alpha,\beta)$ as elements $\alpha = (\alpha_\omega)_\omega \in \oplus_{\omega \in \Omega} \; H^1(I^+_{n-2},k)$ and $\beta = (\beta_\omega)_\omega \in \oplus_{\omega \in \Omega} \; H^1(I^+_{n-1},k)$.

a) By the commutativity of \eqref{d:s+} the assumption that $\beta\tau_{s_{\epsilon(n)}} = 0$ means that $\beta_\omega | I^+_n = 0$, i.e., that
\begin{equation*}
  \beta_\omega \in \Hom(I^+_{n-1}/I^+_n,k) \qquad\text{for any $\omega \in \Omega$}.
\end{equation*}
On the other hand, the assumption that $(\alpha,\beta)\tau_{s_{\epsilon(n+1)}} = 0$ by \eqref{f:explicit} means that
\begin{equation*}
  \mathrm{cores}^{I^+_{n-1}}_{I^+_{n-2}} (\beta_{\omega}) = 0 \quad\text{and}\quad \mathrm{res}^{I^+_{n-2}}_{I^+_{n-1}}(\alpha_{\omega}) = - \sum_{\tilde\omega \in \Omega} u_{w_n,s_{\epsilon(n+1)}}(\tilde\omega \omega^{-1})^*(\beta_{\tilde\omega})
\end{equation*}
for any $\omega \in \Omega$. For any $i \geq 0$ let $U(\mathfrak{M}^i) := \left(
\begin{smallmatrix}
1 & 0 \\
\mathfrak{M}^i & 1
\end{smallmatrix}
\right)$, and $U(\{0\}) := \{1\}$. Since $I^+_{n-2} = U(\mathfrak{M}^{n-1}) I^+_{n-1}$ we may apply \cite{NSW} Cor.\ 1.5.8 and obtain the commutative diagram
\begin{equation*}
  \xymatrix{
    H^1(I^+_{n-1},k) \ar[d]^{\mathrm{res}} \ar[r]^-{\mathrm{cores}} & H^1(I^+_{n-2},k) \ar[d]^{\mathrm{res}} \\
    H^1(U(\mathfrak{M}^n),k) \ar[r]^-{\mathrm{cores}} & H^1(U(\mathfrak{M}^{n-1}),k).   }
\end{equation*}
The lower horizontal arrow is dual to the transfer map $U(\mathfrak{M}^{n-1}) \longrightarrow U(\mathfrak{M}^n)$. The latter, by \cite{Hup} Lemma IV.2.1, coincides with the $q$th power map $g \longmapsto g^q$. It follows that any $\alpha' \in H^1(U(\mathfrak{M}^n),k)$ such that $\mathrm{cores}(\alpha') = 0$ vanishes on $U(\mathfrak{M}^{n-1})^q = U(q\mathfrak{M}^{n-1})$. This, in particular, applies to $\beta_\omega | U(\mathfrak{M}^n)$. On the other hand, $\beta_\omega$ vanishes on $I^+_n$. We conclude that $\beta_\omega$ vanishes on $I^+_n U(q\mathfrak{M}^{n-1})$. If $\mathfrak{F} = \mathbb{Q}_p$ then $q\mathfrak{M}^{n-1} = \mathfrak{M}^n$, hence $I^+_n U(q\mathfrak{M}^{n-1}) = I^+_{n-1}$, and therefore $\beta_\omega = 0$ and $\mathrm{res}^{I^+_{n-2}}_{I^+_{n-1}}(\alpha_{\omega}) = 0$ for any $\omega \in \Omega$.

b) Because of the commutativity of \eqref{d:omega} the first assumption means that $\alpha_\omega = \alpha_1$ and $\beta_\omega = \beta_1$. Hence, using \eqref{f:explicit}, the second assumption can be rewritten as
\begin{equation*}
  \mathrm{cores}^{I^+_{n-1}}_{I^+_{n-2}} (\beta_1) = 0 \quad\text{and}\quad \mathrm{res}^{I^+_{n-2}}_{I^+_{n-1}}(\alpha_1) = - \sum_{\omega \in \Omega} u_{w_n,s_{\epsilon(n+1)}}(\omega)^*(\beta_1) \ .
\end{equation*}
Due to our observation (cf.\ comments after \eqref{d:s-omega}) that the $ u_{w_n,s_{\epsilon(n+1)}}(\omega)^*$ together with the identity describe the natural action of $I^+_{n-2}/I^+_{n-1}$ on $H^1(I^+_{n-1},k)$ we have (cf.\ \cite{NSW} Prop.\ 1.5.6)
\begin{equation*}
  \sum_{\omega \in \Omega} u_{w_n,s_{\epsilon(n+1)}}(\omega)^*(\beta_1) = - \beta_1 + \mathrm{res}^{I^+_{n-2}}_{I^+_{n-1}} \circ \mathrm{cores}^{I^+_{n-1}}_{I^+_{n-2}} (\beta_1) \ .
\end{equation*}
Together with the above equations this obviously implies $\mathrm{res}^{I^+_{n-2}}_{I^+_{n-1}}(\alpha_1) = \beta_1$.

ii. By exactly the same arguments as for i., but using \eqref{f:explicit1}, we obtain
\begin{alignat*}{2}
         \mathrm{res}^I_{K_1}(\alpha_\omega) = \beta_\omega & = 0 & \qquad & \text{in case a)}, \\
         \mathrm{res}^I_{K_1}(\alpha_\omega) = \beta_\omega & & & \text{in case b)}.
\end{alignat*}
The assumption that $(\beta,\alpha)\tau_{w_n} = 0$ means that $\beta_\omega | I^-_n = 0$ and $\alpha_\omega | I^+_{n-1} = 0$. It follows that $\beta_\omega | I^-_n (K_1 \cap I^+_{n-1}) = 0$. But $I^-_n (K_1 \cap I^+_{n-1}) = K_1$ and hence $\beta = 0$ in both cases. Then $\alpha_\omega | I^+_{n-1} K_1 = 0$. But $I^+_{n-1} K_1 = I$ and hence $\alpha = 0$.
\end{proof}

For later use we record the following fact which was shown in the course of the proof of Lemma  \ref{quadratic}.i.a).

\begin{remark}\label{kernels-intersection}
If $\mathfrak{F} = \mathbb{Q}_p$ then, for $i \geq 1$, we have
\begin{equation*}
  \ker \big( H^1(I^\pm_i,k) \xrightarrow{\mathrm{cores}} H^1(I^\pm_{i-1},k) \big) \cap \ker \big( H^1(I^\pm_i,k) \xrightarrow{\mathrm{res}} H^1(I^\pm_{i+1},k) \big) = 0 \ .
\end{equation*}
\end{remark}

\begin{lemma}\phantomsection\label{quadratic-cores}
\begin{itemize}
  \item[i.] Let $2 \leq n < m$ and suppose that $(\alpha,\beta) \in D^\pm_{n-1} \oplus D^\pm_n$ satisfies $(\alpha,\beta)\tau_\omega = (\alpha,\beta)$ for any $\omega \in \Omega$; if we view $(\alpha,\beta)$, under the Shapiro isomorphism, as an element $((\alpha_\omega,\beta_\omega))_\omega \in \oplus_{\omega \in \Omega} \big(H^1(I^+_{n-2},k) \oplus H^1(I^+_{n-1},k) \big)$, resp.\ $((\alpha_\omega,\beta_\omega))_\omega \in \oplus_{\omega \in \Omega} \big(H^1(I^-_{n-1},k) \oplus H^1(I^-_n,k) \big)$, then we have
      \begin{equation*}
        (\alpha,\beta)\tau_{s_{\epsilon(n+1)}} = -(\alpha,\beta) \quad\text{if and only if}\quad \alpha_1 = -\mathrm{cores}^{I^+_{n-1}}_{I^+_{n-2}} (\beta_1)\ , \text{resp.}\ -\mathrm{cores}^{I^-_n}_{I^-_{n-1}} (\beta_1) \ .
      \end{equation*}
  \item[ii.] Let $m\geq 1$ and suppose that $(\beta,\alpha) \in D^-_1 \oplus D^+_1$ satisfies $(\beta,\alpha)\tau_\omega = (\beta,\alpha)$ for any $\omega \in \Omega$; viewed, under the Shapiro isomorphism, as an element $((\beta_\omega,\alpha_\omega))_\omega \in \oplus_{\omega \in \Omega} \big(H^1(I^-_1,k) \oplus H^1(I,k) \big)$ we have
      \begin{equation*}
        (\beta,\alpha)\tau_{s_0} = -(\beta,\alpha) \quad\text{if and only if}\quad \alpha_1 = -\mathrm{cores}^{I^-_1}_I (\beta_1) \ .
      \end{equation*}
\end{itemize}
\end{lemma}
\begin{proof}
The other cases being entirely analogous we only give the argument for $(\alpha,\beta) \in D^\pm_{n-1} \oplus D^\pm_n$. Going back to the reasoning in the proof of Lemma \ref{quadratic}.i.b) we see that $\alpha_\omega = \alpha_1$ and $\beta_\omega = \beta_1$ for any $\omega \in \Omega$ and that
\begin{equation*}
  (\alpha,\beta)\tau_{s_{\epsilon(n+1)}} = (\mathrm{cores}^{I^\pm_{n-1}}_{I^\pm_{n-2}}(\beta_1), \mathrm{res}^{I^\pm_{n-2}}_{I^\pm_{n-1}}(\alpha_1) - \beta_1 + \mathrm{res}^{I^\pm_{n-2}}_{I^\pm_{n-1}} \circ \mathrm{cores}^{I^\pm_{n-1}}_{I^\pm_{n-2}}(\beta_1))_\omega \ .
\end{equation*}
\end{proof}

\subsubsection{The case $\mathfrak{F} = \mathbb{Q}_p$}

\begin{proposition}\label{case0Qp}
Let $m \geq 2$. Suppose that $\mathfrak{F} = \mathbb{Q}_p$.
Then $H^1(I,\mathbf{X}^{K_m})^1(\chi) = 0$ for any character $\chi$ such that $\chi(\tau_{s_i}) = 0$ for $i = 0,1$.
\end{proposition}
\begin{proof}
Let $\chi$ be any character such that $\chi(\tau_{s_i}) = 0$ for $i = 0,1$, and let $x \in H^1(I,\mathbf{X}^{K_m})^1$ be a $\chi$-eigenvector. We then have $x\tau_{s_i} = 0$ for $i = 0,1$. Just like in the proof of Cor.\ \ref{coro:case0} we may write
\begin{equation*}
  x = \alpha^-_{m-1} + \ldots + \alpha^-_1 + \alpha^+_1 + \ldots + \alpha^+_{m-1} \qquad\text{with $\alpha^\pm_n \in D^\pm_n$},
\end{equation*}
and Cor.\ \ref{length-support} and Lemma \ref{tau-grading} together show that
$  \alpha^\pm_{m-1} \tau_{s_{\epsilon(m-1)}} = 0 \ .$
Furthermore, our assumptions together with Lemma \ref{length-inj}.ii imply that
\begin{equation*}
  (\alpha^\pm_{n-1} + \alpha^\pm_n)\tau_{s_{\epsilon(n+1)}} = 0 \quad\text{for any $2 \leq n \leq m-1$ and} \quad (\alpha^-_1 + \alpha^+_1)\tau_{s_0} = 0 \ .
\end{equation*}
By applying Lemma \ref{quadratic}.i.a) inductively, we obtain $\alpha^-_{m-1} = \ldots = \alpha^-_2 = \alpha^+_2 = \ldots = \alpha^+_{m-1} = 0$, and then Lemma \ref{quadratic}.ii.a) gives $\alpha^-_1 = \alpha^+_1 = 0$.
\end{proof}
\begin{proposition}\phantomsection\label{case-0Qp}
 If $\mathfrak{F} = \mathbb{Q}_p$ then $H^1(I,\mathbf{X}^{K_2})^1(\chi_0) = 0$.
\end{proposition}
\begin{proof}
As a consequence of Cor.\ \ref{length-support} the eigenspace $H^1(I,\mathbf{X}^{K_2})^1(\chi_0)$ coincides with the subspace of all $x \in \oplus_{\omega \in \Omega} \big( H^1(I,\mathbf{X}(s_0\omega)^{K_2}) \oplus H^1(I,\mathbf{X}(\omega)^{K_2}) \big)$ such that
\begin{equation*}
  x\tau_\omega = x \ \text{for any $\omega \in \Omega$},\ x\tau_{s_1} = 0,\ \text{and}\ x\tau_{s_0} = -x \ .
\end{equation*}
The first condition implies that $H^1(I,\mathbf{X}^{K_2})^1(\chi_0)$ is isomorphic to the subspace of all $y \in H^1(I,\mathbf{X}(s_0)^{K_2}) \oplus H^1(I,\mathbf{X}(1)^{K_2})$ such that $y\tau_{s_1} = 0$ and $y\tau_{s_0} = -y$. Under the Shapiro isomorphism
\begin{equation*}
  H^1(I,\mathbf{X}(s_0)^{K_2}) \oplus H^1(I,\mathbf{X}(1)^{K_2}) \cong H^1(K_1,k) \oplus H^1(I,k) \ ,
\end{equation*}
with $y$ corresponding to $(\beta,\alpha)$, the second condition, by \eqref{d:s+}, becomes the requirement that $\beta | I^-_2 = 0$ and $\alpha | I^+_1 = 0$, and the third condition, by Lemma \ref{quadratic-cores}, becomes the requirement that $\mathrm{cores}^{K_1}_I (\beta) = - \alpha$. Hence $H^1(I,\mathbf{X}^{K_2})^1(\chi_0)$ is isomorphic to the group
\begin{equation*}
  \{(\beta,\alpha) \in \Hom(K_1/I_2^-,k) \oplus \Hom(I/I_1^+,k) : \mathrm{cores}^{K_1}_I (\beta) = - \alpha \}
\end{equation*}
that is to say to the group
\begin{equation*}
  \{\beta \in \Hom(I_1^-/I_2^-,k) :\:  \mathrm{res}^{I}_{I_1^+}\circ \mathrm{cores}^{I_1^-}_I (\beta) =0 \} = \ker( \mathrm{res}^I_{I^+_1} \circ \mathrm{cores}^{I^-_1}_{I}) \cap \ker( \mathrm{res}^{I^-_1}_{I^-_2}) \ .
\end{equation*}
 If $\mathfrak F = \mathbb Q_p$, then the right hand intersection vanishes by Cor.\ \ref{images-directsum} in the Appendix.
\end{proof}

\begin{proposition}\label{case-0-1}
If $\mathfrak{F} = \mathbb{Q}_p$ then $H^1(I,\mathbf{X}^{K_m})^1(\chi_i) = 0$ for any $m \geq 2$ and $i = 0,1$.
\end{proposition}
\begin{proof}
We will prove this by induction with respect to $m$. The case $m = 2$ is settled by Remark \ref{case-sign}.ii and Prop.\ \ref{case-0Qp}. Suppose therefore that the assertion holds for some $m \geq 2$. By Lemma \ref{inductive} it suffices to show that $H^1(I,\mathbf{X}^{K_{m+1}})^1(\chi_i) \subseteq \oplus_{\ell(w) + \sigma(w) < m} H^1(I,\mathbf{X}(w)^{K_{m+1}})$. If $\epsilon(m) = i$ this is immediate from Lemma \ref{eigenspace-iso}. We therefore suppose for the rest of the proof that $\epsilon(m+1) = i$. Using Cor.\ \ref{length-support} and the Shapiro isomorphism we view an element $x \in H^1(I,\mathbf{X}^{K_{m+1}})^1(\chi_i)$ as
\begin{equation*}
  x = \sum_{\omega \in \Omega} (\alpha^-_{m,\omega} +\ldots + \alpha^-_{1,\omega} + \alpha^+_{1,\omega} + \ldots + \alpha^+_{m,\omega})
\end{equation*}
with $\alpha^-_{n,\omega} \in H^1(I^-_n,k)$ and $\alpha^+_{n,\omega} \in H^1(I^+_{n-1},k)$. Since $\chi_i(\tau_\omega) = 1$ we have $\alpha^\pm_{n,\omega} = \alpha^\pm_{n,1}$. This reduces us to showing that $\alpha^-_{m,1} = 0$ and $\alpha^+_{m,1} = 0$. There are two cases to distinguish.

First suppose that $m \geq 3$. We discuss $\alpha^+_{m,1}$ leaving the analogous argument for $\alpha^-_{m,1}$ to the reader.  Using Lemma \ref{length-inj}.ii we then have
\begin{align*}
   & (\alpha^+_{m-1,1},\alpha^+_{m,1})\tau_{s_i} = - (\alpha^+_{m-1,1},\alpha^+_{m,1}) \ , \\
   & (\alpha^+_{m-2,1},\alpha^+_{m-1,1})\tau_{s_{1-i}} = 0 \ ,  \\
   & \alpha^+_{m,1} \tau_{s_{1-i}} = 0 \ .
\end{align*}
The first identity, by Lemma \ref{quadratic-cores}.i, implies that $\alpha^+_{m-1,1} = - \mathrm{cores}^{I^+_{m-1}}_{I^+_{m-2}} (\alpha^+_{m,1})$. The second and third identities, by Lemma \ref{quadratic}.i.b) (and its proof), imply that $\alpha^+_{m-1,1} = \mathrm{res}^{I^+_{m-3}}_{I^+_{m-2}} (\alpha^+_{m-2,1})$ and $\mathrm{res}^{I^+_{m-1}}_{I^+_m} (\alpha^+_{m,1}) = 0$, respectively. It therefore follows from Cor.\ \ref{images-directsum} that $\alpha^+_{m-1,1} = 0$. Hence $\mathrm{cores}^{I^+_{m-1}}_{I^+_{m-2}} (\alpha^+_{m,1}) = 0$ and $\mathrm{res}^{I^+_{m-1}}_{I^+_m} (\alpha^+_{m,1}) = 0$. Using Remark \ref{kernels-intersection} we deduce that $\alpha^+_{m,1} = 0$.

Finally let $m = 2$ (and $i = 1$). This time we have
\begin{align*}
   & (\alpha^\pm_{1,1},\alpha^\pm_{2,1})\tau_{s_1} = - (\alpha^\pm_{1,1},\alpha^\pm_{2,1}) \ , \\
   & (\alpha^-_{1,1},\alpha^+_{1,1})\tau_{s_0} = 0 \ ,  \\
   & \alpha^\pm_{2,1} \tau_{s_0} = 0
\end{align*}
and hence
\begin{align*}
   & \alpha^-_{1,1} = \mathrm{res}^I_{I^-_1} (\alpha^+_{1,1}) \ ,\ \alpha^+_{1,1} = - \mathrm{cores}^{I^+_1}_{I} (\alpha^+_{2,1}) \ ,\ \mathrm{res}^{I^+_1}_{I^+_2} (\alpha^+_{2,1}) = 0 \ , \\
   & \alpha^-_{1,1} = - \mathrm{cores}^{I^-_2}_{I^-_1} (\alpha^-_{2,1}) \ ,\ \mathrm{res}^{I^-_2}_{I^-_3} (\alpha^-_{2,1}) = 0 \ .
\end{align*}
The two first identities together with Cor.\ \ref{images-directsum} imply that $\alpha^-_{1,1} = 0$. Then it follows from the second line of identities, by Remark \ref{kernels-intersection}, that $\alpha^-_{2,1} = 0$. Moreover, $\alpha^+_{2,1}$ lies in the intersection $\ker( \mathrm{res}^I_{I^-_1} \circ \mathrm{cores}^{I^+_1}_{I}) \cap \ker( \mathrm{res}^{I^+_1}_{I^+_2})$, which is zero by Cor.\ \ref{images-directsum} in the Appendix.
\end{proof}

\begin{corollary}\label{no-character}
If $\mathfrak{F} = \mathbb{Q}_p$ then we have $H^1(I,\mathbf{X}^{K_m})^1(\chi) = 0$ for any $m \geq 2$ and any character $\chi$ of $H$.
\end{corollary}
\begin{proof}
Prop.\ \ref{case0Qp}, Remark \ref{case-sign}.i, and Prop.\ \ref{case-0-1}.
\end{proof}

\begin{corollary}\label{Zm-zero'}
If $\mathfrak{F} = \mathbb{Q}_p$ then $H^1(I,\mathbf{X}^{K_m})^1 = 0$ for any $m \geq 1$.
\end{corollary}
\begin{proof}
The case $m = 1$ is known from Prop.\ \ref{Z1}. By base extension we may assume that $k$ is algebraically closed. In Prop.\ \ref{zeta} we have seen that the central element $\zeta^{m-1}$ annihilates $H^1(I,\mathbf{X}^{K_m})^1$. By the proof of Lemma \ref{onedim} all simple modules of the algebra $H/\zeta H$ and hence of the algebra $H/\zeta^{m-1}H$ are one dimensional. But according to Cor.\ \ref{no-character} the socle of $H^1(I,\mathbf{X}^{K_m})^1$ does not contain any one dimensional submodule. It follows that $H^1(I,\mathbf{X}^{K_m})^1 = 0$. (Alternatively one may use Cor.\ \ref{supersingular}  together with \cite{Oll3} Prop.\ 5.11. The former says that the socle of $H^1(I,\mathbf{X}^{K_m})^1$ is supersingular and the latter that every simple supersingular $H$-module is one dimensional, since the group $\mathbf{SL_2}$ is semisimple and simply connected.)
\end{proof}

\subsubsection{The case $\mathfrak{F} \neq \mathbb{Q}_p$, $p\neq 2$.}

In this section we always \textbf{suppose} that $\mathfrak{F} \neq \mathbb{Q}_p$. Among the supersingular characters of $H$, we already know from Remark \ref{case-sign}.ii that $\chi_1$ does not occur in the socle of $H^1(I,\mathbf{X}^{K_2})^1$.

\begin{proposition}\phantomsection\label{case-0}
Let $\mathfrak{F} \neq \mathbb{Q}_p$; then the projection map to $H^1(I,\mathbf{X}(s_0)^{K_2})$ induces an isomorphism
\begin{equation*}
     H^1(I,\mathbf{X}^{K_2})^1e_1 = H^1(I,\mathbf{X}^{K_2})^1(\chi_0) \cong \Hom(K_1/I_2^-,k) \ ;
\end{equation*}
in particular, this eigenspace has $k$-dimension $f$.
\end{proposition}
\begin{proof}
Just like in the proof of Prop.\ \ref{case-0Qp}, $H^1(I,\mathbf{X}^{K_2})^1(\chi_0)$ is isomorphic to the group
\begin{equation*}
  \{\beta \in \Hom(I_1^-/I_2^-,k) :\:  \mathrm{res}^{I}_{I_1^+}\circ \mathrm{cores}^{I_1^-}_I (\beta) =0 \} = \ker( \mathrm{res}^I_{I^+_1} \circ \mathrm{cores}^{I^-_1}_{I}) \cap \ker( \mathrm{res}^{I^-_1}_{I^-_2}) \ .
\end{equation*}
We now use our computations in the Appendix. If $\mathfrak F \neq \mathbb Q_p$ and  $p\neq 2$, the map
$\mathrm{res}^{I}_{I_1^+}\circ \mathrm{cores}^{I_1^-}_I$ is trivial by Prop.\ \ref{Transfer} and Remark \ref{Transfer-q-small}.i. If $\mathfrak F \neq \mathbb Q_p$ and  $p=2$, the map $\mathrm{res}^{I}_{I_1^+}\circ \mathrm{cores}^{I_1^-}_I$ is  trivial on $\ker( \mathrm{res}^{I^-_1}_{I^-_2})$ by Remark \ref{Transfer-q-small}.iii. Hence in both cases $H^1(I,\mathbf{X}^{K_2})^1(\chi_0)$ is isomorphic to the group $\Hom(I_1^-/I_2^-,k)$.

Finally, since $\chi_0(\tau_\omega) = 1$ for any $\omega \in \Omega$ we have $H^1(I,\mathbf{X}^{K_2})^1(\chi_0) \subseteq H^1(I,\mathbf{X}^{K_2})^1e_1$.
But, by Prop.\ \ref{tau-annihil} and Lemma \ref{-1}, we know that $\tau_{s_1}$ and $\tau_{s_0}$ act by $0$ and $-1$, respectively, on  $H^1(I,\mathbf{X}^{K_2})^1e_1$. Hence this inclusion is  an equality.
\end{proof}


\begin{proposition}\label{nosmall-q}
Suppose that $\mathbb{F}_q \subseteq k$ and $p \ne 2$; for any supersingular $\chi \neq \chi_0, \chi_1$ we have:
\begin{itemize}
  \item[i.] If $q \neq 3$, then
  \begin{equation*}
  \dim_k H^1(I,\mathbf{X}^{K_2})^1(\chi) =
  \begin{cases}
  f-1 & \text{if $z \mapsto \chi(\tau_{\omega_z})^{-1}$ is a Galois automorphism of
         $\mathbb{F}_q$}, \\
  f & \text{otherwise};
  \end{cases}
\end{equation*}
  \item[ii.] if $q = p \neq 3$ and $\chi$ is the unique supersingular character
  such that  $z \mapsto \chi(\tau_{\omega_z})^{-1}$ is the identity map on
  $\mathbb{F}_p ^\times$, then
\begin{equation*}
    \dim_k H^1(I,\mathbf{X}^{K_3})^1(\chi) = 2f = 2 \ ;
\end{equation*}
  \item[iii.] if $q = p = 3$, then $z \mapsto \chi(\tau_{\omega_z})^{-1}$ is the identity on $\mathbb F_p^\times$ and
   \begin{equation*}
  \dim_k H^1(I,\mathbf{X}^{K_2})^1(\chi) = 0 \:\textrm{ and }\:\: \dim_k H^1(I,\mathbf{X}^{K_3})^1(\chi) = f = 1 \ .
\end{equation*}
\end{itemize}
\end{proposition}
\begin{proof}
Under our assumptions we have $\chi(\tau_{s_0}) = \chi(\tau_{s_1}) = 0$. We put $\lambda(z) := \chi(\tau_{\omega_z})$ and note that, as a consequence of the assumptions, $\lambda$ is a nontrivial character of $(\mathfrak{O}/\mathfrak{M})^\times$.

i. We view an element in $x \in H^1(I,\mathbf{X}^{K_2})^1(\chi)$ as $(\beta_\omega,\alpha_\omega)_\omega \in \oplus_\omega (H^1(K_1,k) \oplus H^1(I,k))$. Then $\beta_\omega = \chi(\tau_\omega)\beta_1$ and $\alpha_\omega = \chi(\tau_\omega)\alpha_1$ for any $\omega \in \Omega$. Using \eqref{f:explicit1}, the vanishing of $x\tau_{s_0}$ translates into
\begin{equation*}
  \mathrm{cores}^{K_1}_I(\beta_1) = 0 \quad\text{and}\quad  \mathrm{res}^I_{K_1}(\alpha_1) = - \sum_{\omega \in \Omega} u_{s_0,s_0}(\omega)^*(\beta_\omega) = - \sum_{\omega \in \Omega} \chi(\tau_\omega) u_{s_0,s_0}(\omega)^*(\beta_1) \ .
\end{equation*}
By Prop.\ \ref{Transfer} the left identity is always  satisfied. On the other hand, by Remark \ref{normal}.ii, the $u_{s_0,s_0}(\omega)^*$ describe the natural conjugation action by $I/K_1$ (omitting the identity element) on $H^1(K_1,k)$. Hence the right identity can be rewritten as
\begin{equation*}
  \mathrm{res}^I_{K_1}(\alpha_1) = - \sum_{z \in (\mathfrak{O}/\mathfrak{M})^\times} \lambda(z)
  \left(\begin{smallmatrix}
  1 & [z] \\
  0 & 1
  \end{smallmatrix}\right)^*(\beta_1) \ .
\end{equation*}
Finally the vanishing of $x\tau_{s_1}$ translates into $\mathrm{res}^{K_1}_{I^-_2}(\beta_1) = 0$ and $\mathrm{res}^I_{I^+_1}(\alpha_1) = 0$. Since $K_1/K_1 \cap I^+_1 \cong I/I^+_1$ we therefore see that $H^1(I,\mathbf{X}^{K_2})^1(\chi)$ is isomorphic to the group of all
\begin{equation*}
  \beta_1\in \Hom(K_1/I^-_2, k) \ \text{such that}\ \sum_{z \in (\mathfrak{O}/\mathfrak{M})^\times} \lambda(z)
  \left(\begin{smallmatrix}
  1 & [z] \\
  0 & 1
  \end{smallmatrix}\right)^*(\beta_1) \  \text{is trivial on }K_1\cap I_1^+ \ .
\end{equation*}
By dualizing the isomorphism $(K_1)^{ab} \cong \mathfrak{M}/ \mathfrak{M}^2 \times (1 + \mathfrak{M})/(1 + \mathfrak{M}^2) \times \mathfrak{M}/\mathfrak{M}^2$ from Prop.\ \eqref{abelianization}.ii we view an element $\gamma \in H^1(K_1,k)$ as a triple
\begin{equation*}
  (\alpha, \delta, \beta) \in \Hom(\mathfrak{M}/ \mathfrak{M}^2,k) \times \Hom((1 + \mathfrak{M})/(1 + \mathfrak{M}^2),k) \times \Hom(\mathfrak{M}/ \mathfrak{M}^2,k) \ .
\end{equation*}
We have:
\begin{itemize}
  \item[--] $\gamma | I^-_2 = 0$ if and only if $\alpha = 0$ and $\delta = 0$.
  \item[--] $\gamma | I^+_1 = 0$ if and only if $\delta = 0$ and $\beta = 0$.
  \item[--] $\left(\begin{smallmatrix}
  1 & [z] \\
  0 & 1
  \end{smallmatrix}\right)^*(0,0,\beta) = (-\beta(z^2 .),-2\beta(z(. - 1)), \beta)$ (cf.\ Remark \ref{conjugation}.ii in the Appendix).
\end{itemize}
The above $\beta_1$ therefore corresponds to a triple $(0,0,\beta)$ such that $2 \sum_{z \in (\mathfrak{O}/\mathfrak{M})^\times} \lambda(z) \beta(zy) = 0 = \sum_{z \in (\mathfrak{O}/\mathfrak{M})^\times} \lambda(z) \beta(y)$ for any $y \in \mathfrak{M}/\mathfrak{M}^2$. Since $\lambda$ is nontrivial the vanishing of the second sum is automatic. Since $p \neq 2$ the factor $2$ in the first sum can be omitted. Moreover, $\sum_{z \in (\mathfrak{O}/\mathfrak{M})^\times} \lambda(z) \beta(z\pi y) = \sum_{z \in (\mathfrak{O}/\mathfrak{M})^\times} \lambda(y^{-1})\lambda(z) \beta(\pi z)$ for any $y \in (\mathfrak{O}/\mathfrak{M})^\times$. Replacing $\beta$ by $\beta' := \beta(\pi .)$ we conclude that
\begin{equation*}
  H^1(I,\mathbf{X}^{K_2})^1(\chi) \cong \{\beta' \in \Hom(\mathbb{F}_q,k) : \sum_{z \in \mathbb{F}_q^\times} \lambda(z)\beta'(z) = 0\}.
\end{equation*}
By our assumption that $\mathbb{F}_q \subseteq k$ the Galois automorphisms $\sigma$ of $\mathbb{F}_q$, viewed as homomorphisms into $k$, form a $k$-basis of $\Hom(\mathbb{F}_q,k)$. The orthogonality relations imply
\begin{equation*}
  \sum_{z \in \mathbb{F}_q^\times} \lambda(z)\sigma(z) =
  \begin{cases}
  -1 & \text{if $\lambda^{-1} = \sigma$}, \\
  0 & \text{otherwise}.
  \end{cases}
\end{equation*}

ii. Note that $\chi(\tau_{\omega_z}) = \lambda(z)=z^{-1}=z^{p-2}$ for all $z\in\mathbb F_p^\times$. We view an element in $x \in H^1(I,\mathbf{X}^{K_3})^1(\chi)$ as
\begin{equation*}
  (\alpha_{2,\omega}^-, \alpha_{1,\omega}^- ,\alpha_{1,\omega}^+, \alpha_{2,\omega}^+)_\omega \in \oplus_\omega (H^1(I_2^-,k)\oplus H^1(I_1^-,k) \oplus H^1(I,k)\oplus H^1(I_1^+,k) ) \ .
\end{equation*}
Then $\alpha_{i,\omega}^\pm = \chi(\tau_\omega)\alpha_{i,1}^\pm$ for any $\omega \in \Omega$. The vanishing of $x\tau_{s_0} $ and of $x\tau_{s_1}$ translates into
\begin{equation*}
(\alpha_{1,1}^- ,\alpha_{1,1}^+)\tau_{s_0} = 0 \ , \
\alpha_{2,1}^\pm\tau_{s_0}=0 \ , \ (\alpha_{1,1}^\pm,\alpha_{2,1}^\pm)\tau_{s_1} = 0 \ .
\end{equation*}
The second equation means that
\begin{equation*}
  \mathrm{res}^{I^-_{2}}_{I^-_3} (\alpha^-_{2,1}) = 0 \quad\text{and}\quad \mathrm{res}^{I^+_1}_{I^+_2} (\alpha^+_{2,1}) = 0 \ .
\end{equation*}
Since the corestriction  map $\mathrm{cores}^{I_2^-}_{I_1^-}$ is trivial by Prop.\ \ref{Transfer} in the Appendix the equation $(\alpha_{1,1}^-,\alpha_{2,1}^-)\tau_{s_1} = 0$ can be rewritten as (compare \eqref{f:explicit} and use Remark \ref{normal}.ii)
\begin{equation}\label{f:equation}
  \mathrm{res}^{I_1^-}_{I^-_2} (\alpha^-_{1,1})=-
\sum_{\omega \in \Omega} u_{s_0s_1,s_1}(\omega)^*(\alpha_{2,\omega}^-) = - \sum_{z \in (\mathfrak{O}/\mathfrak{M})^\times} \lambda(z)
  \left(\begin{smallmatrix}
  1 & -\pi[z]^{-1}\\
  0 & 1
  \end{smallmatrix}\right)^*(\alpha_{2,1}^-) \ .
\end{equation}
We are going to show that the right hand side is zero. As before, using Prop.\ \ref{abelianization}.ii, we view $\alpha^-_{2,1}$ as a triple
\begin{equation*}
  (\alpha, \delta, \beta) \in \Hom(\mathfrak{M}/ \mathfrak{M}^2,k) \times \Hom((1 + \mathfrak{M})/(1 + \mathfrak{M}^3),k) \times \Hom(\mathfrak{M}^2/ \mathfrak{M}^3,k) \ .
\end{equation*}
The fact that $\mathrm{res}^{I^-_{2}}_{I^-_3} (\alpha^-_{2,1}) = 0$ means that $\alpha = 0$ and $\gamma = 0$. We compute
\begin{equation*}
  \left(\begin{smallmatrix}
  1 & -\pi [z]^{-1} \\
  0 & 1
  \end{smallmatrix}\right)^*(0,0,\beta) = (0,z^{-1}\beta(\pi(1 - (.)^{-2})), \beta)
\end{equation*}
(observe that the map $(1 + \mathfrak{M})/(1 + \mathfrak{M}^3) \longrightarrow \mathfrak{M}^2/ \mathfrak{M}^3$ sending $t$ to $\pi(1-t^{-2})$ is a homomorphism of groups) and hence get
\begin{equation*}
  \sum_{z \in (\mathfrak{O}/\mathfrak{M})^\times} \lambda(z) \left(\begin{smallmatrix}
  1 & -\pi [z]^{-1} \\
  0 & 1
  \end{smallmatrix}\right)^*(\alpha^-_{2,1}) = (0, (\sum_z z^{-2}) \beta(\pi(1 - (.)^{-2})), (\sum_z z^{-1}) \beta) \ .
\end{equation*}
But $\sum_z z^{-2} = \sum_z z^{-1} = 0$ since $q \neq 2,3$. This shows that $\mathrm{res}^{I_1^-}_{I^-_2} (\alpha^-_{1,1})=0$, i.e., that $\alpha^-_{1,1}\tau_{s_1}=0$. Likewise, the equation $(\alpha_{1,1}^+,\alpha_{2,1}^+)\tau_{s_1} = 0$ implies that $\alpha^+_{1,1}\tau_{s_1}=0$. Therefore, $(\alpha_{1,\omega}^-,\alpha_{1,\omega}^+)_\omega\in H^1(I,\mathbf{X}^{K_2})^1(\chi)$, which is trivial according to i. We have proved that mapping $x$ to $(\alpha^-_{2,1},\alpha^+_{2,1})$ induces an isomorphism
\begin{equation*}
  H^1(I,\mathbf{X}^{K_3})^1(\chi) \cong \Hom(I^-_{2}/I^-_3,k) \oplus \Hom(I^+_{1}/I^+_{2},k) \ .
\end{equation*}
The right hand side has dimension $2$.

iii. Note that in this case we necessarily have $\lambda = \id_{\mathbb{F}_3^\times}$. We begin the argument exactly as in the proof of ii. with identical notations. So $H^1(I,\mathbf{X}^{K_3})^1(\chi)$ is isomorphic to the group of all
\begin{equation*}
  (\alpha_{2,1}^-, \alpha_{1,1}^- ,\alpha_{1,1}^+, \alpha_{2,1}^+) \in H^1(I_2^-/I^-_3,k) \oplus H^1(I_1^-,k) \oplus H^1(I,k)\oplus H^1(I_1^+/I^+_2,k)
\end{equation*}
such that $(\alpha_{1,1}^- ,\alpha_{1,1}^+)\tau_{s_0} = 0$ and $(\alpha_{1,1}^\pm,\alpha_{2,1}^\pm)\tau_{s_1} = 0$. The equation $(\alpha_{1,1}^+,\alpha_{2,1}^+)\tau_{s_1} = 0$, by \eqref{f:explicit}, means that $\mathrm{cores}_{I}^{I_1^+}(\alpha_{2,1}^+) = 0$ and that $\alpha_{2,1}^+$ determines $\alpha_{1,1}^+ | I^+_1$, But, since $\alpha_{2,1}^+ \in H^1(I^+_1/I^+_2,k)$, the former together with Remark \ref{Transfer-q-small}.ii in the Appendix implies $\alpha_{2,1}^+ = 0$ and hence that $\alpha_{1,1}^+ | I^+_1 = 0$. We conclude that $H^1(I,\mathbf{X}^{K_3})^1(\chi)$ is isomorphic to the group of all
\begin{equation*}
  (\alpha_{2,1}^-, \alpha_{1,1}^- ,\alpha_{1,1}^+) \in H^1(I_2^-/I^-_3,k) \oplus H^1(I_1^-,k) \oplus H^1(I/I^+_1,k)
\end{equation*}
such that $(\alpha_{1,1}^- ,\alpha_{1,1}^+)\tau_{s_0} = 0$ and $(\alpha_{1,1}^-,\alpha_{2,1}^-)\tau_{s_1} = 0$. But the equation $(\alpha_{1,1}^-,\alpha_{2,1}^-)\tau_{s_1} = 0$ now only means, in view of \eqref{f:equation}, that
\begin{align}\label{f:equations}
  & \alpha^-_{1,1} |
  \left(\begin{smallmatrix}
  1 & 0 \\
  \mathfrak{M} & 1
  \end{smallmatrix}\right) = 0 \ , \
  \alpha^-_{1,1} |
  \left(\begin{smallmatrix}
  1 & \mathfrak{M}^2 \\
  0 & 1
  \end{smallmatrix}\right) = 0 \ , \\
  & \alpha^-_{1,1} \big(
  \left(\begin{smallmatrix}
  t & 0 \\
  0 & t^{-1}
  \end{smallmatrix}\right) \big) =
  \alpha^-_{2,1} \big(
  \left(\begin{smallmatrix}
  1 & \pi (t^{-2} - 1) \\
  0 & 1
  \end{smallmatrix}\right) \big) \quad\text{for any $t \in 1 + \mathfrak{M}$}. \nonumber
\end{align}
On the other hand, by \eqref{f:explicit1}, the equation $(\alpha_{1,1}^- ,\alpha_{1,1}^+)\tau_{s_0} = 0$ means that $\mathrm{cores}^{I^-_1}_I(\alpha^-_{1,1}) = 0$ and that
\begin{equation}\label{f:equations2}
  \mathrm{res}^I_{I^-_1}(\alpha^+_{1,1}) = - \sum_{z \in (\mathfrak{O}/\mathfrak{M})^\times} \chi(\tau_{\omega_z})
  \left(\begin{smallmatrix}
  1 & [z] \\
  0 & 1
  \end{smallmatrix}\right)^*(\alpha^-_{1,1})
  = - \left(\begin{smallmatrix}
  1 & 1 \\
  0 & 1
  \end{smallmatrix}\right)^*(\alpha^-_{1,1}) +
  \left(\begin{smallmatrix}
  1 & -1 \\
  0 & 1
  \end{smallmatrix}\right)^*(\alpha^-_{1,1}) \ .
\end{equation}
By the proof of Prop.\ \ref{Transfer} (compare also the proof of Remark \ref{Transfer-q-small}.ii) in the Appendix the image of the transfer map $I^{ab} \xrightarrow{\mathrm{tr}} (I^-_1)^{ab}$ is isomorphic to $\{0\} \times \{1\} \times \mathfrak{M}/\mathfrak{M}^2$. Hence the identity $\mathrm{cores}^{I^-_1}_I(\alpha^-_{1,1}) = 0$ is equivalent to $\alpha^-_{1,1} |
\left(\begin{smallmatrix}
  1 & \mathfrak{M} \\
  0 & 1
\end{smallmatrix}\right) = 0$. In view of \eqref{f:equations} we obtain that $\alpha^-_{1,1}$ is the unique map given by
\begin{align*}
  & \alpha^-_{1,1} |
  \left(\begin{smallmatrix}
  1 & 0 \\
  \mathfrak{M} & 1
  \end{smallmatrix}\right) = 0 \ , \
  \alpha^-_{1,1} |
  \left(\begin{smallmatrix}
  1 & \mathfrak{M} \\
  0 & 1
  \end{smallmatrix}\right) = 0 \ , \\
  & \alpha^-_{1,1} \big(
  \left(\begin{smallmatrix}
  t & 0 \\
  0 & t^{-1}
  \end{smallmatrix}\right) \big) =
  \alpha^-_{2,1} \big(
  \left(\begin{smallmatrix}
  1 & \pi (t^{-2} - 1) \\
  0 & 1
  \end{smallmatrix}\right) \big) \quad\text{for any $t \in 1 + \mathfrak{M}$}.
\end{align*}
Finally, an explicit computation based upon these equations
shows that the right hand side of \eqref{f:equations2} vanishes on $I^-_1 \cap I^+_1$, i.e., that
\begin{equation*}
  \alpha^-_{1,1} \big( \left(\begin{smallmatrix}
  1 & 1 \\
  0 & 1
  \end{smallmatrix}\right) g
  \left(\begin{smallmatrix}
  1 & -1 \\
  0 & 1
  \end{smallmatrix}\right) \big) =
  \alpha^-_{1,1} \big( \left(\begin{smallmatrix}
  1 & -1 \\
  0 & 1
  \end{smallmatrix}\right) g
  \left(\begin{smallmatrix}
  1 & 1 \\
  0 & 1
  \end{smallmatrix}\right) \big)
  \qquad\text{for any $g \in I^-_1 \cap I^+_1$}.
\end{equation*}
Since, by the normality of $I^\pm_1$ in $I$, we have $I/I^-_1 \cap I^+_1 = I^-_1/I^-_1 \cap I^+_1 \times I^+_1/I^-_1 \cap I^+_1$, this then implies that, given $\alpha^-_{1,1}$, there is a unique $\alpha^+_{1,1} \in H^1(I/I^+_1,k)$ such that the identity \eqref{f:equations2} holds.
Altogether we have proved that the map sending $x$ to $\alpha^-_{2,1}$ induces an isomorphism
\begin{equation*}
  H^1(I,\mathbf{X}^{K_3})^1(\chi) \xrightarrow{\;\cong\;}
 \Hom(I^-_2/I^-_3,k)  \cong k \ .
\end{equation*}
If the element $x \in H^1(I,\mathbf{X}^{K_3})^1(\chi)$ comes from $H^1(I,\mathbf{X}^{K_2})^1(\chi)$ (cf.\ Lemma \ref{inductive}) then its $\alpha^-_{2,1}$-component is zero. It follows that $H^1(I,\mathbf{X}^{K_2})^1(\chi) = 0$.
\end{proof}

\begin{proposition}\label{nosmall-q-chi1}
For any $m\geq 1$ and $\epsilon(m) = i$ we have
\begin{equation*}
\dim_k H^1(I,\mathbf{X}^{K_m})^1(\chi_i) = \dim_k H^1(I,\mathbf{X}^{K_{m+1}})^1(\chi_i) =(m-1)f \ .
\end{equation*}
\end{proposition}
\begin{proof}
The case $m = 1$ follows from the triviality of $H^1(I,\mathbf{X}^{K_{1}})^1$ and Remark \ref{case-sign}.ii. The case $m = 2$ follows from Prop.\ \ref{case-0}.ii and Lemma \ref{eigenspace-iso}. We therefore assume in the following that $m \geq 3$. The left equality then is Lemma \ref{eigenspace-iso}. It remains to establish that $H^1(I,\mathbf{X}^{K_m})^1(\chi_i)$ has dimension $(m-1)f$. Using Cor.\ \ref{length-support} and the Shapiro isomorphism we view an element $x \in H^1(I,\mathbf{X}^{K_m})^1(\chi_i)$ as
\begin{equation*}
  x = \sum_{\omega \in \Omega} (\alpha^-_{m-1,\omega} +\ldots + \alpha^-_{1,\omega} + \alpha^+_{1,\omega} + \ldots + \alpha^+_{m-1,\omega})
\end{equation*}
with $\alpha^-_{n,\omega} \in H^1(I^-_n,k)$ and $\alpha^+_{n,\omega} \in H^1(I^+_{n-1},k)$. Since $\chi_i(\tau_\omega) = 1$ we have $\alpha^\pm_{n,\omega} = \alpha^\pm_{n,1}$. Using Lemma \ref{length-inj}.ii we see that
\begin{equation*}
    (\alpha^\pm_{m-2,1},\alpha^+_{m-1,1})\tau_{s_i} = - (\alpha^\pm_{m-2,1},\alpha^\pm_{m-1,1}) \quad \text{and} \quad
    \alpha^\pm_{m-1,1} \tau_{s_{1-i}} = 0 \ .
\end{equation*}
By Lemma \ref{quadratic-cores}.i and \eqref{d:s+} these conditions are equivalent to
\begin{align*}
& \alpha^+_{m-2,1} = - \mathrm{cores}^{I^+_{m-2}}_{I^+_{m-3}} (\alpha^+_{m-1,1}) \ ,\ \alpha^-_{m-2,1} = - \mathrm{cores}^{I^-_{m-1}}_{I^-_{m-2}} (\alpha^-_{m-1,1}) \ , \\
& \mathrm{res}^{I^+_{m-2}}_{I^+_{m-1}} (\alpha^+_{m-1,1}) = 0 \ ,\ \mathrm{res}^{I^-_{m-1}}_{I^-_m} (\alpha^-_{m-1,1}) = 0 \ .
\end{align*}
Invoking Prop.\ \ref{Transfer} and Remark \ref{Transfer-q-small}.iii in the Appendix  we see that they are further equivalent to
\begin{multline*}
   \qquad \alpha^\pm_{m-2,1}  = 0 \ , \ \mathrm{res}^{I^+_{m-2}}_{I^+_{m-1}} (\alpha^+_{m-1,1}) = 0 \ ,\ \mathrm{res}^{I^-_{m-1}}_{I^-_m} (\alpha^-_{m-1,1}) = 0 \\
   \text{if either $q \neq 3 , p \neq 2$ or $q = 3 , m \geq 4$ or $p = 2 , m \geq 4$},
\end{multline*}
respectively
\begin{multline*}
  \qquad \alpha^+_{1,1} = - \mathrm{cores}^{I^+_{1}}_I (\alpha^+_{2,1}) \ , \ \alpha^-_{1,1}  = 0 \ , \ \mathrm{res}^{I^+_{1}}_{I^+_{2}} (\alpha^+_{2,1}) = 0 \ , \ \mathrm{res}^{I^-_{2}}_{I^-_3} (\alpha^-_{2,1}) = 0  \\
  \text{if either $q = 3 , m = 3$ or $p = 2 , m = 3$}.
\end{multline*}

We first consider the case $q \neq 3$ and $p \neq 2$. If $m = 3$ then, for $x$ to be an eigenvector for $\chi_1$, there is one additional condition which is $(\alpha^-_{1,1},\alpha^+_{1,1})\tau_{s_0} = 0$. But it is a consequence of the above conditions. Hence we obtain that the map sending $x$ to $(\alpha^-_{2,1},\alpha^+_{2,1})$ induces an isomorphism
\begin{equation*}
  H^1(I,\mathbf{X}^{K_3})^1(\chi_1) \xrightarrow{\cong}
 \Hom(I^-_2/I^-_3,k) \oplus \Hom(I^+_1/I^+_2) \cong k^f \oplus k^f \ .
\end{equation*}
For $m \geq 4$ we deduce, in view of Lemma \ref{inductive}.ii, that $x - \sum_\omega (\alpha^-_{m-1,\omega} + \alpha^+_{m-1,\omega})$ comes from an eigenvector for $\chi_i$ in $H^1(I,\mathbf{X}^{K_{m-2}})^1$. We therefore have established the exact sequence
\begin{multline*}
0 \longrightarrow H^1(I,\mathbf{X}^{K_{m-2}})^1(\chi_{\epsilon(m-2)}) \longrightarrow H^1(I,\mathbf{X}^{K_m})^1(\chi_{\epsilon(m)}) \longrightarrow \\
\Hom(I^-_{m-1}/I^-_m,k)\oplus \Hom(I^+_{m-2}/I^+_{m-1},k) \cong k^{2f} \longrightarrow 0 \ .
\end{multline*}
Our claim now follows by induction.

Now suppose that $q=3$ or $p=2$. If $m=3$ we this time have $ \mathrm{res}^{I^+_{1}}_{I^+_{2}} (\alpha^+_{2,1}) = 0$ and
$\alpha^+_{1,1} = - \mathrm{cores}^{I^+_{1}}_{I} (\alpha^+_{2,1})$. It therefore follows from Remark \ref{Transfer-q-small}.i/iii in the Appendix that  $\alpha_{1,1}^+$ is trivial on $I_1^-$. Hence the condition $(\alpha^-_{1,1},\alpha^+_{1,1})\tau_{s_0} = \alpha^+_{1,1}\tau_{s_0}=0$ is automatically satisfied. We again obtain that the map sending $x$ to $(\alpha^-_{2,1},\alpha^+_{2,1})$ induces an isomorphism
\begin{equation*}
  H^1(I,\mathbf{X}^{K_3})^1(\chi_1) \xrightarrow{\cong}
\Hom(I^-_2/I^-_3,k) \oplus \Hom(I^+_1/I^+_2) \cong k^f \oplus k^f \ \ ,
\end{equation*}
and we conclude by induction as in the previous case.
\end{proof}

\subsubsection{The case $\mathfrak F\neq \mathbb Q_p$,  and $q=2$ or $q=3$\label{sec:smallq}}
\begin{proposition}\label{small-q}
For $q = 2$ or $3$ we have:
\begin{itemize}
  \item[i.] $H^1(I,\mathbf{X}^{K_2})^1 = H^1(I,\mathbf{X}^{K_2})^1(\chi_0)$ has $k$-dimension $f$;
  \item[ii.] $Z_2 = Z_2(\chi_1)$ has $k$-dimension $f$.
\end{itemize}
\end{proposition}
\begin{proof}
i. By Lemma \ref{modulo}.i, $H^1(I,\mathbf{X}^{K_2})^1$ is an $H/\mathfrak{s}_1$-module so that Remark \ref{modulesS1} applies
and we may decompose
\begin{equation*}
  H^1(I,\mathbf{X}^{K_2})^1 = \prod_{\gamma \in \Gamma} H^1(I,\mathbf{X}^{K_2})^1_\gamma
\end{equation*}
such that $H/\mathfrak{s}_1$ acts on the factor $H^1(I,\mathbf{X}^{K_2})^1_\gamma$ through its projection to $A_\gamma$ (notation in  Remark \ref{modulesS1}). The factors $H^1(I,\mathbf{X}^{K_2})^1_\gamma$ for $\gamma \neq \{1\}$ have a socle which can only contain characters which induce a nontrivial character of $\Omega$. There are no such characters for $q = 2$; if $q = 3$ then no such character occurs in $H^1(I,\mathbf{X}^{K_2})^1$ by Prop.\ \ref{nosmall-q}.iii. Therefore all these factors vanish, i.e., we have
\begin{equation*}
  H^1(I,\mathbf{X}^{K_2})^1 = H^1(I,\mathbf{X}^{K_2})^1_{\{1\}} \ .
\end{equation*}
Moreover, $A_{\{1\}}$ is a semisimple $k$-algebra whose only supersingular character is $\chi_0$. It follows that
\begin{equation*}
  H^1(I,\mathbf{X}^{K_2})^1 = H^1(I,\mathbf{X}^{K_2})^1_{\{1\}} = H^1(I,\mathbf{X}^{K_2})^1(\chi_0) \ .
\end{equation*}
According to Prop.\ \ref{case-0}.ii the right hand side has $k$-dimension $f$.

ii. By Cor.\ \ref{Zm-dual} and i. we have
\begin{equation*}
  Z_2 = \Ext^1_H(H^1(I,\mathbf{X}^{K_2})^1,H) \cong \Ext^1_H(\chi_0,H) \oplus \ldots \oplus \Ext^1_H(\chi_0,H)
\end{equation*}
with $f$ summands on the right hand side. As we have recalled earlier we know from \cite{OS} Thm.\ 0.2 that $H$ is Auslander-Gorenstein of self-injective dimension equal to $1$. Hence \cite{OS} Cor.\ 6.12 implies that $\Ext^1_H(\chi_0,H) = \Hom_k(\upiota^*(\chi_0),k)$, where $\upiota$ is the involutive automorphism of $H$ from \S\ref{involution}. It remains to note that $\upiota^*(\chi_0) = \chi_1$.
\end{proof}

\subsection{\label{sec:fullspaces}On the $H$-module structure of $H^j(I,\mathbf{X}^{K_m})$}
\subsubsection{\label{sec:m2}Case $m\geq 2$: the ``trivial'' part of the cohomology.}

As before we fix an $m \geq 2$. We filter $H^j(I,\mathbf{X}^{K_m})$ as a $k$-vector space by
\begin{equation*}
  \Fil^i H^j(I,\mathbf{X}^{K_m}) = \oplus_{w \in \widetilde{W}, \ell(w) + \sigma(w) > i} \; H^j(I,\mathbf{X}(w)^{K_m}) \ .
\end{equation*}
This is a descending filtration with $\Fil^0 H^j(I,\mathbf{X}^{K_m}) = H^j(I,\mathbf{X}^{K_m})$. If $\mathfrak{F}$ has characteristic zero then each graded piece of the associated graded vector space is finite dimensional.

\medskip

For $i = 0,1$ we have the two algebra homomorphisms
\begin{align*}
  X_i : \qquad\qquad H & \longrightarrow k[\Omega] \\
  \tau_\omega, \tau_{s_i},  \tau_{s_{1-i}} & \longmapsto \omega, -e_1, 0 \ .
\end{align*}
They induce $H$-module structures on the finite dimensional vector space $k[\Omega]$ in both of which the central element $\zeta$ introduced in Section \ref{sec:supersing} acts trivially. Therefore these two $H$-modules are supersingular.

\begin{lemma}\label{filtration}
For any $i \geq m$ we have:
\begin{itemize}
  \item[i.] $\Fil^i H^j(I,\mathbf{X}^{K_m})$ is an $H$-submodule of $H^j(I,\mathbf{X}^{K_m})$;
  \item[ii.] the natural map
  \begin{align*}
    \big( H^j(I,\mathbf{X}(s_0w_{i+1})^{K_m}) \oplus H^j(I,\mathbf{X}(w_{i+1})^{K_m}) \big) \otimes_k k[\Omega] & \xrightarrow{\; \cong \;} \gr^i H^j(I,\mathbf{X}^{K_m}) \\
    x \otimes \omega & \longmapsto x\tau_\omega
  \end{align*}
  is an isomorphism of $H$-modules, where $H$ acts on the left hand side only on the second factor through the homomorphism $X_{\epsilon(i)}$.
\end{itemize}
\end{lemma}
\begin{proof}
It immediately follows from \eqref{d:omega} that $\Fil^i H^j(I,\mathbf{X}^{K_m})$ is a $k[\Omega]$-submodule and that the map in ii. is an isomorphism of $k[\Omega]$-modules. Let $w \in \widetilde{W}$ be such that $\ell(w) + \sigma(w) = i + 1$. Lemma \ref{tau-grading}.i implies that $H^j(I,\mathbf{X}(w)^{K_m}) \tau_{s_{1-\epsilon(i)}} \subseteq \Fil^{i+1} H^j(I,\mathbf{X}^{K_m})$. In particular, $\tau_{s_{1-\epsilon(i)}}$ annihilates $\gr^i H^j(I,\mathbf{X}^{K_m})$. Finally, the discussion in section \ref{sec:explicit} before \eqref{d:s-0} and after \eqref{d:s-omega} implies (it is only here that the assumption $i \geq m$ is used) that $\tau_{s_{\epsilon(i)}} | H^j(I,\mathbf{X}(w)^{K_m}) = -e_1 | H^j(I,\mathbf{X}(w)^{K_m})$.

\end{proof}

\begin{proposition}\label{filtered-iso}
For any $i \geq m$ the map
\begin{equation*}
  \xymatrix{
    \big( H^j(I,\mathbf{X}(s_0w_{i+1})^{K_m}) \oplus H^j(I,\mathbf{X}(w_{i+1})^{K_m}) \big) \otimes_k H/(\tau_{s_{\epsilon(i)}} + e_1)H \ar[d]^{x \otimes \tau  \longmapsto x\tau}_{\cong} \\
    \Fil^i H^j(I,\mathbf{X}^{K_m})   }
\end{equation*}
is an isomorphism of $H$-modules, where on the source $H$ acts through multiplication on the second factor.
\end{proposition}
\begin{proof}
In the proof of Lemma \ref{filtration} we have seen that  $H^j(I,\mathbf{X}(s_0w_{i+1})^{K_m}) \oplus H^j(I,\mathbf{X}(w_{i+1})^{K_m})$ is annihilated by $\tau_{s_{\epsilon(i)}} + e_1$. Hence the asserted map is a well defined homomorphism of $H$-modules. Cor.\ \ref{length-iso} implies that $x \otimes \tau \longmapsto x\tau$ restricts to a linear isomorphism from
\begin{equation*}
  \big( H^j(I,\mathbf{X}(s_0w_{i+1})^{K_m}) \oplus H^j(I,\mathbf{X}(w_{i+1})^{K_m}) \big) \otimes_k \big( \oplus_{w, \ell(s_{\epsilon(i)}w) = \ell(w) + 1} k\tau_w \big)
\end{equation*}
onto $\Fil^i H^j(I,\mathbf{X}^{K_m})$. We have
\begin{equation*}
  \big( \oplus_{w, \ell(s_{\epsilon(i)}w) = \ell(w) + 1} k\tau_w \big) \cap (\tau_{s_{\epsilon(i)}} + e_1)H = 0
\end{equation*}
since multiplying from the left by $\tau_{s_{\epsilon(i)}}$ is injective on the first summand and is zero on the second. It remains to show that
\begin{equation*}
  \big( \oplus_{w, \ell(s_{\epsilon(i)}w) = \ell(w) + 1} k\tau_w \big) + (\tau_{s_{\epsilon(i)}} + e_1)H = H \ .
\end{equation*}
For this consider any $w' = s_{\epsilon(i)} v$ be such that $\ell(w') = \ell(v) +1$. Then $\tau_{w'} = \tau_{s_{\epsilon(i)}} \tau_v = (\tau_{s_{\epsilon(i)}} + e_1)\tau_v + \sum_{\omega \in \Omega} \tau_{\omega v}$, where $\ell(s_{\epsilon(i)}\omega v) = \ell(\omega v) + 1$ for any $\omega \in \Omega$.
\end{proof}

We note that the above proof also shows that the map $H/(\tau_{s_{\epsilon(i)}} + e_1)H \xrightarrow[\cong]{\tau_{s_{\epsilon(i)}} \cdot} \tau_{s_{\epsilon(i)}}H$ is an isomorphism of right $H$-modules.

\subsubsection{Case $m=1$: on the structure of $H^1(I,\mathbf{X}^{K_1})$ as an $H$-module}\label{m1full}

Recall that by Prop.\ \ref{Z1} and by \eqref{f:dual2},
we know that $H^1(I,\mathbf{X}^{K_1})$ does not contain any nonzero finite dimensional submodule. We are going to prove the following proposition.

\begin{proposition}\label{propH1m=1}
Suppose that $\mathbb F_q\subseteq k$; then we have:
\begin{itemize}
  \item[i.]  If $\mathfrak{F} \neq \mathbb Q_p$ and $p\neq 2$,  $q \neq 3$, then
\begin{equation*}
    H^1(I,\X^{K_1})\cong  H^f\oplus (\tau_{s_0}H)^{2f}\oplus ((\tau_{s_0}+e_1)H)^{f}  \cong H^f \oplus (H/(\tau_{s_0}+e_1)H)^{2f}\oplus (H/\tau_{s_0}H)^{f}
\end{equation*}
       as right $H$-modules.
  \item[ii.] If $\mathfrak{F} =\mathbb Q_p$, or $\mathfrak{F} \neq \mathbb Q_p$ and $p=q=3$, then
\begin{equation*}
    H^1(I,\X^{K_1})\cong  H^2\oplus \tau_{s_0}H \cong H^2 \oplus H/(\tau_{s_0}+e_1)H
\end{equation*}
       as right $H$-modules.
\end{itemize}
\end{proposition}

By Prop.s \ref{firstlevel}  and \ref{H-free}.i and because the cohomology $H^1(I,.)$ commutes with arbitrary direct sums of discrete $I$-modules (cf.\ \cite{Se1} I.2.2 Prop.\ 8), we have
\begin{equation*}
  H^1(I,\mathbf{X}^{K_1}) = H^1(I,\mathbf{X}_{x_0}) \otimes_{H_{x_0}} H \ .
\end{equation*}
as right $H$-modules. We are therefore reduced to computing the right $H_{x_0}$-module structure of $H^1(I,\mathbf{X}_{x_0})$, which the rest of this section is devoted to. Prop.\  \ref{propH1m=1} is a direct corollary of Cor.\ \ref{corH1m=1}. To prove the latter, we start by proving results about $H_{x_0}$-modules.

Suppose that $\mathbb F_q\subseteq k$. We begin by recalling what we already know from the proof of Lemma  \ref{simples} and the paragraph  thereafter:
\begin{enumerate}
  \item The simple $H_{x_0}$-modules are the characters $\chi_{triv}^0$, $\chi_{sign}^0$, and $\chi_\lambda^0$ for all $\lambda\in \hat\Omega \setminus \{1\}$.
  \item The principal indecomposable right $H_{x_0}$-modules are $\chi_{triv}^0$, $\chi_{sign}^0$, and $e_\lambda H_{x_0}$ for all $\lambda\in \hat\Omega \setminus \{1\}$. The latter has a unique composition series of length two with submodule $\chi_{\lambda^{-1}}^0$ and factor module $\chi_\lambda^0$ (and hence is a projective cover of $\chi_\lambda^0$).
  \item Every right $H_{x_0}$-module is a direct sum of simple modules and modules isomorphic to $e_\lambda H_{x_0}$ for some $\lambda\in \hat\Omega \setminus \{1\}$.
\end{enumerate}

The following result is immediate.

\begin{lemma}\label{lemmaPIM}
Suppose that $\mathbb F_q\subseteq k$. The only projective simple $H_{x_0}$-modules are $\chi_{triv}^0$ and $\chi_{sign}^0$.  Given a
principal indecomposable right $H_{x_0}$-module $M$ nonisomorphic to $ \chi_{sign}^0$, the socle of $M$ coincides with the kernel of the operator
$\tau_{s_0}$ acting on $M$.
\end{lemma}

\begin{lemma}\label{sumchar}
We have the exact sequences of right $H_{x_0}$-modules
\begin{equation*}
    0 \longrightarrow (\tau_{s_0}+e_1)H_{x_0} \xrightarrow{\;\subseteq\;} H_{x_0} \xrightarrow{\tau_{s_0}\, .}  \tau_{s_0}H_{x_0}   \longrightarrow  0
\end{equation*}
and
\begin{equation*}
0 \longrightarrow \tau_{s_0}H_{x_0} \xrightarrow{\;\subseteq\;} H_{x_0} \xrightarrow{(\tau_{s_0}+e_1)\, .}  (\tau_{s_0}+e_1)H_{x_0} \longrightarrow 0 \ .
\end{equation*}
Furthermore,
if $\mathbb F_q\subseteq k$, then  $\tau_{s_0}H_{x_0}$ (resp.\ $ (\tau_{s_0}+e_1)H_{x_0}  $) is, as a right $H_{x_0}$-module,  the sum of all $q$ characters of $H_{x_0}$  except for the character $\chi_{triv}^0$ (resp.\ $\chi_{sign}^0$).
\end{lemma}
\begin{proof}
For the first exact sequence, the argument is the same as in the proof of Prop.\ \ref{filtered-iso} (see the remark thereafter). The involution $\upiota$ of $H$ defined in \S\ref{involution} restricts  to an automorphism of $H_{x_0}$ and exchanges
$\tau_{s_0}$ and $-(\tau_{s_0}+e_1)$. The exactness of the second complex follows.

Now if $\mathbb F_q\subseteq k$, then  consider as in \S\ref{sec:idempo} the  family of idempotents  $\{e_\lambda\}_{\lambda\in \hat\Omega}$ in $H_{x_0}$.   As a right $H_{x_0}$-module,  $\tau_{s_0}H_{x_0}=\oplus_{\lambda\in\hat\Omega} \tau_{s_0}e_\lambda H_{x_0}$.
We saw in the proof of Lemma \ref{simples} that $\tau_{s_0}e_1H_{x_0}=\tau_{s_0}H_{x_0}e_1$ supports the character $\chi_{sign}^0$ and that for $\lambda\neq 1$, the $H_{x_0}$-module $\tau_{s_0}e_\lambda H_{x_0}$ is isomorphic to $\chi_\lambda^0$. We thus obtain the asserted decomposition of   $\tau_{s_0}H_{x_0}$. Applying the involution $\upiota$, which exchanges $\chi_{sign}^0$ and $\chi_{triv}^0$, then gives the decomposition of   $(\tau_{s_0}+e_1)H_{x_0}$.
\end{proof}

Since $K = {{\mathbf G}_{x_0}^\circ(\mathfrak O)}$ is the disjoint union of all $I\omega s_0I$ and $I\omega I=I\omega$ for   $\omega \in \Omega$, the
representation of $I$ on  $\mathbf{X}_{x_0}= \ind_I^K(1)$  is isomorphic to the direct sum
\begin{equation*}
  \bigoplus_{\omega\in \Omega} \mathbf{X}(\omega s_0) \oplus \mathbf{X}(\omega)
\end{equation*}
and hence
\begin{equation*}
  H^1(I,\mathbf{X}_{x_0}) = \bigoplus_{\omega\in \Omega} \ H^1(I,\mathbf{X}(\omega s_0)) \oplus  H^1(I,\mathbf{X}(\omega))
\end{equation*}
with the notation of \S\ref{sec:I-decomp-X}. Following the same convention as before in the calculations, we will denote an element
of  $H^1(I,\mathbf{X}_{x_0})$ by $(\beta_\omega, \alpha_\omega)_\omega$ according to the  above decomposition.
Via the Shapiro isomorphism
\begin{equation*}
H^1(I,\mathbf{X}(\omega s_0)) \xrightarrow{\; \cong \;} H^1(I,\ind _{I \cap s_0Is_0^{-1})}^I(1)) \xrightarrow{\; \cong \;} H^1(K_1,k)
\end{equation*}
and since
\begin{equation*}
H^1(I,\mathbf{X}(\omega)) \xrightarrow{\; \cong \;} H^1(I,\ind _{I}^I(1)) \xrightarrow{\; \cong \;} H^1(I,k) \ ,
\end{equation*}
we see $\beta_\omega$ and $\alpha_\omega$ as  elements of respectively $H^1(K_1, k)$  and $H^1(I, k)$. Since $I$ and $K_1$ are finitely generated pro-$p$ groups (see the App.\ \ref{sec:App} or \cite{Lub} Thm.\ 1), the spaces $H^1(I,k) = H^1(I,\mathbb{F}_p) \otimes_{\mathbb{F}_p} k$  and  $H^1(K_1,k)= H^1(K_1,\mathbb{F}_p) \otimes_{\mathbb{F}_p} k$ are finite dimensional, and therefore $H^1(I,\mathbf{X}_{x_0})$ is finite dimensional.

Note that the space  $ H^1(I,\mathbf{X}_{x_0})$ coincides with the space $D_1$ when $j=m=1$ introduced in \S\ref{sec:defiDn}. Therefore, the observations of  \S\ref{sec:explicit}  and \S\ref{sec:prelilemm} do apply in the current context. In particular, given $\lambda\in \hat\Omega$, the space $ H^1(I,\mathbf{X}_{x_0}) e_\lambda$ is the subspace of all elements $(\beta, \alpha)=(\beta_\omega,\alpha_\omega)_\omega\in H^1(I,\mathbf{X}_{x_0})$ such that $(\beta,\alpha)\tau_\omega=\lambda(\omega)(\beta,\alpha)$ that is to say  $\beta_\omega=\lambda(\omega)\beta_1$ and  $\alpha_\omega=\lambda(\omega)\alpha_1$   for any $\omega\in \Omega$.   We deduce:

\begin{lemma}\label{dim}
$\dim_k H^1(I,\mathbf{X}_{x_0})e_\lambda = \dim_k H^1(K_1, k) + \dim_k H^1(I, k)$.
\end{lemma}

\begin{lemma}\label{e1}
Suppose that $\mathbb{F}_q \subseteq k$. The right $H_{x_0}$-module $H^1(I,\mathbf{X}_{x_0})e_1$ is the direct sum of the two eigenspaces  $H^1(I,\mathbf{X}_{x_0})(\chi^0_{triv})$ and  $H^1(I,\mathbf{X}_{x_0})(\chi^0_{sign})$. The former has dimension $\dim_k H^1(I,k)$, the latter has dimension $\dim_k H^1(K_1,k)$.
In particular, if $\mathfrak{F} \neq \mathbb{Q}_p$ with $p \neq 2$, or if $\mathfrak{F} = \mathbb{Q}_p$, then
\begin{equation*}
  H^1(I,\mathbf{X}_{x_0})e_1 \cong (\chi^0_{triv})^{2f} \oplus (\chi^0_{sign})^{3f}\cong (\chi^0_{sign})^f \oplus (H_{x_0}e_1)^{2f} \ .
\end{equation*}
\end{lemma}
\begin{proof} The eigenspace decomposition of $ H^1(I,\mathbf{X}_{x_0})e_1$ follows from the points 1.-3. before Lemma \ref{lemmaPIM}. Now remark  that  any idempotent $e \in H_{x_0}$ induces naturally an isomorphism  $H^1(I,\mathbf{X}_{x_0})e \cong H^1(I,\mathbf{X}_{x_0}e)$.
We apply this to the idempotents $e'_1:= -e_1\tau_{s_0}$  and $e''_1:= e_1(\tau_{s_0}+1)$  which, by \cite{OS} Proof of Prop.\ 6.19, allows us to study
\begin{equation*}
  H^1(I,\mathbf{X}_{x_0})(\chi^0_{sign}) = H^1(I,\mathbf{X}_{x_0})e'_1 \cong H^1(I,\mathbf{X}_{x_0}e'_1)
\end{equation*}
and
\begin{equation*}
  H^1(I,\mathbf{X}_{x_0})(\chi^0_{triv})=H^1(I,\mathbf{X}_{x_0})e''_1 \cong H^1(I,\mathbf{X}_{x_0}e''_1) \ .
\end{equation*}
The representation $\mathbf{X}_{x_0} e''_1$ of $K$ is generated by $\mathrm{char}_I e_1''$. One easily checks that $\mathrm{char}_I e_1''$ is equal to $- \mathrm{char}_K$. Therefore, the representation of $K$ on the space $\mathbf{X}_{x_0} e''_1$ is the one dimensional trivial representation. This proves that $H^1(I,\mathbf{X}_{x_0})(\chi^0_{triv}) \cong H^1(I,\mathbf{X}_{x_0}e''_1)$ has dimension $\dim_k H^1(I,k)$. Therefore, by Lemma \ref{dim}, $H^1(I,\mathbf{X}_{x_0})(\chi^0_{sign})$ must have dimension $\dim_k H^1(K_1,k)$.

The explicit calculation of the dimensions is made in Prop.\ \ref{abelianization}.ii, Cor.\ \ref{Frattini}, and Prop.\ \ref{frattini-1}.i.  The left displayed isomorphism follows. The right hand one holds since $\chi^0_{triv} \oplus \chi^0_{sign} \cong H_{x_0}e_1$ by the points 1.-2. before Lemma \ref{lemmaPIM}.
\end{proof}

\begin{lemma}
\label{r-s} Suppose that $\mathbb{F}_q \subseteq k$. For $\gamma\neq \{1\}$ an orbit in $\hat\Omega$, the structure of
$H^1(I,\mathbf{X}_{x_0})e_\gamma$ as a right   $H_{x_0}$-module is  given by the following isomorphisms:
\begin{itemize}
   \item Let $\gamma:=\{\lambda, \lambda^{-1}\}$ be an orbit with cardinality $2$.
If  $\mathfrak{F} \neq \mathbb{Q}_p$ and $p \neq 2$, $q \neq 3$, then
\begin{equation*}
  H^1(I,\mathbf{X}_{x_0})e_\gamma \cong  (\chi^0_\lambda \oplus  {\chi^0_{\lambda^{-1}}})^{3f}\oplus (e_\lambda H_{x_0} \oplus e_{\lambda^{-1}} H_{x_0})^f \cong
 (\chi^0_\lambda \oplus  {\chi^0_{\lambda^{-1}}})^{3f}\oplus ({H_{x_0} e_\gamma})^f \ .
\end{equation*}
If  $\mathfrak{F} =\mathbb{Q}_p$ or  if $\mathfrak{F} \neq \mathbb{Q}_p$ and $p=q=3$, then
\begin{equation*}
  H^1(I,\mathbf{X}_{x_0})e_\gamma \cong  \chi^0_\lambda \oplus  {\chi^0_{\lambda^{-1}}}\oplus (e_\lambda H_{x_0} \oplus e_{\lambda^{-1}} H_{x_0})^2 \cong
\chi^0_\lambda \oplus  {\chi^0_{\lambda^{-1}}}\oplus ({H_{x_0} e_\gamma})^2 \ .
\end{equation*}
   \item  Let $\{{\lambda_0 = \lambda_0^{-1}}\}$ be the  unique nontrivial orbit of cardinality $1$.
If  $\mathfrak{F} \neq \mathbb{Q}_p$ and $p \neq 2$, $q \neq 3$, then
\begin{equation*}
  H^1(I,\mathbf{X}_{x_0})e_{{\lambda_0}} \cong (\chi^0_{\lambda_0})^{3f}\oplus ({H_{x_0} e_{{\lambda_0}}})^f \ .
\end{equation*}
If  $\mathfrak{F} = \mathbb{Q}_p$ or  if $\mathfrak{F} \neq \mathbb{Q}_p$ and $p=q=3$, then
\begin{equation*}
  H^1(I,\mathbf{X}_{x_0})e_{{\lambda_0}} \cong \chi^0_{\lambda_0} \oplus (H_{x_0} e_{{\lambda_0}})^2 \ .
\end{equation*}
\end{itemize}
\end{lemma}
\begin{proof}
Let $\gamma:=\{\lambda, \lambda^{-1}\}$ be an orbit with cardinality  $2$.  By point 3. before Lemma \ref{lemmaPIM} the  finite dimensional $H_{x_0}$-module $H^1(I,\mathbf{X}_{x_0})e_\gamma$ is isomorphic to a direct sum of $r$ copies of $\chi^0_\lambda$, of $r'$ copies of $\chi^0_{\lambda^{-1}}$, of $s$ copies of $e_{\lambda^{-1}} H_{x_0}$ and  of  $s'$ copies of $e_\lambda H_{x_0}$, where $r,r',s,s' \geq 0$.
By Lemma \ref{lemmaPIM}, the dimension $d_\lambda$ of the kernel of the operator  $H^1(I,\mathbf{X}_{x_0})e_\lambda \xrightarrow{\cdot \,\tau_{s_0}} H^1(I,\mathbf{X}_{x_0})e_{\lambda^{-1}}$ is equal to $r+s$. Furthermore, $r+s+s'$ is equal to the dimension of the space $H^1(I,\mathbf{X}_{x_0})e_\lambda$.

As recalled above, we write an element in $H^1(I,\mathbf{X}_{x_0})e_\lambda$  as $(\beta, \alpha)=(\lambda(\omega)\beta_1,\lambda(\omega)\alpha_1)_\omega$ with $\beta_\omega\in H^1(K_1,1)$, $\alpha_\omega\in H^1(I,k)$. Suppose that $(\beta,\alpha)\tau_{s_0}=0$.
Using \eqref{f:explicit1}, this translates into
\begin{equation}\label{ts0=0}
  \mathrm{cores}^{K_1}_I(\beta_1) = 0 \quad\text{and}\quad  \mathrm{res}^I_{K_1}(\alpha_1) = - \sum_{\omega \in \Omega} u_{s_0,s_0}(\omega)^*(\beta_\omega) = - \sum_{\omega \in \Omega} \lambda(\omega) u_{s_0,s_0}(\omega)^*(\beta_1) \ .
\end{equation}
 \begin{itemize}
\item Suppose $\mathfrak{F} \neq \mathbb Q_p$ and $p\neq 2$.
First recall that $\dim_k H^1(I, k) = 2f$ and  $\dim_k H^1(K_1, k) = 3f$ by Prop.\ \ref{abelianization}.ii and Cor.\ \ref{Frattini}, respectively.  Hence $\dim_k H^1(I,\mathbf{X}_{x_0})e_\lambda = 5f$ by Lemma \ref{dim}.

Suppose that $q\neq 3$, then by Prop.\ \ref{Transfer}, the first equality in \eqref{ts0=0} is always satisfied.
Furthermore,  using Prop.\ \ref{abelianization} and the second equality, $\alpha_1$ is completely determined by $\beta_1$ and  by its restriction to $(\begin{smallmatrix} 1& \mathfrak O\\ 0&1 \end{smallmatrix})\bmod (\begin{smallmatrix} 1& \mathfrak M\\ 0&1 \end{smallmatrix})$. This proves that $d_\lambda= \dim_k H^1(K_1, k) + f= 4f$. Likewise, $d_{\lambda^{-1}}= 4f$. We have $r+s= r'+s'=4f$. Furthermore, $r+s+s'= 5f$ so  $s=s'=f$ and $r=r'=3f$.

If $p=q=3$, then by the proof of Prop.\ \ref{Transfer}, the first equality  in \eqref{ts0=0}
means that $\beta_1$ is trivial on $(\begin{smallmatrix} 1&  \mathfrak M\\ 0&1 \end{smallmatrix})$.
On the other hand,  $\alpha_1$ is completely determined by $\beta_1$ and  by its restriction to $(\begin{smallmatrix} 1& \mathfrak O\\ 0&1 \end{smallmatrix})\bmod (\begin{smallmatrix} 1& \mathfrak M\\ 0&1 \end{smallmatrix})$. This proves that $d_\lambda= (\dim_kH^1(K_1, k) -1)+ 1= 3$. Likewise, $d_{\lambda^{-1}}= 3$. We have $r+s= r'+s' =3$ and $r+s+s'= 5$ so  $s=s'=2$ and $r=r'=1$.

\item Suppose that $\mathfrak{F} = \mathbb Q_p$. By Prop.\ \ref{frattini-1}, we have $\dim_k H^1(I, k) =2$ and $\dim_k H^1(K_1,k) =3$. By Lemma \ref{transfer}.ii, the first equality in \eqref{ts0=0} means that $\beta_1$ is trivial on $\overline{u^+(p)}$ if $p\neq 2$, respectively $\overline{u^+(2)}+ \overline{d_2}$ if $p=2$. So $\beta_1$ is determined by its values at $\overline{u^-(p)}$ and $\overline{d_p}$. As above, the second equality  in \eqref{ts0=0} means that   $\alpha_1$ is completely determined by $\beta_1$ and by its value at  $\overline{u^+(1)}$. This proves that $d_\lambda= 2+ 1= 3$. Likewise, $d_{\lambda^{-1}}= 3$. We have $r+s= r'+s'=3$ and  $r+s+s'= 5$ so  $s=s'=2$ and $r=r'=1$.
\end{itemize}

Let  $\gamma:=\{\lambda_0\}$ be the nontrivial orbit with cardinality $1$. The finite dimensional $H_{x_0}$-module $H^1(I,\mathbf{X}_{x_0})e_{\lambda_0}$ is a direct sum of $r$ copies of $\chi^0_{\lambda_0}$ and of  $s$ copies of $e_{\lambda_0} H_{x_0}$. It has dimension $r+2s$. The dimension $d_{\lambda_0}$ of the kernel of the map $H^1(I,\mathbf{X}_{x_0})e_{\lambda_0} \xrightarrow{\cdot \tau_{s_0}} H^1(I,\mathbf{X}_{x_0})e_{{\lambda_0}}$
is equal to $r+s$. A discussion as above gives the result announced in the lemma.
\end{proof}

\begin{corollary}\label{corH1m=1}
Suppose that $\mathbb{F}_q \subseteq k$. If $\mathfrak{F} \neq \mathbb{Q}_p$ and $p \neq 2$,  $q \neq 3$, then we have isomorphisms of right $H_{x_0}$-modules  \begin{align*}
H^1(I,\mathbf{X}_{x_0}) & \cong  H_{x_0}^f\oplus (\tau_{s_0}H_{x_0})^{2f}\oplus ((\tau_{s_0}+e_1)H_{x_0})^{f} \\
 & \cong H_{x_0}^f \oplus (H_{x_0}/(\tau_{s_0}+e_1)H_{x_0})^{2f} \oplus (H_{x_0}/\tau_{s_0}H_{x_0})^{f} \ .
\end{align*}
If $\mathfrak{F} = \mathbb{Q}_p$, or $\mathfrak{F} \neq \mathbb{Q}_p$ and $p=q=3$, then we have isomorphisms of right $H_{x_0}$-modules
\begin{equation*}
H^1(I,\mathbf{X}_{x_0})\cong  H_{x_0}^2\oplus \tau_{s_0}H_{x_0} \cong H_{x_0}^2 \oplus H_{x_0}/(\tau_{s_0}+e_1)H_{x_0} \ .
\end{equation*}
\end{corollary}
\begin{proof}
This is a consequence of Lemmas \ref{e1} and \ref{r-s}, together with Lemma \ref{sumchar}.
\end{proof}

\subsection{Appendix: computation of abelian  quotients and transfer maps in the case of $\mathbf{SL_2}(\mathfrak F)$}\label{sec:App}

For any pro-$p$ group $\Delta$, which we denote here multiplicatively, we let
\begin{itemize}
\item $[\Delta,\Delta]$ denote the closed subgroup generated by all commutators and $\Delta^{ab} := \Delta/[\Delta,\Delta]$ the maximal abelian pro-$p$ factor group, and
\item  $\Phi(\Delta)$ denote the Frattini subgroup of $\Delta$, i.e., the closed subgroup generated by all commutators and all $p$-th powers. We put $\Delta_\Phi := \Delta/\Phi(\Delta)$. As it is an $\mathbb{F}_p$-vector space we write the latter group additively. Its relevance for us comes from the standard formula $H^1(\Delta,k) = \Hom(\Delta_\Phi,k)$.
\end{itemize}

For $u \in \Delta$, we will denote by $\overline{u}$ its coset in the quotient group $\Delta_\Phi$ or in $\Delta^{ab}$ depending on the context.

Let $\Gamma$ be a pro-$p$-group and $\Delta$ an open subgroup. The transfer map  $\Gamma^{ab}\rightarrow  \Delta^{ab}$ can be computed as follows (see for example \cite{Bro} III.2 Ex.\ 2): Let $\{\gamma_i\}_{i\in \mathcal I}$ be a system of representatives of the right cosets $\Delta\backslash \Gamma$. Let $g \in \Gamma$. For any $i\in \mathcal I$, there is a unique $j(g,i)\in \mathcal I$ such that $\Delta \gamma_i g= \Delta \gamma_{j(g, i)}$. Then the transfer map $\Gamma^{ab} \rightarrow  \Delta^{ab}$ is given by
\begin{equation}\label{rappel:transfer}
   \overline{g} \mapsto \overline{\prod_{i\in \mathcal I} \gamma_i g \gamma_{j(g,i)}^{-1}} \ .
\end{equation}
It factors through a map $\Gamma_\Phi \rightarrow \Delta_\Phi$.

\subsubsection{Computation of abelian  quotients \label{sec:Frattini}}
Unless otherwise mentioned,   the locally compact nonarchimedean field $\mathfrak F$ is arbitrary.

\begin{proposition}\label{abelianization}

Suppose that $p \neq 2$. Let  $i \geq 0$.  We have:
\begin{itemize}
  \item[i.] $[I^+_i,I^+_i] =
  \left(\begin{smallmatrix}
   1+\mathfrak{M}^{i+1} & \mathfrak{M} \\
   \mathfrak{M}^{i+2} & 1+\mathfrak{M}^{i+1}
   \end{smallmatrix}\right)$ and  $[I^-_i,I^-_i] =
   \left(\begin{smallmatrix}
   1+\mathfrak{M}^{i+1} & \mathfrak{M}^{i+1} \\
   \mathfrak{M}^{2} & 1+\mathfrak{M}^{i+1}
   \end{smallmatrix}\right)$;
  \item[ii.] the maps
\begin{align*}
   \mathfrak{M}^{i+1}/\mathfrak{M}^{i+2} \times (1+\mathfrak{M})/(1+\mathfrak{M}^{i+1}) \times  \mathfrak{O}/\mathfrak{M} & \xrightarrow{\; \cong \;} (I^+_i)^{ab}  \\
  (c,t,b) & \longmapsto
     \left(\begin{smallmatrix}
       1 & 0 \\
       c & 1
     \end{smallmatrix}\right)
     \left(\begin{smallmatrix}
       t & 0 \\
       0 &  t^{-1}
     \end{smallmatrix}\right)
     \left(\begin{smallmatrix}
       1 & b \\
       0 & 1
     \end{smallmatrix}\right) =
     \left(\begin{smallmatrix}
       t & tb \\
       tc & tbc +t^{-1}
\end{smallmatrix}\right)
\end{align*}
and
\begin{align*}
   \mathfrak{M}/\mathfrak{M}^2 \times (1+\mathfrak{M})/(1+\mathfrak{M}^{i+1}) \times  \mathfrak{M}^i/\mathfrak{M}^{i+1} & \xrightarrow{\; \cong \;} (I^-_i)^{ab}  \\
  (c,t,b) & \longmapsto
  \left(\begin{smallmatrix}
       1 & 0 \\
       c & 1
     \end{smallmatrix}\right)
     \left(\begin{smallmatrix}
       t & 0 \\
       0 &  t^{-1}
     \end{smallmatrix}\right)
     \left(\begin{smallmatrix}
       1 & b \\
       0 & 1
     \end{smallmatrix}\right) =
     \left(\begin{smallmatrix}
       t & tb \\
       tc & tbc +t^{-1}
\end{smallmatrix}\right)
\end{align*}
are isomorphisms of groups.
\end{itemize}
\end{proposition}
\begin{proof}
 By Remark \ref{remark:N2} we only need to treat the case of $I^-_i$. Let  $J^-_i :=
\left(\begin{smallmatrix}
   1+\mathfrak{M}^{i+1}  & \mathfrak{M}^{i+1} \\
   \mathfrak{M}^2 & 1+\mathfrak{M}^{i+1}
\end{smallmatrix}\right)$.
This is a normal subgroup of $I^-_i$ which has the Iwahori factorization
$J^-_i =
\left(\begin{smallmatrix}
  1 & 0 \\
  \mathfrak{M}^2 & 1
\end{smallmatrix}\right)  T^{i+1}
\left(\begin{smallmatrix}
  1 & \mathfrak{M}^{i+1} \\
  0 & 1
\end{smallmatrix}\right)$. Recall also the Iwahori factorization
$I^-_i =
\left(\begin{smallmatrix}
  1 & 0 \\
  \mathfrak{M} & 1
\end{smallmatrix}\right)  T^1
\left(\begin{smallmatrix}
  1 & \mathfrak{M}^i \\
  0 & 1
\end{smallmatrix}\right)$. The formulas
\begin{align*}
   [\left(\begin{smallmatrix}
  t & 0 \\
  0 & t^{-1}
  \end{smallmatrix}\right),
  \left(\begin{smallmatrix}
  1 & \mathfrak{M}^i \\
  0 & 1
  \end{smallmatrix}\right) ] & =
  \left(\begin{smallmatrix}
  1 & (t^2 - 1)\mathfrak{M}^i \\
  0 & 1
  \end{smallmatrix}\right) =
  \left(\begin{smallmatrix}
  1 & \mathfrak{M}^{i+1} \\
  0 & 1
  \end{smallmatrix}\right) \ , \\
  [\left(\begin{smallmatrix}
  t & 0 \\
  0 & t^{-1}
  \end{smallmatrix}\right),
  \left(\begin{smallmatrix}
  1 & 0 \\
  \mathfrak{M} & 1
  \end{smallmatrix}\right) ] & =
  \left(\begin{smallmatrix}
  1 & 0 \\
  (t^{-2} - 1)\mathfrak{M} & 1
  \end{smallmatrix}\right) =
   \left(\begin{smallmatrix}
  1 & 0 \\
  \mathfrak{M}^2 & 1
  \end{smallmatrix}\right) \ , \\
  [\left(\begin{smallmatrix}
  1 & 0 \\
  \pi c & 1
  \end{smallmatrix}\right),
  \left(\begin{smallmatrix}
  1 & \pi^i b \\
  0 & 1
  \end{smallmatrix}\right) ] & =
  \left(\begin{smallmatrix}
  1 - \pi^{i+1} bc & \pi^{2i+1} b^2 c \\
  - \pi^{i+2} bc^2 & 1 + \pi^{i+1} bc + \pi^{2i+2} b^2c^2
  \end{smallmatrix}\right) \in
  \left(\begin{smallmatrix}
  1 & 0 \\
  \mathfrak{M}^2 & 1
  \end{smallmatrix}\right)
  \left(\begin{smallmatrix}
  1 - \pi^{i+1} bc & 0 \\
  0 & (1 - \pi^{i+1} bc)^{-1}
  \end{smallmatrix}\right)
  \left(\begin{smallmatrix}
  1 & \mathfrak{M}^{2i+1} \\
  0 & 1
  \end{smallmatrix}\right)
\end{align*}
for $t \in 1+\mathfrak{M}$ and $b,c \in \mathfrak{O}$
then show that $J^-_i = [I^-_i,I^-_i]$ and that the map in question is a homomorphism of groups. Its bijectivity more or less is obvious.
\end{proof}

\begin{corollary}\label{Frattini}
For $p \neq 2$ the map
\begin{align*}
  (K_1)_\Phi = K_1/K_2 & \xrightarrow{\; \cong \;} \mathfrak{sl}_2(\mathbb{F}_q) \\
  1 + \pi A & \longmapsto A \bmod \mathfrak{M}
\end{align*}
is an isomorphism of groups, where $\mathfrak{sl}_2(\mathbb{F}_q)$ denotes the additive group of all $2 \times 2$-matrices over $\mathfrak{O}/\mathfrak{M} \cong \mathbb{F}_q$ of trace zero.
\end{corollary}

We suppose until the end of this paragraph that  $\mathfrak F=\mathbb Q_p$. In this case we are able to treat  the case $p=2$.
 We introduce, for any $y \in \mathbb{Q}_p$, the matrices
\begin{equation*}
    u_-(y) := \left(\begin{smallmatrix}
1  & 0 \\
y & 1
\end{smallmatrix}\right)
\qquad\text{and}\qquad
u_+(y) := \left(\begin{smallmatrix}
1  & y \\
0 & 1
\end{smallmatrix}\right) \ .
\end{equation*}
We also need the diagonal matrices
\begin{equation*}
  d_p := \left(\begin{smallmatrix}
1 + p  & 0 \\
0 & (1 + p)^{-1}
\end{smallmatrix}\right)
\quad\text{for any $p$ and}\quad
d_{2,-1} := \left(\begin{smallmatrix}
-1  & 0 \\
0 & -1
\end{smallmatrix}\right), \
d_{2,5} := \left(\begin{smallmatrix}
5  & 0 \\
0 & 1/5
\end{smallmatrix}\right)
\quad\text{for $p = 2$}
\end{equation*}
in $\mathbf{SL_2}(\mathbb{Q}_p)$. For $n \geq 1$ let $T^n \subseteq T^1$ denote the subgroup of all matrices whose diagonal entries are congruent to one modulo $p^n$. Then $(T^1)_\Phi = \mathbb{F}_p \overline{d_p}$, resp.\ $= \mathbb{F}_2 \overline{d_{2,-1}} \oplus \mathbb{F}_2 \overline{d_{2,5}}$, if $p \neq 2$, resp.\ $p=2$, and always $T^1/T^2 = \mathbb{F}_p \overline{d_p}$. Recall that we denote by $\overline{?}$ the coset of $?$ in the respective quotient group.

\begin{proposition}\label{frattini-1}  Suppose that $\mathfrak F=\mathbb Q_p$. Then,

\begin{itemize}
  \item[i.]  $I/\Phi(I) \cong \mathbb{F}_p^2$ is generated by the cosets of $u_-(p)$ and $u_+(1)$.

  \end{itemize}
Now let $i\geq 1$. If $p \neq 2$, or $p = 2$ and $i = 1$, we have:
\begin{itemize}
  \item[ii.] $(I^+_i)_\Phi \cong \mathbb{F}_p^3$ is generated by the cosets of $u_-(p^{i+1})$, $d_p$, and $u_+(1)$;
  \item[iii.] $(I^-_i)_\Phi \cong \mathbb{F}_p^3$ is generated by the cosets of $u_-(p)$, $d_p$, and $u_+(p^i)$.
\end{itemize}
If $p = 2$ and $i \geq2$ then the above Frattini quotients are four dimensional with generators as above replacing $d_p$ by $d_{2,-1}$ and $d_{2,5}$.

\end{proposition}
\begin{proof}
 If $p \neq 2$, then the result can be drawn directly from Prop.\ \ref{abelianization} and the fact that the $p$-adic logarithm induces an isomorphism $1+p\mathbb{Z}_p \cong p\mathbb{Z}_p$. We give here a proof that involves the case $p=2$.

i. (The statement was pointed out to us by R.\ Greenberg. But our argument is different from his.)
By the Iwahori factorization the multiplication map induces a set theoretic bijection
\begin{equation*}
    u_-(p)^{\mathbb{Z}_p} \cdot T^1 \cdot u_+(1)^{\mathbb{Z}_p} = I \ .
\end{equation*}
Obviously $u_-(p)^{p\mathbb{Z}_p}$ and $u_+(1)^{p\mathbb{Z}_p}$ are contained in $\Phi(I)$. Moreover, for any $y \in \mathbb{Z}_p$ we have the identity
\begin{equation*}
    \left(\begin{smallmatrix}
1+py  & 0 \\
0 & (1+py)^{-1}
\end{smallmatrix}\right)
= [u_+(1),u_-(py)] \cdot \big(u_-(py) u_+(1)^{p\frac{y}{1+py}} u_-(py)^{-1} \big) \cdot u_-(p)^{-py^2}
\end{equation*}
The first factor is a commutator. The second, resp.\ third, factor is conjugate to an element visibly in $\Phi(I)$, resp.\ visibly lies in $\Phi(I)$. Hence $T^1$ is contained in $\Phi(I)$. We conclude that the normal subgroup $I \cap K_1 = u_-(p)^{p\mathbb{Z}_p}  T^1  u_+(1)^{p\mathbb{Z}_p}$, which is of index $p^2$ in $I$, is contained in $\Phi(I)$. It follows that the $\mathbb{F}_p$-vector space $I/\Phi(I)$ has dimension $\leq 2$ and is generated by the asserted elements. But its dimension cannot be $\leq 1$ since the pro-$p$ group $I$ is not commutative.

ii. and iii. By Remark \ref{remark:N2} we only need to give the argument for $I^+_i$. Being the intersection of two Iwahori subgroups belonging to the same apartment, the group $I^+_i$ has the Iwahori factorization
\begin{equation*}
  I^+_i = u_-(p^{i+1})^{\mathbb{Z}_p} T^1 u_+(1)^{\mathbb{Z}_p} \ .
\end{equation*}
Obviously $u_-(p^{i+1})^{p\mathbb{Z}_p}$, $(T^1)^p$, and $u_+(1)^{p\mathbb{Z}_p}$ are contained in $\Phi(I^+_i)$. We have $(T^1)^p = T^2$ for $p \neq 2$ whereas $(T^1)^2 = T^3$ for $p = 2$. The commutator formula
\begin{align}\label{f:commutator}
  [u_-(p^{i+1}), u_+(1)] & = \left(\begin{smallmatrix}
         1 - p^{i+1} & p^{i+1} \\
         -p^{2i+2} & 1 + p^{i+1} + p^{2i+2}
      \end{smallmatrix}\right) \\
                  & = \left(\begin{smallmatrix}
         1 & 0 \\
         -\frac{p^{2i+2}}{1-p^{i+1}} & 1
      \end{smallmatrix}\right)
      \left(\begin{smallmatrix}
         1 - p^{i+1} & 0 \\
         0 & (1 - p^{i+1})^{-1}
      \end{smallmatrix}\right)
      \left(\begin{smallmatrix}
         1 & \frac{p^{i+1}}{1-p^{i+1}} \\
         0 & 1
      \end{smallmatrix}\right)    \nonumber
\end{align}
shows that $T^2 \subseteq \Phi(I^+_i)$ even if $p = 2$ and $i = 1$. Suppose first that $p \neq 2$ or $i = 1$. Then the subgroup $I^+_{i+1} \cap K_2 u_+(p)^{\mathbb{Z}_p} = u_-(p^{i+2})^{\mathbb{Z}_p} T^2 u_+(p)^{\mathbb{Z}_p}$ is contained in $\Phi(I^+_i)$. We claim that $I^+_{i+1} \cap K_2 u_+(p)^{\mathbb{Z}_p}$ is normal in $I^+_i$. We know from Remark \ref{normal}.i that $I^+_i$ normalizes $I^+_{i+1}$. We will show that, in fact, $I^+_i$ also normalizes $K_2 u_+(p)^{\mathbb{Z}_p}$. Obviously $I^+_i$ normalizes $K_2$, and $T^1 u_+(1)^{\mathbb{Z}_p}$ normalizes $u_+(p)^{\mathbb{Z}_p}$. By the above Iwahori factorization it therefore remains to compute
\begin{equation*}
  u_-(p^{i+1}) u_+(p)  u_-(-p^{i+1}) =
    \left(\begin{smallmatrix}
        1 - p^{i+2}  & p^{i+3} \\
        -p^{2i+3} & 1 + p^{i+2} + p^{2i+4}
\end{smallmatrix}\right) u_+(p) \in K_2 u_+(p) \ .
\end{equation*}
Next we claim that the factor group $I^+_i/I^+_{i+1} \cap K_2 u_+(p)^{\mathbb{Z}_p}$ is abelian. This follows from the Iwahori factorization of $I^+_i$ and  the formulas \eqref{f:commutator} and
\begin{align*}
   & [\left(\begin{smallmatrix}
         t & 0 \\
         0 & t^{-1}
      \end{smallmatrix}\right), u_+(1)^{\mathbb{Z}_p}] = u_+(t^2 - 1)^{\mathbb{Z}_p} \ , \\
   & [\left(\begin{smallmatrix}
         t & 0 \\
         0 & t^{-1}
      \end{smallmatrix}\right), u_-(p^{i+1})^{\mathbb{Z}_p}] = u_-((t^{-2} - 1)p^{i+1})^{\mathbb{Z}_p} \ ,
\end{align*}
which show that the elements $u_-(p^{i+1})$, $d_p$, and $u_+(1)$ of $I^+_i$ commute modulo $I^+_{i+1} \cap K_2 u_+(p)^{\mathbb{Z}_p}$. Since all three have order $p$ modulo $I^+_{i+1} \cap K_2 u_+(p)^{\mathbb{Z}_p}$ it follows that $I^+_{i+1} \cap K_2 u_+(p)^{\mathbb{Z}_p} = \Phi(I^+_i)$ and that $(I^+_i)_\Phi$ is generated as asserted. By considering the filtration
\begin{equation*}
  I^+_i \supseteq I^+_i \cap K_1 \supseteq I^+_i \cap K_2 u_+(p)^{\mathbb{Z}_p} \supseteq I^+_{i+1} \cap K_2 u_+(p)^{\mathbb{Z}_p}
\end{equation*}
one easily checks that the Frattini quotient $(I^+_i)_\Phi$ has order $p^3$. In the case $p=2$ and $i \geq2$ the arguments are the same except that one has to work with $K_3$ instead of $K_2$.
\end{proof}

\subsubsection{Transfer maps when $\mathfrak F\neq \mathbb  Q_p$.}

\begin{proposition}\label{Transfer}
Let $p \neq 2$ and $\mathfrak{F} \neq \mathbb{Q}_p$; if either $i \geq 1$ or $i = 0$ and $q \neq 3$ then the transfer maps $(I^+_i)_\Phi \longrightarrow (I^+_{i+1})_\Phi$ and $(I^-_i)_\Phi \longrightarrow (I^-_{i+1})_\Phi$ are the zero maps.
\end{proposition}
\begin{proof}
Again we only need to treat the case of $I^-_i$. By the Iwahori factorization of $I^-_i$ it suffices to compute the transfer of elements of the form
$\left(\begin{smallmatrix}
       1 & 0 \\
       \pi v & 1
       \end{smallmatrix}\right)
       $,
$\left(\begin{smallmatrix}
       t & 0 \\
       0 & t^{-1}
       \end{smallmatrix}\right)$,
and
$\left(\begin{smallmatrix}
       1 & \pi^i u \\
       0 & 1
       \end{smallmatrix}\right)
$ in $I^-_i$. Let $S \subseteq \mathfrak{O}$  be a set of representatives for the cosets in $\mathfrak{O}/\mathfrak{M}$. Then the matrices
$\left(\begin{smallmatrix}
  1 & \pi^i b \\
  0 & 1
\end{smallmatrix}\right)$, for $b \in S$, form a set of representatives for the cosets in $I^-_{i+1}\backslash I^-_i$. Since
$\left(\begin{smallmatrix}
       1 & 0 \\
       \pi v & 1
       \end{smallmatrix}\right)
        \in I^-_{i+1}$, which is normal in $I^-_i$, we have (see \eqref{rappel:transfer})
\begin{equation*}
  \mathrm{tr}(
\left(\begin{smallmatrix}
1 & 0 \\
\pi v & 1
\end{smallmatrix}\right) ) \equiv
\prod_{b \in S}
\left(\begin{smallmatrix}
1 & \pi^i b \\
0 & 1
\end{smallmatrix}\right)
\left(\begin{smallmatrix}
1 & 0 \\
\pi v & 1
\end{smallmatrix}\right)
\left(\begin{smallmatrix}
1 & -\pi^i b \\
0 & 1
\end{smallmatrix}\right)  \equiv \prod_{b \in S}
\left(\begin{smallmatrix}
1+b\pi^{i+1} v & -b^2 \pi^{2i+1} v \\
\pi v & 1-b\pi^{i+1} v
\end{smallmatrix}\right) \mod [I^-_{i+1},I^-_{i+1}]
\end{equation*}
which, under the isomorphism $(I^-_{i+1})^{ab} \cong \mathfrak{M}/\mathfrak{M}^2 \times (1+\mathfrak{M}/1+\mathfrak{M}^{i+2}) \times  \mathfrak{M}^{i+1}/\mathfrak{M}^{i+2} $ of Prop.\ \ref{abelianization}.ii, corresponds to the element
\begin{align*}
  & (q\pi v \bmod \mathfrak M^2,\: \prod_b (1+ b \pi^{i+1} v) \bmod 1 + \mathfrak{M}^{i+2}, \:- \pi^{2i+1}v \sum_b  b^2 \bmod \mathfrak{M}^{i+2})  \\
  & = (0,\: 1+ \pi^{i+1} v \sum_b b \bmod 1 + \mathfrak{M}^{i+2}, \:- \pi^{2i+1}v \sum_b  b^2 \bmod \mathfrak{M}^{i+2}) \ .
\end{align*}
View $b \longmapsto b$ and $b \longmapsto b^2$ as  $\mathbb{F}_q$-valued characters of the group $\mathbb{F}_q^\times$ of order prime to $p$. By the orthogonality relation for characters the sum $\sum_{b \in \mathbb{F}_q^\times} b$, resp.\  $\sum_{b \in \mathbb{F}_q^\times} b^2$,  vanishes if and only if the respective character is nontrivial if and only if $q \neq 2$, resp.\ $q \neq 2,3$. Since we assume $p \neq 2$ the second component is zero whereas the last component is zero if either $i \geq 1$ or $i = 0$ and $q \neq 3$.

For $t\in 1+\mathfrak{M}$, the element
$\left(\begin{smallmatrix}
       t & 0 \\
       0 & t^{-1}
\end{smallmatrix}\right)$ again lies in $I^-_i$ so that we have
\begin{equation*}
  \mathrm{tr}( \left(\begin{smallmatrix}
t & 0 \\
0 & t^{-1}
\end{smallmatrix}\right) ) \equiv
\prod_{b \in S}
\left(\begin{smallmatrix}
1 & \pi^ib \\
0 & 1
\end{smallmatrix}\right)
\left(\begin{smallmatrix}
t & 0 \\
0 & t^{-1}
\end{smallmatrix}\right)
\left(\begin{smallmatrix}
1 & -\pi^ib \\
0 & 1
\end{smallmatrix}\right)  \equiv \prod_{b \in S}
\left(\begin{smallmatrix}
t & \pi^ib(t^{-1}-t) \\
0 & t^{-1}
\end{smallmatrix}\right) \mod [I^-_{i+1},I^-_{i+1}]
\end{equation*}
which corresponds to
\begin{equation*}
  (0,\: t^q \bmod 1 + \mathfrak{M}^{i+2} , \:\pi^i(t^{-2} - 1) \sum_b  b \bmod \mathfrak{M}^{i+2}) \ .
\end{equation*}
The second component is a $p$th power and therefore maps to $\Phi(I^-_{i+1})$. The last component is zero since $q \neq 2$.

For
$\mathrm{tr}( \left(\begin{smallmatrix}
1 & \pi^i u \\
0 & 1
\end{smallmatrix}\right) )$ it suffices, by the same argument as the one in the proof of Lemma \ref{quadratic}.i.a), to compute the image of $\pi^i u$ under the transfer map $\mathrm{tr} : \mathfrak{M}^i \longrightarrow \mathfrak{M}^{i+1}$, which has image $q\mathfrak{M}^i$. But, by our assumption that $\mathfrak{F} \neq \mathbb{Q}_p$, we have $q\mathfrak{M}^i \subseteq \mathfrak{M}^{i+2}$. Therefore
$\mathrm{tr}(
\left(\begin{smallmatrix}
1 & \pi^i u \\
0 & 1
\end{smallmatrix}\right) ) \equiv 0 \mod [I^-_{i+1},I^-_{i+1}]$ in view of Prop.\ \ref{abelianization}.i.
\end{proof}

\begin{remark}\label{Transfer-q-small}
For $\mathfrak{F} \neq \mathbb{Q}_p$ we have:
\begin{itemize}
\item[i.]  If $q=3$, then the
the composite maps
\begin{equation*}
    H^1(I^\pm_1,k) \xrightarrow{\mathrm{cores}} H^1(I,k) \xrightarrow{\;\mathrm{res}\;}   H^1(I^\mp_1,k)
\end{equation*}
are trivial;
\item[ii.] if $q=2$ or $3$, then the composite map
\begin{equation*}
  H^1(I^\pm_1/I^\pm_2,k) \xrightarrow{\;\mathrm{inf}\;} H^1(I^\pm_1,k) \xrightarrow{\mathrm{cores}} H^1(I,k)
\end{equation*} is injective;
\item[iii.] if $p=2$ then the composite maps
\begin{equation*}
  H^1(I^\pm_1/I^\pm_2,k) \xrightarrow{\;\mathrm{inf}\;}H^1(I^\pm_1,k) \xrightarrow{\mathrm{cores}} H^1(I,k) \xrightarrow{\;\mathrm{res}\;}   H^1(I^\mp_1,k)
\end{equation*}
and
\begin{equation*}
  H^1(I^\pm_{i+1}/I^\pm_{i+2},k) \xrightarrow{\;\mathrm{inf}\;}H^1(I^\pm_{i+1},k) \xrightarrow{\mathrm{cores}} H^1(I^\pm_i,k) \qquad\text{for $i\geq 1$}
\end{equation*}
are trivial.
\end{itemize}
\end{remark}
\begin{proof}
We check some of these statements in their dual version for Frattini quotients and invoke Remark \ref{remark:N2} for the remaining ones.

i. From the proof of Prop.\ \ref{Transfer}, we see that when $q=3$ the composite map $$(I_1^+)_\Phi\rightarrow I_\Phi \xrightarrow{\;\mathrm{tr}\;} (I_1^-)_\Phi$$ is trivial, where  $(I_1^+)_\Phi\rightarrow I_\Phi$ is the natural map induced by the inclusion.

ii.  The reason for this result is the fact that, for $q = 2,3$, the transfer map followed by the projection
\begin{equation*}
  I_\Phi \xrightarrow{\;\mathrm{tr}\;} (I_1^-)_\Phi \xrightarrow{\;\pr\;} I_1^-/I^-_2
\end{equation*}
is surjective. To see this we will compute the transfer of the elements
$\left(\begin{smallmatrix}
1 & 0 \\
c & 1
\end{smallmatrix}\right) \in I$ with $c \in \mathfrak{M}$. It is convenient to use the identification
\begin{align*}
  I^-_1/I^-_2 & \xrightarrow{\;\cong\;} \mathfrak{M}/\mathfrak{M}^2 \\
  \left(\begin{smallmatrix}
  a & b \\
  c & d
  \end{smallmatrix}\right) & \longmapsto b \ .
\end{align*}
In the proof of Prop.\ \ref{Transfer} we have seen that the image of $\left(\begin{smallmatrix}
1 & 0 \\
c & 1
\end{smallmatrix}\right)$ under this composite map is $-c(\sum_{b \in \mathfrak{O}/\mathfrak{M}} b^2) \in \mathfrak{M}/\mathfrak{M}^2$ and that $\sum_{b \in \mathfrak{O}/\mathfrak{M}} b^2 \neq 0$  if and only if $q = 2$ or $3$.

iii. Suppose that $p=2$ and  consider for $i\geq 0$
the composite map
\begin{equation*}
  (I^-_i)_\Phi \xrightarrow{\;\mathrm{tr}\;} (I^-_{i+1})_\Phi \xrightarrow{\;\pr\;} I^-_{i+1}/I^-_{i+2} \ .
\end{equation*}
Using the identification
\begin{align*}
  I^-_{i+1}/I^-_{i+2} & \xrightarrow{\;\cong\;} \mathfrak{M}^{i+1}/\mathfrak{M}^{i+2} \\
  \left(\begin{smallmatrix}
  a & b \\
  c & d
  \end{smallmatrix}\right) & \longmapsto b
\end{align*}
we deduce from the proof of Prop.\ \ref{Transfer} that:
\begin{itemize}
  \item[--] For $t\in 1+\mathfrak M$,  the image of
$\left(\begin{smallmatrix}
t & 0 \\
0 & t^{-1}
\end{smallmatrix}\right)$
is $\pi^i(t^{-2} - 1) \sum_{b\in\mathfrak O/\mathfrak M}  b \equiv 0 \bmod \mathfrak M^{i+2}$ since $t^{-2} -1 \equiv 2(t^{-1} - 1) \equiv 0 \bmod \mathfrak{M}^2$.
  \item[--] The image of
 $\left(\begin{smallmatrix}
1 &  \mathfrak{ M }^ i\\
0 & 1
\end{smallmatrix}\right)$ is trivial since $\mathfrak F\neq \mathbb Q_p$.
  \item[--] For $v\in \mathfrak O$, the image of
$\left(\begin{smallmatrix}
1 & 0 \\
\pi v & 1
\end{smallmatrix}\right)$
is $- \pi^{2i+1}v \sum_{b\in\mathfrak O/\mathfrak M}  b^2 \bmod \mathfrak M^{i+2}$, which is zero if either $i \geq 1$ or $i=0$ and $v \in \mathfrak{M}$.
\end{itemize}
Therefore, the composite map above is trivial if $i\geq 1$. If $i=0$, then its restriction to the image of $(I^+_1)_\Phi \rightarrow I_\Phi$ is trivial.
\end{proof}

\begin{remark}\label{conjugation}
Let $p \neq 2$; under the identifications $I_\Phi \cong \mathfrak{M}/\mathfrak{M}^2 \times \mathfrak{O}/\mathfrak{M}$ and $(K_1)_\Phi \cong \mathfrak{sl}_2(\mathbb{F}_q)$ from Prop.\ \ref{abelianization} and Cor.\ \ref{Frattini} we have:
\begin{itemize}
  \item[i.] The canonical map $(K_1)_\Phi \longrightarrow I_\Phi$ is given by
  $\left(\begin{smallmatrix}
  \bar{a} & \bar{b} \\
  \bar{c} & \bar{d}
  \end{smallmatrix}\right) \longmapsto (\pi\bar{c},0)$;
  \item[ii.] the conjugation action of $I/K_1 \cong \mathfrak{O}/\mathfrak{M}$ on $(K_1)_\Phi$ is given by
     \begin{equation*}
       (\bar{u},
       \left(\begin{smallmatrix}
       \bar{a} & \bar{b} \\
       \bar{c} & \bar{d}
       \end{smallmatrix}\right) ) \longmapsto
       \left(\begin{smallmatrix}
       \bar{a} & \bar{b} \\
       \bar{c} & \bar{d}
       \end{smallmatrix}\right) + \bar{u}
       \left(\begin{smallmatrix}
       \bar{c} & 2\bar{d} \\
       0 & -\bar{c}
       \end{smallmatrix}\right) + \bar{u}^2
       \left(\begin{smallmatrix}
       0 & -\bar{c} \\
       0 & 0
       \end{smallmatrix}\right) \ .
     \end{equation*}
\end{itemize}
\end{remark}
\begin{proof}
i. is obvious, and ii. is a simple computation.
\end{proof}


\subsubsection{Transfer maps for  $\mathfrak{F}= \mathbb Q_p$}

In this subsection we always \textbf{assume} that $\mathfrak{F} = \mathbb{Q}_p$.

\begin{lemma}\phantomsection\label{transfer}
\begin{itemize}
  \item[i.] $\mathrm{tr} : I_\Phi = (I^+_0)_\Phi \longrightarrow (I^+_1)_\Phi$ satisfies
   \begin{align*}
     \overline{u_-(p)} & \longmapsto \overline{u_-(p^2)} \, , \\
     \overline{u_+(1)} & \longmapsto
                 \begin{cases}
                 0 & \text{if $p \neq 2,3$}, \\
                 \overline{u_-(p^2)}  & \text{if $p = 3$}, \\
                 \overline{u_-(p^2)} + \overline{d_2} & \text{if $p = 2$}.
                 \end{cases}
   \end{align*}
  \item[ii.] $\mathrm{tr} : I_\Phi = (I^-_0)_\Phi \longrightarrow (I^-_1)_\Phi = (K_1)_\Phi$ satisfies
   \begin{align*}
     \overline{u_-(p)} & \longmapsto
                  \begin{cases}
                 0 & \text{if $p \neq 2,3$}, \\
                 \overline{u_+(p)} & \text{if $p = 3$}, \\
                 \overline{u_+(p)} + \overline{d_2} & \text{if $p = 2$},
                 \end{cases}  \\
     \overline{u_+(1)} & \longmapsto \overline{u_+(p)} \, .
   \end{align*}
  \item[iii.] $\mathrm{tr} : (I^+_i)_\Phi \longrightarrow (I^+_{i+1})_\Phi$, for $i \geq 1$, satisfies
   \begin{align*}
      \overline{u_-(p^{i+1})} & \longmapsto \overline{u_-(p^{i+2})} \, , \\
      \overline{u_+(1)} & \longmapsto 0 \, , \\
      \overline{d_p} & \longmapsto 0\, ,\ \text{resp.}\ \overline{d_{2,-1}} \mapsto 0 \, , \overline{d_{2,5}} \mapsto 0 \ \text{if $p=2$}.
   \end{align*}
  \item[iv.] $\mathrm{tr} : (I^-_i)_\Phi \longrightarrow (I^-_{i+1})_\Phi$, for $i \geq 1$, satisfies
   \begin{align*}
     \overline{u_-(p)} & \longmapsto 0 \, \ \\
     \overline{u_+(p^i)} & \longmapsto \overline{u_+(p^{i+1})} \, , \\
     \overline{d_p} & \longmapsto 0\, ,\ \text{resp.}\ \overline{d_{2,-1}} \mapsto 0 \, , \overline{d_{2,5}} \mapsto 0 \ \text{if $p=2$}.
  \end{align*}
\end{itemize}
\end{lemma}
\begin{proof}
These are straightforward computations. We only give the details for iii. and leave the other cases to the reader. In fact, already
in the proof of Lemma \ref{quadratic}.i.a) we have seen that $\mathrm{tr}(\overline{u_-(p^{i+1})}) = \overline{u_-(p^{i+1})^p} = \overline{u_-(p^{i+2})}$. For the other generators we start from the actual definition of the transfer map where we use the elements $u_-(jp^{i+1})$ for $j = 0, \ldots, p-1$ as coset representatives for $I^+_{i+1}$ in $I^+_i$. We recall (see \eqref{rappel:transfer}): For a given $g \in I^+_i$ and any $0 \leq j \leq p-1$ let $0 \leq j_g \leq p-1$ be the unique integer and $g(j) \in I^+_{i+1}$ be the unique element such that
\begin{equation*}
  u_-(jp^{i+1})g = g(j) u_-(j_g p^{i+1}) \ ;
\end{equation*}
then
\begin{equation*}
  \mathrm{tr}(g) = \prod_{j=0}^{p-1} g(j) \mod [I^+_{i+1},I^+_{i+1}] \ .
\end{equation*}
First let $d = \left(\begin{smallmatrix}
         1+pa & 0 \\
         0 & (1+pa)^{-1}
      \end{smallmatrix}\right)$
be a diagonal matrix. We then have
\begin{equation*}
  u_-(jp^{i+1}) d = d u_-((1+pa)^2 j p^{i+1}) = d u_-(2ajp^{i+2}) u_-(a^2jp^{i+3}) u_-(jp^{i+1})
\end{equation*}
and hence $j_d = j$ and $d(j) = d u_-(2ajp^{i+2}) u_-(a^2jp^{i+3}) \in d u_-(2ajp^{i+2}) \Phi(I^+_{i+1})$. It follows that
\begin{equation*}
  \mathrm{tr}(\overline{d_p}) = \sum_{j=0}^{p-1} \big(\overline{d_p} + 2j \overline{u_-(p^{i+2})} \big) = p \big( \overline{d_p} + 2\tfrac{(p-1)}{2} \overline{u_-(p^{i+2})} \big) = 0 \ ,
\end{equation*}
as well as $ \mathrm{tr}(\overline{d_{2,-1}}) =  \mathrm{tr}(\overline{d_{2,5}}) = 0$ if $p=2$.

Finally, for $u_+(1)$ we have
\begin{equation*}
  u_-(jp^{i+1}) u_+(1) =
          \left(\begin{smallmatrix}
         1 - jp^{i+1} & 1 \\
         -j^2 p^{2i+2} & 1 + jp^{i+1}
      \end{smallmatrix}\right)  u_-(jp^{i+1})
\end{equation*}
and hence $j_{u_+(1)} = j$ and
\begin{align*}
  u_+(1)(j) & =
     \left(\begin{smallmatrix}
         1 - jp^{i+1} & 1 \\
         -j^2 p^{2i+2} & 1 + jp^{i+1}
      \end{smallmatrix}\right)
    =  \left(\begin{smallmatrix}
         1 & 0 \\
         - \frac{j^2 p^{2i+2}}{1-jp^{i+1}} & 1
      \end{smallmatrix}\right)
      \left(\begin{smallmatrix}
         1 - jp^{i+1} & 0 \\
         0 & \frac{1}{1 - jp^{i+1}}
      \end{smallmatrix}\right)
      \left(\begin{smallmatrix}
         1 & \frac{1}{1-jp^{i+1}} \\
         0 & 1
      \end{smallmatrix}\right) \\
    & \in u_+(1) \Phi(I^+_{i+1}) \ .
\end{align*}
It follows that $\mathrm{tr}(\overline{u_+(1)}) = p \overline{u_+(1)} = 0$.

We point out that in case $i = 0$, which is part of the statement in i., we get
\begin{align*}
  u_+(1)(j) & =
  \left(\begin{smallmatrix}
         1 & 0 \\
         - \frac{j^2 p^2}{1-jp} & 1
      \end{smallmatrix}\right)
      \left(\begin{smallmatrix}
         1 - jp & 0 \\
         0 & \frac{1}{1 - jp}
      \end{smallmatrix}\right)
      \left(\begin{smallmatrix}
         1 & \frac{1}{1-jp} \\
         0 & 1
      \end{smallmatrix}\right) \\
    & \in u_-(-j^2p^2)
        \left(\begin{smallmatrix}
         1 - jp & 0 \\
         0 & \frac{1}{1 - jp}
      \end{smallmatrix}\right)
          u_+(1) \Phi(I^+_1) \ .
\end{align*}
We deduce that
\begin{equation*}
  \mathrm{tr}(\overline{u_+(1)}) = - (\sum_{j=1}^{p-1} j^2) \overline{u_-(p^2)} + \sum_{j=0}^{p-1} \overline{
                                          \left(\begin{smallmatrix}
                                                      1 - jp & 0 \\
                                                           0 & \frac{1}{1 - jp}
                                          \end{smallmatrix}\right)}
                                          + p \overline{u_+(1)} \in (I^+_1)_\Phi \ .
\end{equation*}
The last summand is zero. The middle sum is over all elements in $T^1/T^2 \cong \mathbb{F}_p$ and hence is zero, resp.\ $\overline{d_2}$, for $p \neq 2$, resp.\ $p = 2$. The coefficient of the first summand is equal to $-\tfrac{1}{6}(2(p-1)+1)p(p-1)$ which is an integer divisible by $p$ if $p \neq 2,3$, is equal to $-5$ if $p = 3$, and is equal to $-1$ if $p = 2$.
\end{proof}

\begin{corollary}\label{images-directsum}
For $i \geq 1$ we have
\begin{equation*}
  \im \big(H^1(I^\pm_{i-1},k) \xrightarrow{\mathrm{res}} H^1(I^\pm_i,k)\big) \cap \im \big(H^1(I^\pm_{i+1},k) \xrightarrow{\mathrm{cores}} H^1(I^\pm_i,k)\big) = 0 \ .
\end{equation*}
Moreover,
\begin{equation*}
  \ker\big(  H^1(I^\pm_{1},k) \xrightarrow{\mathrm{cores}} H^1(I,k) \xrightarrow{\mathrm{res}} H^1(I^\mp_1,k)\big) \cap  \ker \big(H^1(I^\pm_1,k) \xrightarrow{\mathrm{res}} H^1(I^\pm_2,k)\big) = 0 \ .
\end{equation*}
\end{corollary}
\begin{proof}
Using Prop.\ \ref{frattini-1} and  Lemma \ref{transfer} it is easy to see that the canonical map $(I^\pm_i)_\Phi \longrightarrow (I^\pm_{i-1})_\Phi$ has kernel containing $\mathbb{F}_p \overline{u_-(p^{i+1})}$ and $\mathbb{F}_p \overline{u_+(p^i)}$, respectively. Lemma \ref{transfer} then implies that
\begin{equation*}
  (I^\pm_i)_\Phi = \ker \big((I^\pm_i)_\Phi \longrightarrow (I^\pm_{i-1})_\Phi \big) + \ker \big((I^\pm_i)_\Phi \xrightarrow{\;\mathrm{tr}\;} (I^\pm_{i+1})_\Phi \big) \ .
\end{equation*}
Dualizing this equation gives the first identity. The second identity, again by duality, is equivalent to
\begin{equation*}
  (I^\pm_1)_\Phi = \im \big( (I^\mp_1)_\Phi \xrightarrow{} I_\Phi \xrightarrow{\mathrm{tr}} (I^\pm_1)_\Phi \big) \oplus \im \big( (I^\pm_2)_\Phi \xrightarrow{} (I^\pm_1)_\Phi \big) \ ,
\end{equation*}
which can be checked using Prop.\  \ref{frattini-1} and Lemma \ref{transfer}.
\end{proof}

\noindent Rachel Ollivier\\
Mathematics Department\\
The University of British Columbia\\
1984 Mathematics Road\\
Vancouver, BC V6T 1Z2, Canada\\
ollivier@math.ubc.ca\\
http://www.math.ubc.ca/$\sim$ollivier\\

\noindent Peter Schneider\\
Mathematisches Institut\\
Westf\"alische Wilhelms-Universit\"at M\"unster\\
Einsteinstr.\ 62\\
D-48149 M\"unster, Germany\\
pschnei@uni-muenster.de\\
http://www.uni-muenster.de/math/u/schneider\\

\end{document}